\documentclass[10pt,a4paper,reqno]{amsart}

\usepackage[utf8]{inputenc}
\usepackage[T1]{fontenc}
\usepackage[english]{babel}
\usepackage{lmodern}
\usepackage{microtype}
\usepackage[a4paper,margin=2cm]{geometry}
\usepackage{amssymb,mathtools,mathrsfs}
\usepackage{bm}
\usepackage{enumitem}
\usepackage{needspace}
\usepackage{array}
\usepackage{tikz}
\usepackage{xcolor}
\usepackage{cite}
\usepackage[hypertexnames=false]{hyperref}
\hypersetup{
 colorlinks=true,
 linkcolor=blue!55!black,
 citecolor=blue!55!black,
 urlcolor=blue!65!black,
 pdftitle={Bulk--surface spectral homogenization with indefinite mass in periodic perforated domains},
 pdfauthor={Hermann DOUANLA},
 pdfsubject={Periodic homogenization of an indefinite bulk--surface spectral problem},
 pdfkeywords={periodic homogenization, indefinite spectral problem, perforated domain, Steklov boundary condition, sign-changing density, spectral convergence}
}

\allowdisplaybreaks
\numberwithin{equation}{section}

\newtheorem{theorem}{Theorem}[section]
\newtheorem{proposition}[theorem]{Proposition}
\newtheorem{lemma}[theorem]{Lemma}
\newtheorem{corollary}[theorem]{Corollary}
\theoremstyle{definition}
\newtheorem{definition}[theorem]{Definition}

\theoremstyle{remark}
\newtheorem{remark}[theorem]{Remark}

\newcommand{\eps}{\varepsilon}
\newcommand{\Om}{\Omega}
\newcommand{\OmE}{\Omega_\eps}
\newcommand{\Gam}{\Gamma}
\newcommand{\GamE}{\Gamma_\eps}
\newcommand{\Dep}{D_\eps}
\newcommand{\Ys}{Y^{\ast}}
\newcommand{\R}{\mathbb{R}}
\newcommand{\aep}{a_\eps}
\newcommand{\be}{b_\eps}
\newcommand{\ce}{c_\eps}
\newcommand{\Ah}{A^{\mathrm{hom}}}
\newcommand{\Peps}{P_\eps}
\newcommand{\tA}{\widetilde A}
\newcommand{\tAh}{\widetilde A^{\mathrm{hom}}}
\newcommand{\tq}{\widetilde q}
\newcommand{\trho}{\widetilde\rho}
\newcommand{\tmeff}{\widetilde m_{\mathrm{eff}}}
\newcommand{\mminus}{m_{-}}
\newcommand{\thetaone}{\theta_1^-}
\newcommand{\lambdaone}{\lambda_1^-}
\newcommand{\norm}[1]{\left\lVert #1\right\rVert}
\newcommand{\abs}[1]{\left\lvert #1\right\rvert}
\newcommand{\dist}{\operatorname{dist}}
\newcommand{\Div}{\operatorname{div}}
\newcommand{\Ran}{\operatorname{Ran}}
\newcommand{\ind}{\mathbf 1}
\newcommand{\Keps}{\mathcal K_\eps}

\title[Bulk--surface spectral homogenization]{Bulk--surface spectral homogenization with indefinite mass in periodic perforated domains}
\author{Hermann DOUANLA}
\dedicatory{\upshape Department of Mathematics, Faculty of Science, University of Yaound\'e I,\\ P.O. Box 812, Yaound\'e, Cameroon\\[2pt] E-mail address: \texttt{hermann.douanla@facsciences-uy1.cm}}
\subjclass[2020]{35B27, 35P15, 35J25, 35B40, 47A75}
\keywords{Periodic homogenization; indefinite spectral problem; perforated domain; Steklov boundary condition; sign-changing density; spectral convergence}
\date{}

\begin{document}

\begin{abstract}
We study the homogenization of an elliptic eigenvalue problem in a periodically perforated domain, where the spectral parameter
appears both in the equation, through a bulk density $q$, and in a  Steklov condition on the boundaries of the holes, through a surface
density $\rho$, and where the resulting bulk--surface spectral measure $q\,dy+\rho\,dS_y$ is 
indefinite. We show that the asymptotic behavior of the spectrum depends essentially on whether the combined mean
$m_{\rm eff}=\int_{Y^*}q\,dy+\int_\Gamma\rho\,dS_y$ is positive, negative, or zero. In all three cases, we prove convergence of eigenvalue clusters with multiplicity and of eigenspaces and, under additional regularity,
$O(\varepsilon)$ estimates for the shifted or rescaled eigenvalues, spectral
projectors, and aligned eigenspaces.
\end{abstract}
\maketitle

\section{Introduction}\label{sec:introduction}

\subsection{Motivation and problem setting}\label{subsec:intro-setting}

Eigenvalue problems in which the spectral parameter multiplies a weight of
variable sign arise from the linearization of reaction--diffusion models. A
motivating prototype is the thermal stability of porous catalysts
\cite{Aris1975,FrankKamenetskii1969}. Heat is released or absorbed both by
reactions in the solid matrix and by reactions on the pore walls.
Linearization of the heat balance about the ambient state suggests a spectral
problem of the form
\[
 -\Div(A\nabla u)=\lambda\,q\,u\quad\text{in the solid},\qquad
 A\nabla u\cdot n=\lambda\,\rho\,u\quad\text{on the pore walls},
\]
where $\lambda$ measures the strength of the reactions and $q$, $\rho$ are
the normalized linearized heat-release rates in the solid and on the walls.
They are positive where the reaction is exothermic and negative where it is
endothermic or where heat is removed. In linearized thermal-stability models, the first positive eigenvalue
may serve as an instability threshold. Indefinite weights
also arise in spatial ecology, where the growth rate is positive in favorable
and negative in unfavorable habitats \cite{CantrellCosner2003,HessKato1980}.
Spectral conditions on
interfaces also describe sloshing and the vibration of fluids in periodic
structures \cite{ConcaPlanchardVanninathan1995}. These examples fit into a common
\emph{bulk--surface} spectral framework, in which one of the two components
may be absent; in the catalyst, both are present and may change sign.

When the medium is periodically perforated at a small scale $\eps$, the
fine structure makes a direct analysis impractical and raises the question
of the effective behavior of the spectrum as $\eps\to0$. For positive
weights this is a classical homogenization problem. For indefinite weights
the answer is more subtle: the spectrum has two branches, and it is not
clear a priori at which scale each branch lives, nor which cell mechanism
governs it. Within this idealized model, this is the question of how a fine-scale arrangement of
exothermic and endothermic sites, in the solid and on the pore walls, affects
the runaway threshold. Our results give a
precise answer in terms of the combined mean $m_{\rm eff}$ defined in
\eqref{eq:intro-meff} below. In particular, the bulk and surface means may cancel each other and thereby change the spectral regime, an effect that cannot be
seen when only one of the two densities is present. The purpose of this paper is to study the asymptotic spectral behavior when the
bulk and the surface mechanisms act simultaneously, with weights of
arbitrary sign, in a periodic geometry where the perforations are allowed to meet the
outer boundary.

Throughout the paper, $\Om\subset\R^N$, $N\ge2$, is a bounded Lipschitz
domain. The periodicity cell is $Y=(0,1)^N$, and $T\Subset Y$ is a Lipschitz
domain with boundary $\Gam=\partial T$. We write $\Ys=Y\setminus\overline T$
and assume that $\Ys$ is connected. The unit normal $n$ on $\Gam$ is the
exterior normal to $\Ys$, hence it points into the hole $T$. Set
\begin{equation}
 \mathcal T:=\bigcup_{k\in\mathbb Z^N}(k+\overline T),\qquad
 \mathcal{W}=\R^N\setminus\mathcal T.
 \notag
\end{equation}
The set $\mathcal W$ is open, connected, and $Y$-periodic. For $0<\eps<1$ we define
\begin{equation}
 \OmE:=\Om\cap\eps\mathcal{W},
 \label{eq:Omegaeps-geometry}
\end{equation}
\begin{equation}
 I_\eps^{\rm int}:=\bigl\{k\in\mathbb Z^N:
        \eps(k+\overline Y)\subset\Om\bigr\},\qquad
 \GamE:=\bigcup_{k\in I_\eps^{\rm int}}\eps(k+\Gam),
 \label{eq:Gammaeps-cut}
\end{equation}
and 
\begin{equation}
 \Dep:=\partial\OmE\setminus\GamE.
 \label{eq:Deps-cut}
\end{equation}
No boundary--hole separation and no restriction to a special vanishing sequence are imposed; all limits are taken as $\eps\downarrow0$.
The geometry is shown in Figure~\ref{fig:geometry} below.

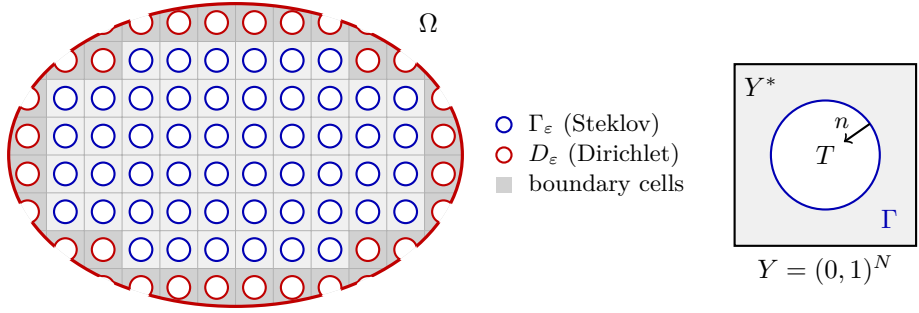
\begin{figure}[ht]
\centering
\begin{tikzpicture}[scale=1]
\begin{scope}
\clip (0,0) ellipse (3 and 2);
\fill[black!6] (-3,-2) rectangle (3,2);
\fill[black!18] (-3.00,-1.50) rectangle (-2.50,-1.00);
\fill[black!18] (-3.00,-1.00) rectangle (-2.50,-0.50);
\fill[black!18] (-3.00,-0.50) rectangle (-2.50,0.00);
\fill[black!18] (-3.00,0.00) rectangle (-2.50,0.50);
\fill[black!18] (-3.00,0.50) rectangle (-2.50,1.00);
\fill[black!18] (-3.00,1.00) rectangle (-2.50,1.50);
\fill[black!18] (-2.50,-1.50) rectangle (-2.00,-1.00);
\fill[black!18] (-2.50,1.00) rectangle (-2.00,1.50);
\fill[black!18] (-2.00,-2.00) rectangle (-1.50,-1.50);
\fill[black!18] (-2.00,-1.50) rectangle (-1.50,-1.00);
\fill[black!18] (-2.00,1.00) rectangle (-1.50,1.50);
\fill[black!18] (-2.00,1.50) rectangle (-1.50,2.00);
\fill[black!18] (-1.50,-2.00) rectangle (-1.00,-1.50);
\fill[black!18] (-1.50,1.50) rectangle (-1.00,2.00);
\fill[black!18] (-1.00,-2.00) rectangle (-0.50,-1.50);
\fill[black!18] (-1.00,1.50) rectangle (-0.50,2.00);
\fill[black!18] (-0.50,-2.00) rectangle (0.00,-1.50);
\fill[black!18] (-0.50,1.50) rectangle (0.00,2.00);
\fill[black!18] (0.00,-2.00) rectangle (0.50,-1.50);
\fill[black!18] (0.00,1.50) rectangle (0.50,2.00);
\fill[black!18] (0.50,-2.00) rectangle (1.00,-1.50);
\fill[black!18] (0.50,1.50) rectangle (1.00,2.00);
\fill[black!18] (1.00,-2.00) rectangle (1.50,-1.50);
\fill[black!18] (1.00,1.50) rectangle (1.50,2.00);
\fill[black!18] (1.50,-2.00) rectangle (2.00,-1.50);
\fill[black!18] (1.50,-1.50) rectangle (2.00,-1.00);
\fill[black!18] (1.50,1.00) rectangle (2.00,1.50);
\fill[black!18] (1.50,1.50) rectangle (2.00,2.00);
\fill[black!18] (2.00,-1.50) rectangle (2.50,-1.00);
\fill[black!18] (2.00,1.00) rectangle (2.50,1.50);
\fill[black!18] (2.50,-1.50) rectangle (3.00,-1.00);
\fill[black!18] (2.50,-1.00) rectangle (3.00,-0.50);
\fill[black!18] (2.50,-0.50) rectangle (3.00,0.00);
\fill[black!18] (2.50,0.00) rectangle (3.00,0.50);
\fill[black!18] (2.50,0.50) rectangle (3.00,1.00);
\fill[black!18] (2.50,1.00) rectangle (3.00,1.50);
\draw[black!30,very thin] (-3,-2) grid[step=0.5] (3,2);
\filldraw[fill=white,draw=blue!70!black,thick] (-2.25,-0.75) circle (0.15);
\filldraw[fill=white,draw=blue!70!black,thick] (-2.25,-0.25) circle (0.15);
\filldraw[fill=white,draw=blue!70!black,thick] (-2.25,0.25) circle (0.15);
\filldraw[fill=white,draw=blue!70!black,thick] (-2.25,0.75) circle (0.15);
\filldraw[fill=white,draw=blue!70!black,thick] (-1.75,-0.75) circle (0.15);
\filldraw[fill=white,draw=blue!70!black,thick] (-1.75,-0.25) circle (0.15);
\filldraw[fill=white,draw=blue!70!black,thick] (-1.75,0.25) circle (0.15);
\filldraw[fill=white,draw=blue!70!black,thick] (-1.75,0.75) circle (0.15);
\filldraw[fill=white,draw=blue!70!black,thick] (-1.25,-1.25) circle (0.15);
\filldraw[fill=white,draw=blue!70!black,thick] (-1.25,-0.75) circle (0.15);
\filldraw[fill=white,draw=blue!70!black,thick] (-1.25,-0.25) circle (0.15);
\filldraw[fill=white,draw=blue!70!black,thick] (-1.25,0.25) circle (0.15);
\filldraw[fill=white,draw=blue!70!black,thick] (-1.25,0.75) circle (0.15);
\filldraw[fill=white,draw=blue!70!black,thick] (-1.25,1.25) circle (0.15);
\filldraw[fill=white,draw=blue!70!black,thick] (-0.75,-1.25) circle (0.15);
\filldraw[fill=white,draw=blue!70!black,thick] (-0.75,-0.75) circle (0.15);
\filldraw[fill=white,draw=blue!70!black,thick] (-0.75,-0.25) circle (0.15);
\filldraw[fill=white,draw=blue!70!black,thick] (-0.75,0.25) circle (0.15);
\filldraw[fill=white,draw=blue!70!black,thick] (-0.75,0.75) circle (0.15);
\filldraw[fill=white,draw=blue!70!black,thick] (-0.75,1.25) circle (0.15);
\filldraw[fill=white,draw=blue!70!black,thick] (-0.25,-1.25) circle (0.15);
\filldraw[fill=white,draw=blue!70!black,thick] (-0.25,-0.75) circle (0.15);
\filldraw[fill=white,draw=blue!70!black,thick] (-0.25,-0.25) circle (0.15);
\filldraw[fill=white,draw=blue!70!black,thick] (-0.25,0.25) circle (0.15);
\filldraw[fill=white,draw=blue!70!black,thick] (-0.25,0.75) circle (0.15);
\filldraw[fill=white,draw=blue!70!black,thick] (-0.25,1.25) circle (0.15);
\filldraw[fill=white,draw=blue!70!black,thick] (0.25,-1.25) circle (0.15);
\filldraw[fill=white,draw=blue!70!black,thick] (0.25,-0.75) circle (0.15);
\filldraw[fill=white,draw=blue!70!black,thick] (0.25,-0.25) circle (0.15);
\filldraw[fill=white,draw=blue!70!black,thick] (0.25,0.25) circle (0.15);
\filldraw[fill=white,draw=blue!70!black,thick] (0.25,0.75) circle (0.15);
\filldraw[fill=white,draw=blue!70!black,thick] (0.25,1.25) circle (0.15);
\filldraw[fill=white,draw=blue!70!black,thick] (0.75,-1.25) circle (0.15);
\filldraw[fill=white,draw=blue!70!black,thick] (0.75,-0.75) circle (0.15);
\filldraw[fill=white,draw=blue!70!black,thick] (0.75,-0.25) circle (0.15);
\filldraw[fill=white,draw=blue!70!black,thick] (0.75,0.25) circle (0.15);
\filldraw[fill=white,draw=blue!70!black,thick] (0.75,0.75) circle (0.15);
\filldraw[fill=white,draw=blue!70!black,thick] (0.75,1.25) circle (0.15);
\filldraw[fill=white,draw=blue!70!black,thick] (1.25,-1.25) circle (0.15);
\filldraw[fill=white,draw=blue!70!black,thick] (1.25,-0.75) circle (0.15);
\filldraw[fill=white,draw=blue!70!black,thick] (1.25,-0.25) circle (0.15);
\filldraw[fill=white,draw=blue!70!black,thick] (1.25,0.25) circle (0.15);
\filldraw[fill=white,draw=blue!70!black,thick] (1.25,0.75) circle (0.15);
\filldraw[fill=white,draw=blue!70!black,thick] (1.25,1.25) circle (0.15);
\filldraw[fill=white,draw=blue!70!black,thick] (1.75,-0.75) circle (0.15);
\filldraw[fill=white,draw=blue!70!black,thick] (1.75,-0.25) circle (0.15);
\filldraw[fill=white,draw=blue!70!black,thick] (1.75,0.25) circle (0.15);
\filldraw[fill=white,draw=blue!70!black,thick] (1.75,0.75) circle (0.15);
\filldraw[fill=white,draw=blue!70!black,thick] (2.25,-0.75) circle (0.15);
\filldraw[fill=white,draw=blue!70!black,thick] (2.25,-0.25) circle (0.15);
\filldraw[fill=white,draw=blue!70!black,thick] (2.25,0.25) circle (0.15);
\filldraw[fill=white,draw=blue!70!black,thick] (2.25,0.75) circle (0.15);
\filldraw[fill=white,draw=red!75!black,thick] (-2.75,-1.25) circle (0.15);
\filldraw[fill=white,draw=red!75!black,thick] (-2.75,-0.75) circle (0.15);
\filldraw[fill=white,draw=red!75!black,thick] (-2.75,-0.25) circle (0.15);
\filldraw[fill=white,draw=red!75!black,thick] (-2.75,0.25) circle (0.15);
\filldraw[fill=white,draw=red!75!black,thick] (-2.75,0.75) circle (0.15);
\filldraw[fill=white,draw=red!75!black,thick] (-2.75,1.25) circle (0.15);
\filldraw[fill=white,draw=red!75!black,thick] (-2.25,-1.25) circle (0.15);
\filldraw[fill=white,draw=red!75!black,thick] (-2.25,1.25) circle (0.15);
\filldraw[fill=white,draw=red!75!black,thick] (-1.75,-1.75) circle (0.15);
\filldraw[fill=white,draw=red!75!black,thick] (-1.75,-1.25) circle (0.15);
\filldraw[fill=white,draw=red!75!black,thick] (-1.75,1.25) circle (0.15);
\filldraw[fill=white,draw=red!75!black,thick] (-1.75,1.75) circle (0.15);
\filldraw[fill=white,draw=red!75!black,thick] (-1.25,-1.75) circle (0.15);
\filldraw[fill=white,draw=red!75!black,thick] (-1.25,1.75) circle (0.15);
\filldraw[fill=white,draw=red!75!black,thick] (-0.75,-1.75) circle (0.15);
\filldraw[fill=white,draw=red!75!black,thick] (-0.75,1.75) circle (0.15);
\filldraw[fill=white,draw=red!75!black,thick] (-0.25,-1.75) circle (0.15);
\filldraw[fill=white,draw=red!75!black,thick] (-0.25,1.75) circle (0.15);
\filldraw[fill=white,draw=red!75!black,thick] (0.25,-1.75) circle (0.15);
\filldraw[fill=white,draw=red!75!black,thick] (0.25,1.75) circle (0.15);
\filldraw[fill=white,draw=red!75!black,thick] (0.75,-1.75) circle (0.15);
\filldraw[fill=white,draw=red!75!black,thick] (0.75,1.75) circle (0.15);
\filldraw[fill=white,draw=red!75!black,thick] (1.25,-1.75) circle (0.15);
\filldraw[fill=white,draw=red!75!black,thick] (1.25,1.75) circle (0.15);
\filldraw[fill=white,draw=red!75!black,thick] (1.75,-1.75) circle (0.15);
\filldraw[fill=white,draw=red!75!black,thick] (1.75,-1.25) circle (0.15);
\filldraw[fill=white,draw=red!75!black,thick] (1.75,1.25) circle (0.15);
\filldraw[fill=white,draw=red!75!black,thick] (1.75,1.75) circle (0.15);
\filldraw[fill=white,draw=red!75!black,thick] (2.25,-1.25) circle (0.15);
\filldraw[fill=white,draw=red!75!black,thick] (2.25,1.25) circle (0.15);
\filldraw[fill=white,draw=red!75!black,thick] (2.75,-1.25) circle (0.15);
\filldraw[fill=white,draw=red!75!black,thick] (2.75,-0.75) circle (0.15);
\filldraw[fill=white,draw=red!75!black,thick] (2.75,-0.25) circle (0.15);
\filldraw[fill=white,draw=red!75!black,thick] (2.75,0.25) circle (0.15);
\filldraw[fill=white,draw=red!75!black,thick] (2.75,0.75) circle (0.15);
\filldraw[fill=white,draw=red!75!black,thick] (2.75,1.25) circle (0.15);
\end{scope}
\begin{scope}
\begin{scope}[even odd rule]\clip (-3.5,-2.5) rectangle (3.5,2.5) (-2.75,-1.25) circle (0.15) (-2.75,-0.75) circle (0.15) (-2.75,-0.25) circle (0.15) (-2.75,0.25) circle (0.15) (-2.75,0.75) circle (0.15) (-2.75,1.25) circle (0.15) (-2.25,-1.25) circle (0.15) (-2.25,1.25) circle (0.15) (-1.75,-1.75) circle (0.15) (-1.75,-1.25) circle (0.15) (-1.75,1.25) circle (0.15) (-1.75,1.75) circle (0.15) (-1.25,-1.75) circle (0.15) (-1.25,1.75) circle (0.15) (-0.75,-1.75) circle (0.15) (-0.75,1.75) circle (0.15) (-0.25,-1.75) circle (0.15) (-0.25,1.75) circle (0.15) (0.25,-1.75) circle (0.15) (0.25,1.75) circle (0.15) (0.75,-1.75) circle (0.15) (0.75,1.75) circle (0.15) (1.25,-1.75) circle (0.15) (1.25,1.75) circle (0.15) (1.75,-1.75) circle (0.15) (1.75,-1.25) circle (0.15) (1.75,1.25) circle (0.15) (1.75,1.75) circle (0.15) (2.25,-1.25) circle (0.15) (2.25,1.25) circle (0.15) (2.75,-1.25) circle (0.15) (2.75,-0.75) circle (0.15) (2.75,-0.25) circle (0.15) (2.75,0.25) circle (0.15) (2.75,0.75) circle (0.15) (2.75,1.25) circle (0.15);
\draw[red!75!black,very thick] (0,0) ellipse (3 and 2);
\end{scope}
\end{scope}
\node at (2.55,1.75) {$\Omega$};
\begin{scope}[shift={(6.6,-1.2)}]
\fill[black!6] (0,0) rectangle (2.4,2.4);
\draw[thick] (0,0) rectangle (2.4,2.4);
\filldraw[fill=white,draw=blue!70!black,thick] (1.2,1.2) circle (0.72);
\node at (1.2,1.2) {$T$};
\node at (0.35,2.1) {$Y^*$};
\node[blue!70!black] at (2.05,0.35) {$\Gamma$};
\draw[->,thick] ({1.2+0.72*cos(35)},{1.2+0.72*sin(35)}) -- ({1.2+0.3*cos(35)},{1.2+0.3*sin(35)});
\node at (1.42,1.62) {\small $n$};
\node[below] at (1.2,0) {$Y=(0,1)^N$};
\end{scope}
\begin{scope}[shift={(3.55,0.1)}]
\draw[blue!70!black,thick] (0,0.3) circle (0.1); \node[right] at (0.2,0.3) {\small $\Gamma_\eps$ (Steklov)};
\draw[red!75!black,thick] (0,-0.1) circle (0.1); \node[right] at (0.2,-0.1) {\small $D_\eps$ (Dirichlet)};
\fill[black!18] (-0.1,-0.6) rectangle (0.1,-0.4); \node[right] at (0.2,-0.5) {\small boundary cells};
\end{scope}
\end{tikzpicture}
\caption{The cut-perforation geometry. Holes of complete cells
$\eps(k+Y)\subset\Om$ carry the Steklov condition ($\GamE$); the outer
boundary and all holes of cells meeting $\partial\Om$, including those cut
by $\partial\Om$, carry the Dirichlet condition ($\Dep$); the part of
$\partial\Om$ lying inside a hole does not belong to $\partial\OmE$. Right:
the reference cell.}
\label{fig:geometry}
\end{figure}

In $\OmE$ we consider the spectral problem
\begin{equation}
\left\{
\begin{aligned}
 -\Div\!\left(A\!\left(\frac{x}{\eps}\right)\nabla u_\eps\right)
 &=\lambda_\eps q\!\left(\frac{x}{\eps}\right)u_\eps
 &&\text{in }\OmE,\\
 A\!\left(\frac{x}{\eps}\right)\nabla u_\eps\cdot n
 &=\eps\lambda_\eps\rho\!\left(\frac{x}{\eps}\right)u_\eps
 &&\text{on }\GamE,\\
 u_\eps&=0 &&\text{on }\Dep
\end{aligned}
\right.
\tag{$\mathscr P_\eps$}\label{eq:intro-problem}
\end{equation}
with the following hypotheses.

\Needspace{12\baselineskip}
\begin{enumerate}[label=\textup{(H\arabic*)},ref=H\arabic*,leftmargin=1.2cm]
 \item\label{hyp:diffusion} \textbf{Diffusion.}
 The matrix $A\in L^\infty(\mathcal{W})^{N\times N}$ is real, symmetric and $Y$-periodic, and 
 there exist $0<\alpha_0\le\beta_0<\infty$ such that for almost every $y\in \Ys$
 \begin{equation}
  \alpha_0|\xi|^2\le A(y)\xi\cdot\xi\le\beta_0|\xi|^2\qquad \text{for all} \ \xi\in\R^N.
 \notag
 \end{equation}

 \item\label{hyp:spectral-densities} \textbf{Densities and combined indefiniteness.}
 The densities $q\in L^\infty(\mathcal{W})$ and $\rho\in L^\infty(\partial\mathcal{W})$ are
 real, $Y$-periodic and the signed measure $q\,dy+\rho\,dS_\Gam$ on $\overline{\Ys}$ has
 nontrivial positive and negative parts; equivalently,
 \begin{equation}
 \bigl|\{q>0\}\bigr|
  +\mathcal H^{N-1}\bigl(\{\rho>0\}\bigr)>0,\qquad \bigl|\{q<0\}\bigr|
  +\mathcal H^{N-1}\bigl(\{\rho<0\}\bigr)>0,
 \notag
 \end{equation}
 the volume measures being taken in $\Ys$ and the surface measures on
 $\Gam$. The two signs may be carried by either component; both the purely Steklov case $q\equiv0$ and the purely bulk case $\rho\equiv0$ are included.

 \item\label{hyp:cell-regularity} \textbf{Cell regularity.}
 For some $\alpha_Y\in(0,1)$,
 \[
 \tag{$\mathrm{H3}_A$}
  \Gam\in C^{2,\alpha_Y},\qquad
  A\in C^{0,\alpha_Y}_\#(\overline\Ys)^{N\times N},
 \]
 (that is, $A$ is H\"older continuous on $\overline{\mathcal W}$)
 and
 \[
 \tag{$\mathrm{H3}_m$}
  \rho\in C^{0,\alpha_Y}(\Gam).
 \]

 \item\label{hyp:macro-regularity} \textbf{Macroscopic regularity.}
 Either \textup{(a)} $\partial\Om\in C^{1,1}$, or \textup{(b)} $\Om$ is
 convex.
\end{enumerate}

Hypothesis \textup{(H4)} is used to obtain the $H^2$-regularity of the
constant-coefficient Dirichlet problem, which holds for $C^{1,1}$ domains
\cite[Chapter~2]{Grisvard1985} and for convex domains
\cite[Theorem~3.2.1.2]{Grisvard1985} (after a linear change of variables
reducing $\Ah$ to the identity, which preserves convexity). Unless stated
otherwise, constants depend on the fixed geometry, on the ellipticity and
regularity data, and on a fixed spectral cluster, but never on $\eps$. 
The qualitative results use \textup{(H1)--(H2)} on the regular branch and \textup{(H1)--(H3)} on the factorized branch and in the critical regime; all $O(\eps)$ estimates use \textup{(H1)--(H4)}.

Since $\eps|\GamE|$ remains of
order one, the factor $\eps$ in the Steklov condition puts the bulk and
surface contributions at the same macroscopic order. The analysis below shows that the quantity governing
the spectral structure is the combined effective mass
\begin{equation}
 m_{\mathrm{eff}}
 :=\int_{\Ys}q(y)\,dy+\int_\Gam\rho(y)\,dS_y.
 \label{eq:intro-meff}
\end{equation}
Its sign, and especially its cancellation, does not merely modify the coefficients of the limit problem: it changes the spectral scale and the cell mechanism responsible for homogenization.

\subsection{Literature review}\label{subsec:intro-literature}

Periodic homogenization by two-scale convergence goes back to Nguetseng
\cite{Nguetseng1989} and Allaire \cite{Allaire1992}; standard references for
periodic and perforated media include
\cite{CioranescuDonato1999,CioranescuSaintJeanPaulin1979,JikovKozlovOleinik1994,OleinikShamaevYosifian1992}.
Two-scale convergence on periodic hypersurfaces was developed by Neuss-Radu
\cite{NeussRadu1996} and by Allaire--Damlamian--Hornung
\cite{AllaireDamlamianHornung1996}; the equivalent unfolding approach in
domains with holes is in \cite{CioranescuDamlamianDonatoGrisoZaki2012}.

Spectral homogenization goes back to Kesavan \cite{Kesavan1979} and, in
perforated domains, to Vanninathan \cite{Vanninathan1981}; Steklov
eigenvalues in periodic structures are studied in
\cite{ConcaPlanchardVanninathan1995,Douanla2010}, and quantitative estimates for a
perforated Steklov problem were obtained by Pastukhova
\cite{Pastukhova1995}. For indefinite weights, a quadratic operator pencil already appears in the
zero-mean locally periodic case \cite{PankratovaPiatnitski2011}. In a recent
contribution, Aiyappan--Chattaraj--Pettersson
\cite{AiyappanChattarajPettersson2026} study a locally periodic indefinite
bulk spectrum in a thin cylinder, where factorization is coupled with
dimension reduction and localization. Recent work on Steklov homogenization
includes \cite{DonatoPerezMartinez2026}. These settings do not contain a
simultaneous signed bulk--surface mass. The works closest to the present one
are the following.

Nazarov--Pankratova--Piatnitski \cite{NazarovPankratovaPiatnitski2011} treat
a purely bulk indefinite weight without holes, and Chiad\`o
Piat--Nazarov--Piatnitski \cite{ChiadoNazarovPiatnitski2012} a purely surface
indefinite weight on holes that may be cut by $\partial\Om$; both describe the
two-sided spectrum in three regimes and obtain $O(\eps^{1/2})$
estimates. The bulk and surface indefinite weights were treated separately,
by two-scale convergence and without rates, in \cite{Douanla2013Volume} and
\cite{Douanla2013Steklov}, for holes that do not meet $\partial\Om$.
Girouard--Henrot--Lagac\'e \cite{GirouardHenrotLagace2021},
Khrabustovskyi--Taskinen \cite{KhrabustovskyiTaskinen2026} and
Huang--Venkatraman \cite{HuangVenkatraman2026} consider a positive Steklov
parameter on small holes at the critical size, where the surface mass
produces a positive bulk weight in the limit and only a one-sided spectrum
appears; \cite{GirouardHenrotLagace2021} obtains convergence to a
dynamical-boundary spectrum, interpolating between the Steklov and Neumann
problems, \cite{KhrabustovskyiTaskinen2026} operator-norm resolvent
convergence with quantitative spectral estimates, and
\cite{HuangVenkatraman2026} optimal eigenpair rates and a second-order term
through an indefinite Coulomb-type interaction between boundary and bulk
densities. In all these works the
spectral mass is carried either by the bulk or by the perforation boundaries.
Here it is carried by a single signed measure $q\,dy+\rho\,dS_\Gam$, acting
simultaneously in the bulk and on the perforation boundaries at the same
scale and with arbitrary signs. The novelty of the paper lies in three
consequences of this structure: (i) a mixed bulk--surface compensation, which
allows the critical regime $m_{\rm eff}=0$ to arise through bulk--surface
cancellation even when neither density has zero mean; (ii) the compatibility of the exact ground-state factorization, of
the compensation identity $b_\eps=\eps c_\eps$
(Section~\ref{sec:critical-regime}) and of an exact Dirichlet layer with a
cut-perforation geometry in which holes may meet $\partial\Om$; and (iii) the
resulting $O(\eps)$ fixed-cluster estimates for eigenvalues, spectral
projectors and aligned eigenspaces in all three regimes, compared with the
$O(\eps^{1/2})$ rates obtained in
\cite{ChiadoNazarovPiatnitski2012,NazarovPankratovaPiatnitski2011}.

\subsection{Main results}\label{subsec:intro-main-results}

Let $\Ah$ be the homogenized matrix of $A$ in the perforated cell (see
\eqref{eq:Ah}), and let
\begin{equation}
 -\Div(\Ah\nabla\phi)=\mu\phi\quad\text{in }\Om,
 \qquad \phi=0\quad\text{on }\partial\Om,
 \notag
\end{equation}
have eigenvalues $0<\mu_1^0\le\mu_2^0\le\cdots\to+\infty$, repeated
according to multiplicity. By Proposition~\ref{prop:prelim-two-sign-spectrum},
under \textup{(H1)--(H2)} and for every $\eps\in(0,1)$ such that $\Om$ contains at
least one complete cell (in particular for $\eps<\eps_\Om$, see
\eqref{eq:eps-Omega}), the spectrum of \eqref{eq:intro-problem} consists of
two sequences $0<\lambda_{\eps,1}^+\le\lambda_{\eps,2}^+\le\cdots\to+\infty$
and $0>\lambda_{\eps,1}^-\ge\lambda_{\eps,2}^-\ge\cdots\to-\infty$.
A \emph{cluster} is a set of indices $k,\ldots,k+r-1$ with
$\mu_{k-1}^0<\mu_k^0=\cdots=\mu_{k+r-1}^0<\mu_{k+r}^0$ (with $\mu_0^0:=0$); all statements are
made for whole clusters, and ``aligned eigenspaces'' refers to convergence
of eigenfunctions after multiplication by an orthogonal $r\times r$ matrix.
The cell mass form is
\begin{equation}
 \mathfrak b_Y(\phi,\psi)
 :=\int_{\Ys}q\phi\psi\,dy+\int_\Gam\rho\phi\psi\,dS_y,
 \qquad \phi,\psi\in H^1_\#(\Ys).
 \label{eq:intro-cell-mass-form}
\end{equation}

For a $Y$-periodic function $g$ we write $g^\eps(x):=g(x/\eps)$, and
$V_\eps$ denotes the space of functions in $H^1(\OmE)$ vanishing on $\Dep$,
normed by $\|\nabla v\|_{L^2(\OmE)}$ (see \eqref{eq:Veps}).

We first state the regular branch.

\begin{theorem}
\label{thm:intro-regular}
Assume \textup{(H1)--(H2)} and $m_{\rm eff}>0$. For every fixed $k\ge1$,
\begin{equation}
 \lambda_{\eps,k}^{+}\longrightarrow
 \lambda_k^0:=\frac{\mu_k^0}{m_{\rm eff}},
 \notag
\end{equation}
every cluster of multiplicity $r$ attracts exactly $r$ microscopic
eigenvalues, and the aligned eigenspaces converge. Under \textup{(H1)--(H4)}, for every cluster $\{k,\ldots,k+r-1\}$,
\begin{equation}
 \max_{1\le j\le r}
 \left|\lambda_{\eps,k+j-1}^+-\frac{\mu_k^0}{m_{\rm eff}}\right|
 \le C_k\eps,
 \notag
\end{equation}
and the spectral projectors and aligned eigenspaces satisfy $O(\eps)$
estimates with respect to explicit first-order reconstructions
(Theorems~\ref{thm:qual-positive} and \ref{thm:quantitative-positive-cluster}).
\end{theorem}

We next state the factorized branch.

\begin{theorem}
\label{thm:intro-factorized}
Assume \textup{(H1)--(H3)} and $m_{\rm eff}>0$. The cell problem
\begin{equation}
 \int_{\Ys}A\nabla_y\theta\cdot\nabla_y\psi\,dy
 =\lambda\,\mathfrak b_Y(\theta,\psi)
 \qquad\forall\psi\in H^1_\#(\Ys)
 \notag
\end{equation}
has a simple first negative eigenvalue $\lambda_1^-<0$, with an
eigenfunction $\theta_1^-\in C^{1,\alpha'}(\overline{\Ys})$,
$0<\alpha'<\alpha_Y$, such that $\min_{\overline{\Ys}}\theta_1^->0$. With $\Theta_\eps=(\theta_1^-)^\eps$ one
has the exact factorization
\begin{equation}
 a_\eps(\Theta_\eps v,\Theta_\eps w)
 -\frac{\lambda_1^-}{\eps^2}b_\eps(\Theta_\eps v,\Theta_\eps w)
 =\int_{\OmE}\bigl((\theta_1^-)^2A\bigr)^\eps\nabla v\cdot\nabla w\,dx
 \qquad\forall v,w\in V_\eps,
 \notag
\end{equation}
where $a_\eps$ and $b_\eps$ are the energy and mass forms \eqref{eq:ae}--\eqref{eq:be}.
Let $\widetilde\mu_k^0$ be the eigenvalues of
$-\Div(\widetilde A^{\rm hom}\nabla\cdot)$ in $H_0^1(\Om)$, where
$\widetilde A^{\rm hom}$ is the homogenized matrix of
$\widetilde A=(\theta_1^-)^2A$ in $\Ys$, and
$\widetilde m_{\rm eff}:=\int_{\Ys}(\theta_1^-)^2q\,dy+\int_\Gam(\theta_1^-)^2\rho\,dS_y<0$.
Then
\begin{equation}
 \lambda_{\eps,k}^{-}-\frac{\lambda_1^-}{\eps^2}
 \longrightarrow \frac{\widetilde\mu_k^0}{\widetilde m_{\rm eff}},
 \notag
\end{equation}
with preservation of multiplicity and convergence of the aligned
transformed eigenspaces. Under \textup{(H1)--(H4)}, for every cluster
$\{k,\ldots,k+r-1\}$ of $(\widetilde\mu_j^0)_{j\ge1}$ (defined as above
with $\mu_j^0$ replaced by $\widetilde\mu_j^0$),
\begin{equation}
 \max_{1\le j\le r}\left|\lambda_{\eps,k+j-1}^{-}
 -\frac{\lambda_1^-}{\eps^2}
 -\frac{\widetilde\mu_k^0}{\widetilde m_{\rm eff}}\right|
 \le C_k\eps,
 \label{eq:intro-negative-rate}
\end{equation}
together with $O(\eps)$ projector and eigenspace estimates; the original
modes carry the factor $(\theta_1^-)^\eps$
(Theorems~\ref{neg:thm:global-factorization}, \ref{neg:thm:qualitative}
and \ref{neg:thm:quantitative-cluster}).
\end{theorem}

The ratio $\widetilde\mu_k^0/\widetilde m_{\rm eff}$ does not depend on the
normalization of $\theta_1^-$. Estimate \eqref{eq:intro-negative-rate} gives
an absolute $O(\eps)$ error at the scale $\eps^{-2}$: there is no term of
order $\eps^{-1}$, as a consequence of the exact ground-state factorization,
which uses the symmetry of $A$. If
$m_{\rm eff}<0$, the exact involution
$(q,\rho,\lambda)\mapsto(-q,-\rho,-\lambda)$ of
Proposition~\ref{sym:prop:exact-involution} exchanges the two branches, and
Theorems~\ref{thm:intro-regular}--\ref{thm:intro-factorized} transfer
with the branches exchanged (Corollary~\ref{sym:cor:negative-effective-mass}).

At zero effective mass, the spectral scale changes.

\begin{theorem}
\label{thm:intro-critical}
Assume \textup{(H1)--(H3)} and $m_{\rm eff}=0$. Let
$\chi^0\in H^1_\#(\Ys)/\R$ solve
\begin{equation}
 \int_{\Ys}A\nabla_y\chi^0\cdot\nabla_y\psi\,dy
 =\mathfrak b_Y(1,\psi)\quad\forall\psi\in H^1_\#(\Ys),
 \notag
\end{equation}
and set $\nu^2:=\int_{\Ys}A\nabla_y\chi^0\cdot\nabla_y\chi^0\,dy$, which is
positive by Lemma~\ref{crit:lem:nu-positive}.
The limit problem is the quadratic pencil
\begin{equation}
 \int_\Om\Ah\nabla u_0\cdot\nabla v\,dx
 =\Lambda^2\nu^2\int_\Om u_0v\,dx
 \qquad\forall v\in H_0^1(\Om),
 \notag
\end{equation}
so that each $\mu_k^0$ generates the two levels
$\Lambda_k^\sigma=\sigma\sqrt{\mu_k^0}/\nu$, $\sigma\in\{-1,+1\}$, and
\begin{equation}
 \eps\lambda_{\eps,k}^{\sigma}\longrightarrow\Lambda_k^\sigma.
 \notag
\end{equation}
For each sign, a cluster of multiplicity $r$ attracts exactly $r$
rescaled eigenvalues, and the aligned eigenspaces converge to the same
macroscopic eigenspace; for $\phi$ in the eigenspace of $\mu_k^0$, the
first-order corrector is
$\sum_{j=1}^N\chi^j\partial_j\phi+\Lambda_k^\sigma\phi\chi^0$ and thus
contains the critical term $\Lambda_k^\sigma\phi\chi^0$. Under \textup{(H1)--(H4)}, for every cluster $\{k,\ldots,k+r-1\}$,
\begin{equation}
 \max_{1\le j\le r}
 |\eps\lambda_{\eps,k+j-1}^\sigma-\Lambda_k^\sigma|
 \le C_k\eps,
 \qquad\text{that is,}\qquad
 \lambda_{\eps,k+j-1}^\sigma=\frac{\Lambda_k^\sigma}{\eps}+O_k(1),
 \notag
\end{equation}
together with $O(\eps)$ projector and eigenspace estimates
(Theorems~\ref{crit:thm:qual-critical} and
\ref{crit:thm:quantitative-cluster}).
\end{theorem}

In the critical regime the error on $\lambda_\eps$ itself is thus $O(1)$
(see Section~\ref{sec:conclusion}).

\subsection{Ideas of the proofs and organization}\label{subsec:intro-ideas}

The analysis faces three difficulties. First, the mass form $b_\eps$ is
indefinite, so the classical theory for positive weights does not apply;
we work with compact self-adjoint operators on the energy space and with a
fixed-cluster principle on varying Hilbert spaces
(Theorem~\ref{thm:abstract-fixed-cluster}), which adapts the perturbation
and spectral approximation theory of \cite{BabuskaOsborn1991,Kato1976,Osborn1975} to $\eps$-dependent spaces. Second, the bulk and surface
masses must be compared with their mean at the rate $O(\eps)$; this is done
through an exact representation of the centered bulk--surface measure as the
divergence of a bounded periodic field, which in the critical regime becomes
the compensation identity $b_\eps=\eps\,c_\eps$ with a uniformly bounded
form $c_\eps$. Third, perforations may meet $\partial\Om$, so the usual comparison
with a problem posed on a fixed domain away from $\partial\Om$, available when
the holes are separated from the boundary, cannot be used. We
compare instead with a cut-off first-order expansion and add an exact
$A^\eps$-harmonic Dirichlet layer on the whole microscopic Dirichlet
boundary. Although its energy norm is $O(\eps^{1/2})$, it is orthogonal to
the diffusion residual, while its $L^2$, weighted and trace norms are
$O(\eps)$ by a duality argument; this is what yields the $O(\eps)$ rates. 

Section~\ref{sec:preliminaries} collects the two-scale tools, the uniform
extension, the common boundary-layer estimates and the abstract
fixed-cluster principle. Section~\ref{sec:noncritical-regimes} treats the
regular and factorized branches together with the case $m_{\rm eff}<0$, while
Section~\ref{sec:critical-regime} treats the critical regime.
Section~\ref{sec:conclusion} lists open questions.

\section{Two-scale convergence and common analytical tools}\label{sec:preliminaries}

In the abstract two-scale tools below, let
$\mathscr E\subset(0,1)$ be an arbitrary set having $0$ as an accumulation
point. Abstract families are written $(z_\eps)_{\eps\in\mathscr E}$, and
$\eps\to0$ means $\eps\to0$ with $\eps\in\mathscr E$. For the physical cut-perforation problem we use the full parameter set $(0,1)$. Whenever compactness requires extraction, we start from an arbitrary sequence tending to zero, extract a subsequence, and thereafter denote the extracted parameter again by $\eps$. For every measurable set $E$, we
denote its indicator function by $\ind_E$, and $\ind_I(\mathcal A)$ denotes
the spectral projector of a self-adjoint operator $\mathcal A$ associated
with a Borel set $I\subset\R$. The letter $\chi$ is reserved for cell
correctors.

Throughout, the subscript $\#$ denotes $Y$-periodicity. For $X=Y$ or $Y^*$, the spaces $L^p_\#(X)$, $W^{m,p}_\#(X)$, $C^m_\#(X)$, and $C^{m,\alpha}_\#(\overline X)$ denote the corresponding periodic spaces on $\R^N$ when $X=Y$, and on $\mathcal W$ when $X=Y^*$. We set $H^1_\#(X)=W^{1,2}_\#(X)$ and use
$H^1_\#(X)/\R$ for the quotient by constants, with the zero-mean
representative on $X$ whenever needed. The analogous convention applies to
periodic distributions; in particular, $\mathcal D'_\#(Y)$ denotes
distributions on the flat torus. On $\Gam$, the subscript $\#$ denotes
periodic repetition on $\bigcup_{k\in\mathbb Z^N}(k+\Gam)$. For every
periodically extended cell function $g$, we use the standard shorthand
$g^\eps(x):=g(x/\eps)$ whenever the expression is meaningful.

\subsection{Volume and surface two-scale convergence}
\label{subsec:two-scale-preliminaries}
Let $\varepsilon\in(0,1)$. Since
$\GamE$ consists of the holes of the complete cells, and
\begin{equation}
 \Omega\setminus\bigcup_{k\in I_\eps^{\rm int}}\eps(k+\overline Y)\subset\{x:\dist(x,\partial\Omega)<C_Y\eps\}
 \label{eq:omitted-cells-strip}
\end{equation}
for a geometric constant $C_Y$ (one may take $C_Y=2\operatorname{diam}Y$), the cells omitted from the surface terms
occupy an $O(\eps)$ boundary strip. With the same constant, every cell
$\eps(k+\overline Y)$ meeting $\partial\Om$, and hence every point of
$\Dep\setminus\partial\Om$, lies in $\{x:\dist(x,\partial\Om)<C_Y\eps\}$. 

\begin{definition}
\label{def:volume-two-scale}
Let $(z_\eps)_{\eps\in\mathscr E}\subset L^2(\Om)$. We say that this
family two-scale converges weakly to $z_0\in L^2(\Om\times Y)$, and write
$z_\eps\overset{2}{\rightharpoonup}z_0$, if
\begin{equation}
 \lim_{\substack{\eps\to0\\ \eps\in\mathscr E}}
 \int_\Om z_\eps(x)
 \psi\!\left(x,\frac{x}{\eps}\right)\,dx
 =\int_\Om\int_Y z_0(x,y)\psi(x,y)\,dy\,dx
 \notag
\end{equation}
for every $\psi\in L^2(\Om;C_\#(Y))$. We say that the family
$(z_\eps)_{\eps\in\mathscr E}$ two-scale converges strongly to $z_0$, and
write $z_\eps\overset{2}{\longrightarrow}z_0$, if
$z_\eps\overset{2}{\rightharpoonup}z_0$ and
$\|z_\eps\|_{L^2(\Om)}\to\|z_0\|_{L^2(\Om\times Y)}$ as $\eps\to0$,
$\eps\in\mathscr E$.
\end{definition}

\begin{definition}
\label{def:surface-two-scale}
Let $(z_\eps)_{\eps\in\mathscr E}$ satisfy
$z_\eps\in L^2(\GamE)$ for every
$\eps\in\mathscr E$. We say that this family two-scale converges on the
periodic surfaces to $z_0\in L^2(\Om\times\Gam)$ if
\begin{equation}
 \lim_{\substack{\eps\to0\\ \eps\in\mathscr E}}
 \eps\int_{\GamE}z_\eps(x)
 \psi\!\left(x,\frac{x}{\eps}\right)\,dS_x
 =\int_\Om\int_\Gam z_0(x,y)\psi(x,y)\,dS_y\,dx
 \label{eq:two-scale-surface-definition}
\end{equation}
for every $\psi\in C(\overline\Om;C_\#(Y))$.
\end{definition}

Near the outer boundary, perforated domains require some care: extension
properties and the admissible behavior of holes close to $\partial\Om$ are
discussed in \cite{ChiadoPiatPiatnitski2010,DamlamianDonato2002}. In the cut
geometry considered here, the Dirichlet condition on the cut holes makes a
uniform extension available without any separation assumption. The Dirichlet
condition is imposed in the Sobolev-capacitary sense.  Set
\begin{equation}
 V_\eps:=H^1_{\Dep}(\OmE)
 :=\overline{\mathcal D_\eps}^{\,H^1(\OmE)},\qquad
 \|v\|_{V_\eps}:=\|\nabla v\|_{L^2(\OmE)},
 \label{eq:Veps}
\end{equation}
where
\begin{equation}
 \mathcal D_\eps:=
 \left\{\phi|_{\OmE}:\ \phi\in C_c^\infty(\mathbb R^N),\quad
 \operatorname{dist}(\operatorname{supp}\phi,\Dep)>0\right\}.
 \notag
\end{equation}

\begin{lemma}
\label{lem:cut-uniform-extension}
There exists a linear operator
\begin{equation}
 P_\eps:V_\eps\to H_0^1(\Om),\qquad P_\eps v=v\ \hbox{a.e. in }\OmE,
 \notag
\end{equation}
with constants independent of sufficiently small $\eps$ such that
\begin{align}
 \|P_\eps v\|_{L^2(\Om)}+\|\nabla P_\eps v\|_{L^2(\Om)}
 &\le C\|v\|_{H^1(\OmE)},\label{eq:extension}\\
 \|\nabla P_\eps v\|_{L^2(\Om)}
 &\le C\|\nabla v\|_{L^2(\OmE)}.\label{eq:common-cut-extension-gradient}
\end{align}
On every complete cell $Y_{\eps,k}=\eps(k+Y)$,
$k\in I_\eps^{\rm int}$, it may be chosen so that
\begin{equation}
 \|P_\eps v\|_{L^2(Y_{\eps,k})}
 \le C\bigl(\|v\|_{L^2(Y_{\eps,k}\cap\OmE)}
 +\eps\|\nabla v\|_{L^2(Y_{\eps,k}\cap\OmE)}\bigr).
 \label{eq:common-cut-extension-local-L2}
\end{equation}
\end{lemma}

\begin{proof}
Let $E:H^1(Y^*)\to H^1(Y)$ be a constant-preserving extension furnished by
\cite{AcerbiEtAl1992}, with the corresponding gradient and $L^2$ estimates.
For $v\in\mathcal D_\eps$, use its rescaling in each complete cell.  In a
cell meeting $\partial\Omega$, every perforation boundary belongs to
$D_\eps$; since $v$ vanishes near $D_\eps$, fill those holes by zero, which
creates no $H^1$ jump.  As $T\Subset Y$, cell interfaces lie in the solid
phase, where neighboring constructions both equal $v$.  The resulting
function lies in $H_0^1(\Omega)$ and equals $v$ in $\Omega_\eps$.
Rescaling the fixed-cell estimates and summing gives
\eqref{eq:extension}--\eqref{eq:common-cut-extension-local-L2}; preservation
of constants gives \eqref{eq:common-cut-extension-gradient}.  Uniformity in
$\eps$ permits passage to the $H^1(\Omega_\eps)$-closure of
$\mathcal D_\eps$.
\end{proof}
The following classical compactness assertions are understood in the sequential sense specified above. Their proofs are omitted since they are classical; see, for instance, \cite{Douanla2013Steklov}.
\begin{theorem}
\label{thm:two-scale-compactness-prelim}

\begin{enumerate}[label=\textup{(\roman*)}]
 \item If $(z_\eps)_{\eps\in\mathscr E}\subset L^2(\Om)$ is bounded,
 then, along a subsequence, there exists
 $z_0\in L^2(\Om\times Y)$ such that
 $z_\eps\overset{2}{\rightharpoonup}z_0$.
 \item If $(z_\eps)_{\eps\in\mathscr E}$ satisfies
\begin{equation}
 \sup_{\eps\in\mathscr E}\
 \eps\int_{\GamE}|z_\eps|^2\,dS_x<\infty,
 \label{eq:surface-two-scale-bound}
\end{equation}
 then, along a subsequence, it two-scale converges on surfaces in the sense of
 \eqref{eq:two-scale-surface-definition}.
 \item If $(v_\eps)_{0<\eps<1}$, with $v_\eps\in V_\eps$, is bounded
 in the energy norm, then, along a subsequence, there exist functions
 \[
  v\in H_0^1(\Om),\qquad
  \widehat v_1\in L^2\bigl(\Om;H^1_\#(Y)/\R\bigr)
 \]
 such that
 \begin{align}
  P_\eps v_\eps&\rightharpoonup v&&\text{in }H_0^1(\Om),\label{eq:prelim-H1-compactness}\\
  P_\eps v_\eps&\longrightarrow v&&\text{in }L^2(\Om),\label{eq:prelim-L2-compactness}\\
  \nabla(P_\eps v_\eps)&\overset{2}{\rightharpoonup}
  \nabla_xv+\nabla_y\widehat v_1&&\text{in }L^2(\Om\times Y)^N.
  \label{eq:prelim-gradient-two-scale}
 \end{align}
 Moreover, along the same subsequence,
 \begin{equation}
  \ind_{\Omega_\eps}=(\ind_{\mathcal W})^\varepsilon
  \overset{2}{\longrightarrow}
  \ind_{Y^*}\quad\hbox{strongly two-scale},
  \label{eq:prelim-characteristic-strong-two-scale}
 \end{equation}
 and 
 \begin{equation}
  \ind_{\Omega_\eps}\nabla(P_\eps v_\eps)
  \overset{2}{\rightharpoonup}
  \ind_{Y^*}
  \bigl(\nabla_xv+\nabla_y\widehat v_1\bigr)
  \quad\text{in }L^2(\Om\times Y)^N.
  \label{eq:prelim-solid-gradient-two-scale}
 \end{equation}
 Thus the physical corrector is
 $v_1:=\widehat v_1|_{Y^*}\in
 L^2(\Om;H^1_\#(Y^*)/\R)$, after fixing its mean on $Y^*$.
\end{enumerate}
\end{theorem}

The next proposition turns the compactness above into the bulk and surface mass limits needed in the spectral analysis.

\begin{proposition}
\label{prop:prelim-trace-mass-limit}
Let $(u_\eps)_{0<\eps<1}\subset V_\eps$ and $(v_\eps)_{0<\eps<1}\subset V_\eps$ be bounded in the energy norm. Suppose that $P_\eps u_\eps\to u$ and
$P_\eps v_\eps\to v$ strongly in $L^2(\Om)$. Then, for every
$\psi\in C_c^\infty(\Om;C_\#^\infty(Y))$,
\begin{equation}
 \eps\int_{\GamE}u_\eps(x)\psi\!\left(x,\frac{x}{\eps}\right)\,dS_x
 \longrightarrow
 \int_\Om\int_\Gam u(x)\psi(x,y)\,dS_y\,dx.
 \label{eq:prelim-trace-two-scale}
\end{equation}
If, in addition,
$q\in L^\infty_\#(Y^*)$ and $\rho\in L^\infty_\#(\Gamma)$ are real, then
\begin{align}
 \int_{\OmE}q^\eps u_\eps v_\eps\,dx
 &\longrightarrow
 \left(\int_{\Ys} q(y)\,dy\right)\int_\Om uv\,dx,
 \label{eq:prelim-bulk-mass-limit}\\
 \eps\int_{\GamE}\rho^\eps u_\eps v_\eps\,dS_x
 &\longrightarrow
 \left(\int_\Gam\rho(y)\,dS_y\right)\int_\Om uv\,dx.
 \label{eq:prelim-surface-mass-limit}
\end{align}
Consequently,
\begin{equation}
 \int_{\OmE}q^\eps u_\eps v_\eps\,dx
 +\eps\int_{\GamE}\rho^\eps u_\eps v_\eps\,dS_x
 \longrightarrow
 m_{\mathrm{eff}}\int_\Om uv\,dx.
 \label{eq:prelim-combined-mass-limit}
\end{equation}
\end{proposition}

\begin{proof}
Set $U_\eps=P_\eps u_\eps$ and $W_\eps=P_\eps v_\eps$. By assumption,
$U_\eps$ and $W_\eps$ are bounded in $H_0^1(\Om)$ and
\[
 U_\eps\to u,\qquad W_\eps\to v
 \quad\text{strongly in }L^2(\Om).
\]

The trace convergence \eqref{eq:prelim-trace-two-scale} follows from
\cite[Proposition~2.8]{Douanla2013Steklov}, applied to $U_\eps$. In the
present geometry $\GamE$ contains the surfaces of complete cells only. The
omitted cells lie in the boundary strip \eqref{eq:omitted-cells-strip}; since
$\psi$ has compact support in $\Om$, they do not meet
$\operatorname{supp}\psi$ for all sufficiently small $\eps$. Since the trace
of $U_\eps$ on $\GamE$ is $u_\eps$, this proves
\eqref{eq:prelim-trace-two-scale}.

For the bulk term, extend $q$ by zero to $T$, denote the resulting
$Y$-periodic function by $q^\circ$, and set
\[
 \bar q:=\int_Y q^\circ(y)\,dy=\int_{\Ys} q(y)\,dy,
 \qquad h:=q^\circ-\bar q.
\]
Then $\int_Y h\,dy=0$, and \cite[Lemma~4.4]{ChiadoNazarovPiatnitski2012}
gives
\[
 \left|\int_\Om h^\eps U_\eps W_\eps\,dx\right|
 \le C\eps\|U_\eps\|_{H_0^1(\Om)}\|W_\eps\|_{H_0^1(\Om)}=o(1).
\]
Consequently,
\begin{equation*}
 \int_{\OmE}q^\eps u_\eps v_\eps\,dx
 =\int_\Om(q^\circ)^\eps U_\eps W_\eps\,dx
 =\bar q\int_\Om U_\eps W_\eps\,dx+o(1)
 \longrightarrow \bar q\int_\Om uv\,dx,
\end{equation*}
which is \eqref{eq:prelim-bulk-mass-limit}.

For the surface term, set $Y_{\eps,k}:=\eps(k+Y)$,
$K_\eps:=\bigcup_{k\in I_\eps^{\rm int}}Y_{\eps,k}$ and
$\bar\rho:=\int_\Gam\rho(y)\,dS_y$. For $Z\in H^1(\Om)$ and
$k\in I_\eps^{\rm int}$ let $m_k(Z)$ be the mean value of $Z$ over
$Y_{\eps,k}$, and let $M_\eps Z$ be the function on $K_\eps$ equal to
$m_k(Z)$ on each $Y_{\eps,k}$. The trace and Poincar\'e--Wirtinger
inequalities in $H^1(Y)$ give
$\|g-\bar g\|_{L^2(\Gam)}+\|g-\bar g\|_{L^2(Y)}\le C\|\nabla_yg\|_{L^2(Y)}$,
$\bar g:=\int_Yg\,dy$; the change of variables $x=\eps(k+y)$ yields
\begin{equation}
 \eps\int_{\eps(k+\Gam)}|Z-m_k(Z)|^2\,dS_x
 +\int_{Y_{\eps,k}}|Z-m_k(Z)|^2\,dx
 \le C\eps^2\int_{Y_{\eps,k}}|\nabla Z|^2\,dx.
 \label{eq:prelim-cell-average}
\end{equation}
Let $(Z_\eps)$ be bounded in $H_0^1(\Om)$ with $Z_\eps\to z$ strongly in
$L^2(\Om)$, and write $m_k=m_k(Z_\eps)$. Since
$\eps\int_{\eps(k+\Gam)}\rho^\eps\,dS_x=\eps^N\bar\rho=\bar\rho\,|Y_{\eps,k}|$,
\begin{equation*}
 \eps\int_{\GamE}\rho^\eps Z_\eps^2\,dS_x
 =\bar\rho\int_{K_\eps}(M_\eps Z_\eps)^2\,dx+R_\eps,\qquad
 R_\eps:=\sum_{k\in I_\eps^{\rm int}}\eps\int_{\eps(k+\Gam)}
 \rho^\eps(Z_\eps-m_k)(Z_\eps+m_k)\,dS_x.
\end{equation*}
By \eqref{eq:prelim-cell-average} and
$\eps\int_{\eps(k+\Gam)}m_k^2\,dS_x=|\Gam|\,\eps^Nm_k^2
\le|\Gam|\int_{Y_{\eps,k}}Z_\eps^2\,dx$, the Cauchy--Schwarz inequality gives
\[
 |R_\eps|\le C\|\rho\|_{L^\infty(\Gam)}\,
 \eps\|\nabla Z_\eps\|_{L^2(\Om)}\|Z_\eps\|_{H^1(\Om)}
 \le C\eps.
\]
Moreover, \eqref{eq:prelim-cell-average} gives
$\|M_\eps Z_\eps-Z_\eps\|_{L^2(K_\eps)}\le C\eps\|\nabla Z_\eps\|_{L^2(\Om)}\to0$,
while $|\Om\setminus K_\eps|\to0$ by \eqref{eq:omitted-cells-strip} and
$Z_\eps^2\to z^2$ in $L^1(\Om)$ imply
$\int_{\Om\setminus K_\eps}Z_\eps^2\,dx\to0$. Hence
$\int_{K_\eps}(M_\eps Z_\eps)^2\,dx\to\int_\Om z^2\,dx$, and
\begin{equation}
 \eps\int_{\GamE}\rho^\eps Z_\eps^2\,dS_x
 \longrightarrow
 \left(\int_\Gam\rho(y)\,dS_y\right)\int_\Om z^2\,dx.
 \label{eq:prelim-surface-diagonal-average}
\end{equation}
No sign condition on $\rho$ or $\bar\rho$ is used.
Apply \eqref{eq:prelim-surface-diagonal-average} to
$Z_\eps=U_\eps+W_\eps$ and $Z_\eps=U_\eps-W_\eps$. Since the traces of
$U_\eps$ and $W_\eps$ on $\GamE$ are $u_\eps$ and $v_\eps$, respectively,
polarization gives
\begin{equation*}
 \eps\int_{\GamE}\rho^\eps u_\eps v_\eps\,dS_x
 =\frac14\eps\int_{\GamE}\rho^\eps
 \bigl[(U_\eps+W_\eps)^2-(U_\eps-W_\eps)^2\bigr] \,dS_x
 \longrightarrow
 \left(\int_\Gam\rho(y)\,dS_y\right)\int_\Om uv\,dx,
\end{equation*}
which proves \eqref{eq:prelim-surface-mass-limit}. Finally,
\eqref{eq:prelim-combined-mass-limit} follows by adding
\eqref{eq:prelim-bulk-mass-limit} and \eqref{eq:prelim-surface-mass-limit}.
\end{proof}

\subsection{Uniform preliminary estimates and variational setting}
\label{subsec:uniform-preliminaries}

The cut geometry has two features that must be kept in mind: $\OmE$ need not
be Lipschitz near $\partial\Om$ (holes may meet $\partial\Om$ tangentially),
and it need not be connected (a hole close to $\partial\Om$ may cut off a
small piece of solid). Neither feature is used anywhere: all compactness and
coercivity properties of $V_\eps$ are obtained through the extension
$P_\eps$ of Lemma~\ref{lem:cut-uniform-extension}, which requires neither.
In particular, \eqref{eq:common-cut-extension-gradient} and the Poincar\'e
inequality in $H_0^1(\Om)$ yield the uniform Poincar\'e inequality
\begin{equation}
 \norm{v}_{L^2(\OmE)}\le C\norm{v}_{V_\eps}.
 \label{eq:Poincare}
\end{equation}
The rescaled trace estimate, obtained cell by cell on the complete cells,
holds for every $v\in H^1(\OmE)$:
\begin{equation}
 \eps\norm{v}_{L^2(\GamE)}^2
 \le C\left(\norm{v}_{L^2(\OmE)}^2
 +\eps^2\norm{\nabla v}_{L^2(\OmE)}^2\right).
 \label{eq:scaled-trace}
\end{equation}
Hence
\begin{equation}
 \norm{v}_{L^2(\OmE)}
 +\eps^{1/2}\norm{v}_{L^2(\GamE)}
 \le C\norm{v}_{V_\eps},
 \qquad v\in V_\eps.
 \label{eq:mass-energy-bound}
\end{equation}
With $d(x)=\operatorname{dist}(x,\partial\Om)$, Hardy's inequality applied to $P_\eps v\in H_0^1(\Om)$ and restricted to $\OmE$ gives
\begin{equation}
 \left\|\frac{v}{d}\right\|_{L^2(\OmE)}
 \le C\norm{v}_{V_\eps},\qquad v\in V_\eps.
 \notag
\end{equation}
We first record the needed boundary-strip estimates.

\begin{lemma}
\label{lem:strips}
Let $\Om$ be a bounded Lipschitz domain. There exist $t_0>0$ and $c_1,C>0$
such that, for $0<t<t_0$,
\begin{align}
 \norm{f}_{L^2(\{d<t\})}^2
 &\le Ct\,\norm{f}_{L^2(\Om)}\norm{f}_{H^1(\Om)},
 &&f\in H^1(\Om),
 \label{eq:strip-trace-type}\\
 \norm{u}_{L^2(\{d<t\})}
 &\le Ct\,\norm{\nabla u}_{L^2(\{d<c_1t\})},
 &&u\in H_0^1(\Om).
 \label{eq:strip-hardy-local}
\end{align}
Consequently, for fixed $L>0$ and $\eps$ small, every
$u\in H^2(\Om)\cap H_0^1(\Om)$ satisfies
\begin{equation}
 \norm{\nabla u}_{L^2(\{d<L\eps\})}
 \le C_L\eps^{1/2}\norm{u}_{H^2(\Om)},
 \label{eq:strip-grad-u}
\end{equation}
\begin{equation}
 \norm{u}_{L^2(\{d<L\eps\})}
 \le C_L\eps\norm{\nabla u}_{L^2(\{d<c_1L\eps\})}
 \le C_L\eps^{3/2}\norm{u}_{H^2(\Om)}.
 \label{eq:strip-L2-u}
\end{equation}
\end{lemma}

\begin{proof}
Cover $\partial\Om$ by finitely many Lipschitz graph cylinders.  Since the
cover is finite, it is enough to work in one of them.  After a rigid motion
we may write, for some fixed $s_0>0$,
\[
 \Om\cap Q
 =\{(x',g(x')+s):x'\in Q',\ 0<s<s_0\},
\]
where $g$ is Lipschitz with constant at most $L_g$.  Shrinking $t_0$ if
necessary, there is $c_1\ge1$, depending only on $L_g$, such that
\[
 \{d<t\}\cap Q
 \subset E_t:=\{(x',g(x')+s):x'\in Q',\ 0<s<c_1t\}
 \subset\{d<c_1t\}\cap Q
\]
for $0<t<t_0$.  We may take, for instance,
$c_1=\sqrt{1+L_g^2}$.

Fix $x'$ and set
\[
 h(s):=f(x',g(x')+s),\qquad 0<s<s_0.
\]
For $0<s,\sigma<s_0$, the fundamental theorem of calculus gives
\[
 |h(s)|^2
 \le |h(\sigma)|^2
     +2\int_0^{s_0}|h(r)|\,|h'(r)|\,dr .
\]
Averaging with respect to $\sigma\in(0,s_0)$ yields
\[
 |h(s)|^2
 \le \frac1{s_0}\int_0^{s_0}|h(r)|^2\,dr
     +2\int_0^{s_0}|h(r)|\,|h'(r)|\,dr .
\]
Integrating this inequality for $0<s<c_1t$ gives
\[
 \int_0^{c_1t}|h(s)|^2\,ds
 \le Ct\left(
       \int_0^{s_0}|h|^2\,ds
       +\int_0^{s_0}|h|\,|h'|\,ds
      \right).
\]
Since the change of variables
$(x',s)\mapsto(x',g(x')+s)$ has unit Jacobian, integration in $x'$ and
Cauchy--Schwarz imply
\[
 \|f\|_{L^2(E_t)}^2
 \le Ct\left(
       \|f\|_{L^2(\Om\cap Q)}^2
       +\|f\|_{L^2(\Om\cap Q)}
        \|\nabla f\|_{L^2(\Om\cap Q)}
      \right) \le Ct\,\|f\|_{L^2(\Om\cap Q)}
          \|f\|_{H^1(\Om\cap Q)} .
\]
Summing over the finitely many boundary charts proves
\eqref{eq:strip-trace-type}.

Now let $u\in H_0^1(\Om)$.  Its trace vanishes on the graph, so for almost
every $x'$ the corresponding slice
$h(s)=u(x',g(x')+s)$ satisfies $h(0)=0$.  Hence
\[
 h(s)=\int_0^s h'(r)\,dr,
\]
and therefore, by Cauchy--Schwarz,
\[
 |h(s)|^2
 \le s\int_0^s|h'(r)|^2\,dr
 \le c_1t\int_0^{c_1t}|h'(r)|^2\,dr .
\]
Integrating once more in $s$ gives
\[
 \int_0^{c_1t}|h(s)|^2\,ds
 \le (c_1t)^2\int_0^{c_1t}|h'(s)|^2\,ds .
\]
After integration in $x'$ and summation over the boundary charts,
\[
 \|u\|_{L^2(\{d<t\})}^2
 \le Ct^2\|\nabla u\|_{L^2(\{d<c_1t\})}^2,
\]
which is \eqref{eq:strip-hardy-local} after taking square roots.

Finally, let $u\in H^2(\Om)\cap H_0^1(\Om)$ and put $t=L\eps$.  Applying
\eqref{eq:strip-trace-type} to each $\partial_j u$ and summing over
$j=1,\ldots,N$ yields
\[
 \|\nabla u\|_{L^2(\{d<L\eps\})}^2
 \le C_L\eps
   \sum_{j=1}^N
   \|\partial_j u\|_{L^2(\Om)}
   \|\partial_j u\|_{H^1(\Om)} \le C_L\eps\,\|u\|_{H^2(\Om)}^2.
\]
Taking square roots proves \eqref{eq:strip-grad-u}.  Using
\eqref{eq:strip-hardy-local} with $t=L\eps$ and then
\eqref{eq:strip-grad-u} with $L$ replaced by $c_1L$, we obtain
\[
 \|u\|_{L^2(\{d<L\eps\})}
 \le C_L\eps\,
      \|\nabla u\|_{L^2(\{d<c_1L\eps\})}
 \le C_L\eps^{3/2}\|u\|_{H^2(\Om)},
\]
which is \eqref{eq:strip-L2-u}.
\end{proof}

We next record the needed cell-regularity consequences.

\begin{lemma}
\label{lem:cell-regularity-H4}
Let $0<\alpha'<\alpha_Y$ and assume
$\Gamma\in C^{2,\alpha_Y}$. The following statements hold with
constants depending only on the fixed cell data. In \textup{(ii)} assume
\textup{(H1)}--\textup{(H3)}; in \textup{(iii)} assume \textup{(H1)}
and $\mathrm{(H3}_A\mathrm{)}$.
\begin{enumerate}[label=\textup{(\roman*)}]
 \item If $B\in C^{0,\alpha_Y}_\#(\overline{Y^*})^{N\times N}$ is uniformly
 elliptic and $\chi_B^j\in H^1_\#(Y^*)/\R$ solves the periodic
 conormal corrector problem, then
 $\chi_B^j\in C^{1,\alpha'}(\overline{Y^*})$. Under \textup{(H1)} and
 $\mathrm{(H3}_A\mathrm{)}$, this applies to the classical correctors
 $\chi^j$; after the ground-state transform it applies to
 $\widetilde\chi^j$ once $\widetilde A$ satisfies the same assumptions.
 \item Every weak solution of the periodic Robin cell equation
 \[
 -\Div_y(A\nabla_y\theta)=\lambda q\theta\quad\hbox{in }Y^*,
 \qquad
 A\nabla_y\theta\cdot n=\lambda\rho\theta\quad\hbox{on }\Gamma
 \]
 belongs to $C^{1,\alpha'}(\overline{Y^*})$.
 \item Let $g\in L^\infty_\#(Y^*)$ and
 $h\in C^{0,\alpha_Y}_\#(\Gamma)$ satisfy
 $\int_{Y^*}g\,dy+\int_\Gamma h\,dS_y=0$.  The mean-zero periodic solution $\zeta$ of
 \[
  \int_{Y^*}A\nabla_y\zeta\cdot\nabla_y\psi\,dy
  =\int_{Y^*}g\psi\,dy+\int_\Gamma h\psi\,dS_y
  \qquad\forall\psi\in H^1_\#(Y^*)
 \]
 belongs to $C^{1,\alpha'}(\overline{Y^*})$; consequently the zero extension
 of $A\nabla_y\zeta$ belongs to $L^\infty_\#(Y)^N$.  This generic statement
 will be applied to the critical compensator introduced in
 Section~\ref{sec:critical-regime}.
\end{enumerate}
\end{lemma}

\begin{proof}
Because $T\Subset Y$, periodicity reduces the issue to interior estimates and
local estimates near the compact boundary $\Gamma$.  We use the following
consequence of the weak conormal Schauder estimate
\cite[Theorem~5.54]{Lieberman2013}: if
\[
 -\Div(B\nabla w)=f+\Div G,
 \qquad (B\nabla w+G)\cdot n=h,
\]
with $B\in C^{0,\alpha}$ uniformly elliptic, $f\in L^p$, $p>N$,
$G\in C^{0,\alpha}$ and $h\in C^{0,\alpha}$, then locally up to the
boundary
\begin{equation}
 \|w\|_{C^{1,\beta}}
 \le C\bigl(\|w\|_{L^2}+\|f\|_{L^p}
 +\|G\|_{C^{0,\alpha}}+\|h\|_{C^{0,\alpha}}\bigr),
 \qquad \beta=\min\{\alpha,1-N/p\}.
 \label{eq:cell-conormal-Schauder-local}
\end{equation}
A finite covering gives the corresponding estimate on $\overline{Y^*}$.

For \textup{(i)}, the corrector equation can be written as
\[
 -\Div(B\nabla\chi_B^j)=\Div(Be_j),
 \qquad (B\nabla\chi_B^j+Be_j)\cdot n=0\quad\text{on }\Gamma.
\]
Thus \eqref{eq:cell-conormal-Schauder-local} applies with
$f=0$, $G=Be_j$, $h=0$.  Since $B\in C^{0,\alpha_Y}$, choosing
$p>N$ so that $1-N/p>\alpha'$ yields
$\chi_B^j\in C^{1,\alpha'}(\overline{Y^*})$.  This gives the assertions for
$\chi^j$ and, after the ground-state transform, for $\widetilde\chi^j$.

For \textup{(ii)}, let $m\ge2$ and test the weak Robin equation with
$|\theta|^{m-2}\theta$ (by truncation).  Ellipticity gives
\begin{align}
 \alpha_0(m-1)\int_{Y^*}|\theta|^{m-2}|\nabla\theta|^2\,dy
 &\le C_\lambda\left(
      \int_{Y^*}|\theta|^m\,dy+\int_\Gamma|\theta|^m\,dS_y\right).
 \label{eq:cell-robin-moser-1}
\end{align}
With $z=|\theta|^{m/2}$,
$|\nabla z|^2=(m^2/4)|\theta|^{m-2}|\nabla\theta|^2$, so
\begin{equation}
 \frac{4\alpha_0(m-1)}{m^2}\|\nabla z\|_2^2
 \le C_\lambda\bigl(\|z\|_2^2+\|z\|_{L^2(\Gamma)}^2\bigr).
 \label{eq:cell-robin-moser-2}
\end{equation}
Using the trace inequality
$\|z\|_{L^2(\Gamma)}^2\le\delta\|\nabla z\|_2^2+C_\delta\|z\|_2^2$
and choosing $\delta\asymp m^{-1}$, we absorb the gradient term and obtain
\begin{equation}
 \|\nabla z\|_2\le C_\lambda m\|z\|_2.
 \label{eq:cell-robin-moser-3}
\end{equation}
Let $\kappa=N/(N-2)$ if $N>2$, and take any $\kappa>1$ if $N=2$.
Sobolev embedding then gives
\begin{equation}
 \|\theta\|_{L^{m\kappa}(Y^*)}
 \le (C_\lambda m)^{2/m}\|\theta\|_{L^m(Y^*)}.
 \label{eq:cell-robin-moser-iteration}
\end{equation}
Iterating with $m_{j+1}=\kappa m_j$, $m_0=2$, yields
$\theta\in L^\infty(Y^*)$.

Hence
\[
 f_0:=\lambda q\theta\in L^\infty(Y^*),
 \qquad
 A\nabla\theta\cdot n+(-\lambda\rho)\theta=0
 \quad\text{on }\Gamma.
\]
Since the Robin coefficient $-\lambda\rho$ is bounded,
\cite[Theorem~3.14]{Nittka2011} gives
$\theta\in C^{0,\gamma}(\overline{Y^*})$ for some $\gamma>0$.  Therefore
$\lambda\rho\theta\in C^{0,\gamma_1}(\Gamma)$ with
$\gamma_1=\min\{\gamma,\alpha_Y\}$.  Applying
\eqref{eq:cell-conormal-Schauder-local} with
$f=f_0$, $G=0$, $h=\lambda\rho\theta$ first gives
$\theta\in C^{1,\beta_1}$ for some $\beta_1>0$.  In particular
$\theta|_\Gamma$ is Lipschitz, so
$\rho\theta\in C^{0,\alpha_Y}(\Gamma)$.  A second application, with
$p>N$ chosen so that $1-N/p>\alpha'$, gives
\[
 \theta\in C^{1,\alpha'}(\overline{Y^*}).
\]

Finally, in \textup{(iii)} the weak equation is exactly
\[
 -\Div(A\nabla\zeta)=g\quad\text{in }Y^*,
 \qquad A\nabla\zeta\cdot n=h\quad\text{on }\Gamma,
\]
with periodicity and zero mean.  Since $g\in L^\infty\subset L^p$ for every
finite $p$ and $h\in C^{0,\alpha_Y}(\Gamma)$,
\eqref{eq:cell-conormal-Schauder-local}, with $p>N$ and
$1-N/p>\alpha'$, gives
$\zeta\in C^{1,\alpha'}(\overline{Y^*})$.  Thus
$A\nabla\zeta\in L^\infty(Y^*)^N$, and its zero extension to $T$, followed
by periodic extension, belongs to $L^\infty_\#(Y)^N$.
\end{proof}

\subsection{A common outer-strip resolvent estimate}
\label{subsec:common-RBL}

Near a Dirichlet boundary, the first-order two-scale expansion only gives an
$O(\eps^{1/2})$ energy error, and it is classical that adding a
boundary-layer corrector restores $O(\eps)$ estimates, both for
boundary-value problems and for eigenvalues of fixed index
\cite{AllaireAmar1999,MoskowVogelius1997,OleinikShamaevYosifian1992}; see
also \cite{KenigLinShen2012,KenigLinShen2013,ShenBoundary2017,ShenMonograph2018,ShenPerforated2023}
and the operator estimates of
\cite{PakhninSuslina2013,Suslina2013,ZhikovPastukhova2016}. In the cut
geometry the boundary layer must also cancel the corrector on the holes that
meet $\partial\Om$. We isolate the only boundary estimate needed for this in
all three quantitative branches. Fix $\alpha\in(0,1)$ and let $B=B(y)$ be any real, symmetric,
$Y$-periodic matrix in $\mathcal W$ satisfying
\begin{equation}
 \alpha_B|\xi|^2\le B(y)\xi\cdot\xi\le\beta_B|\xi|^2,
 \qquad
 B\in C^{0,\alpha}_\#(\overline{Y^*})^{N\times N}.
 \label{eq:common-B-assumptions}
\end{equation}
For $j=1,\ldots,N$, let $\chi_B^j\in
H^1_\#(\Ys)/\R$ solve
\begin{equation}
 \int_{\Ys}B(y)(e_j+\nabla_y\chi_B^j)\cdot\nabla_y\psi\,dy=0
 \qquad\forall\psi\in H^1_\#(\Ys)/\R,
 \label{eq:common-B-cell}
\end{equation}
and define
\begin{equation}
 B^{\rm hom}_{ij}
 :=\int_{\Ys}
 B(y)(e_j+\nabla_y\chi_B^j)\cdot e_i\,dy.
 \label{eq:common-B-hom}
\end{equation}
We now fix an explicit whole-cell representative of each corrector.  Let
$\widehat\chi_B^j\in W^{1,\infty}_\#(Y)$ equal $\chi_B^j$ in $Y^*$ and,
inside $T$, let it be the harmonic extension of the trace
$\chi_B^j|_\Gamma$:
\begin{equation}
 \Delta_y\widehat\chi_B^j=0\quad\hbox{in }T,
 \qquad
 \widehat\chi_B^j=\chi_B^j\quad\hbox{on }\Gamma.
 \label{eq:common-corrector-extension}
\end{equation}
The additive constant is fixed by the zero-mean representative of
$\chi_B^j$ in $Y^*$.  The boundary regularity and the cell Schauder
estimate give the asserted $W^{1,\infty}$ regularity.  Only the restriction
to $Y^*$ enters the flux, since the zero extension $B_0=\ind_{\mathcal W} B$
vanishes in $T$; the harmonic choice in \eqref{eq:common-corrector-extension}
merely makes the notation intrinsic.
On $V_\eps$ set
\begin{equation}
 a_{\eps,B}(u,v)
 :=\int_{\OmE}B^\eps\nabla u\cdot\nabla v\,dx.
 \notag
\end{equation}

Throughout this subsection, $\mathcal F_\eps$ denotes the Friedrichs
smoothing operator at scale $\eps$, with a nonnegative kernel supported in
the ball of radius $\eps/2$ centered at the origin. We use the standard
estimates
\[
 \|g-\mathcal F_\eps g\|_{L^2(\R^N)}
 \le C\eps\|\nabla g\|_{L^2(\R^N)},
 \qquad
 \|\mathcal F_\eps g\|_{L^2(\R^N)}\le C\|g\|_{L^2(\R^N)},
 \qquad g\in H^1(\R^N);
\]
they apply in particular to every $g\in H_0^1(\Omega)$ extended by zero.

The residual estimates use the following periodic-potential representation.

\begin{lemma}
\label{lem:periodic-potentials}
\begin{enumerate}[label=\textup{(\alph*)}]
 \item Let the field $G\in L^\infty_\#(Y)^N$ have zero mean on $Y$ and
 satisfy $\Div_yG=0$ in $\mathcal D'_\#(Y)$. There exists an antisymmetric
 tensor $\Psi=(\Psi_{i\ell})$, with entries in $L^\infty_\#(Y)$, such that
 $G_i=\sum_{\ell=1}^N\partial_{y_\ell}\Psi_{i\ell}$, $1\le i\le N$, and $\Psi_{i\ell}=-\Psi_{\ell i}$.
 \item Let $h\in L^\infty_\#(Y)$ satisfy $\int_Yh\,dy=0$. There exists
 $H\in L^\infty_\#(Y)^N$ such that $h=\Div_yH$.
 \item In the situation of \textup{(a)}, for $w\in H^1(\Om)$ and
 $\psi\in H_0^1(\Om)$,
 \begin{equation}
  \Bigl|\int_\Om G^\eps\cdot\nabla\psi\;w\,dx\Bigr|
  \le\eps\|\Psi\|_\infty\|\nabla w\|_{L^2(\Om)}\|\nabla\psi\|_{L^2(\Om)}.
 \notag
 \end{equation}
 \item In the situation of \textup{(b)}, for $\varphi\in W_0^{1,1}(\Om)$,
 \begin{equation}
  \Bigl|\int_\Om h^\eps\varphi\,dx\Bigr|
  \le\eps\|H\|_\infty\|\nabla\varphi\|_{L^1(\Om)}.
 \notag
 \end{equation}
\end{enumerate}
The norms of $\Psi$ and $H$ are bounded by a constant depending only on $N$
and on $\|G\|_\infty$, respectively $\|h\|_\infty$.
\end{lemma}

\begin{proof}
For \textup{(a)}, solve $\Delta_y\Phi_i=G_i$ on the torus with
$\int_Y\Phi_i\,dy=0$ and set
$\Psi_{i\ell}=\partial_{y_\ell}\Phi_i-\partial_{y_i}\Phi_\ell$. Since
$\Delta_y(\Div_y\Phi)=\Div_yG=0$ and $\Div_y\Phi$ has zero mean,
$\Div_y\Phi=0$, whence $\sum_{\ell=1}^N\partial_{y_\ell}\Psi_{i\ell}=G_i$. For
\textup{(b)}, solve $\Delta_y\varphi=h$ with zero mean and set
$H=\nabla_y\varphi$. In both cases the Calder\'on--Zygmund estimates give
$\Phi,\varphi\in W^{2,p}_\#(Y)$ for every $p<\infty$, and Sobolev
embedding with $p>N$ gives the $L^\infty$ bounds.

For \textup{(c)}, let first $\psi\in C_c^\infty(\Om)$. Since
$G_i^\eps=\eps\sum_{\ell=1}^N\partial_{x_\ell}(\Psi_{i\ell})^\eps$ and $w\,\partial_i\psi$
has compact support,
\[
 \int_\Om G^\eps\cdot\nabla\psi\;w\,dx
 =\sum_{i=1}^N\int_\Om G_i^\eps\,w\,\partial_i\psi\,dx
 =-\eps\sum_{i=1}^N\sum_{\ell=1}^N\int_\Om(\Psi_{i\ell})^\eps
 \bigl(\partial_\ell w\,\partial_i\psi+w\,\partial_{\ell i}\psi\bigr)\,dx,
\]
and the double sum containing $\partial_{\ell i}\psi$ vanishes by antisymmetry of $\Psi$.
The resulting bound passes to $\psi\in H_0^1(\Om)$ by density. For
\textup{(d)}, $h^\eps=\eps\Div_x(H^\eps)$, and integration by parts against
$\varphi\in W_0^{1,1}(\Om)$ gives the claim.
\end{proof}

We next control the boundary cut-off error.

\begin{lemma}
\label{lem:common-cutoff-stability}
Assume $\Gamma\in C^{2,\alpha}$ and let $B$ satisfy
\eqref{eq:common-B-assumptions}. Let $M\ge1$ be fixed and let
$\kappa_{\eps,M}\in C_c^\infty(\Omega)$ satisfy
\[
 0\le\kappa_{\eps,M}\le1,\qquad
 \kappa_{\eps,M}=0\ \hbox{if }d<M\eps,\qquad
 \kappa_{\eps,M}=1\ \hbox{if }d>2M\eps,
 \qquad |\nabla\kappa_{\eps,M}|\le C_M\eps^{-1}.
\]
For $v\in H^2(\Omega)\cap H_0^1(\Omega)$ set
\[
 K_j=\mathcal F_\eps(\kappa_{\eps,M}\partial_jv),\quad 1\le j\le N,\qquad
 U_\eps^{(1)}=v+\eps\sum_{j=1}^N(\widehat\chi_B^j)^\eps K_j,
\]
and $K:=(K_1,\ldots,K_N)$. Then, for every $\psi\in H_0^1(\Omega)$,
\begin{equation}
 \left|\int_\Omega B^{\rm hom}\nabla v\cdot\nabla\psi\,dx
 -\int_\Omega B_0^\eps\nabla U_\eps^{(1)}\cdot\nabla\psi\,dx\right|
 \le C_{B,M}\left(\|\nabla v\|_{L^2(\{d<3M\eps\})}
 +\eps\|\nabla^2v\|_{L^2(\Omega)}\right)\|\nabla\psi\|_2.
 \label{eq:common-cutoff-stability}
\end{equation}
\end{lemma}

\begin{proof}
With the extension
\eqref{eq:common-corrector-extension}, set
$G_B^j:=B_0(e_j+\nabla_y\widehat\chi_B^j)-B^{\rm hom}e_j$. The cell
equation and the zero conormal flux on $\Gamma$ give $\Div_yG_B^j=0$ in
$\mathcal D'_\#(Y)$, and $\int_YG_B^j\,dy=0$ by \eqref{eq:common-B-hom}.
Expanding $\nabla U_\eps^{(1)}$ and adding and subtracting
$\sum_{j=1}^NB_0^\eps e_jK_j=B_0^\eps K$ gives
\begin{align}
 &\int_\Omega B^{\rm hom}\nabla v\cdot\nabla\psi\,dx
 -\int_\Omega B_0^\eps\nabla U_\eps^{(1)}\cdot\nabla\psi\,dx =\int_\Omega(B^{\rm hom}-B_0^\eps)
 (\nabla v-K)\cdot\nabla\psi\,dx
 -\sum_{j=1}^N\int_\Omega(G_B^j)^\eps\cdot\nabla\psi\;K_j\,dx\notag\\
 &\qquad\quad-\eps\sum_{j=1}^N\int_\Omega B_0^\eps(\widehat\chi_B^j)^\eps
 \nabla K_j\cdot\nabla\psi\,dx.\label{eq:common-I123-explicit}
\end{align}
Lemma~\ref{lem:periodic-potentials}\textup{(c)} with $w=K_j$, for each $j$, bounds the
second term by $C\eps\|\nabla K\|_2\|\nabla\psi\|_2$. The support
properties of $1-\kappa_{\eps,M}$, the Friedrichs estimate and
$|\nabla\kappa_{\eps,M}|\le C_M\eps^{-1}$ give
\begin{equation}
 \|\nabla v-K\|_{L^2(\Omega)}
 +\eps\|\nabla K\|_{L^2(\Omega)}
 \le C_M\left(
 \|\nabla v\|_{L^2(\{d<3M\eps\})}
 +\eps\|\nabla^2v\|_{L^2(\Omega)}\right),
 \notag
\end{equation}
which bounds the three terms of \eqref{eq:common-I123-explicit} and proves
\eqref{eq:common-cutoff-stability}.
\end{proof}

The strip source satisfies the following uniform resolvent estimate.

\begin{lemma}
\label{lem:common-RBL}
Let $L>0$ be fixed and let $B$ satisfy
\eqref{eq:common-B-assumptions}. For $f\in L^2(\OmE)$, let
$z_{\eps,B}\in V_\eps$ be the solution of
\begin{equation}
 \int_{\OmE}
 B^\eps\nabla z_{\eps,B}\cdot\nabla\varphi\,dx
 =\int_{\OmE}f\varphi\,dx
 \qquad\forall\varphi\in V_\eps.
 \label{eq:common-B-resolvent}
\end{equation}
Assume \textup{(H4)} and $\Gamma\in C^{2,\alpha}$. Then there exist
$\eps_{B,L}>0$ and $C_{B,L}>0$, independent of
$0<\eps<\eps_{B,L}$ and of $f$, such that
\begin{equation}
 \|\nabla z_{\eps,B}\|_{L^2(\OmE\cap\{d<L\eps\})}
 \le C_{B,L}\eps^{1/2}
 \|f\|_{L^2(\OmE)}.
 \label{eq:common-RBL}
\end{equation}
In particular, under \textup{(H1)} and \textup{(H3)} the lemma applies to
$B=A$. Under \textup{(H1)--(H3)} and $m_{\mathrm{eff}}>0$, it also applies
after the negative ground-state transform to
$B=\widetilde A=(\theta_1^-)^2A$.
\end{lemma}

\begin{proof}
We give the proof for a generic $B$. A function of $H_0^1(\Omega)$ need
not vanish on the portions of periodic holes which meet $\partial\Omega$,
so the degenerate fixed-domain equation used in the separated geometry is
not available. We instead compare the microscopic resolvent with a cut-off
first-order expansion and cancel the additional Dirichlet trace by a lifting
supported in an $O(\eps)$ strip.

We first compare with the homogenized resolvent and record the corresponding boundary-strip estimates.
Extend $f$ by zero from $\OmE$ to $\Omega$ and denote the
extension by $\widetilde f_\eps$. Let
$v_{\eps,B}\in H_0^1(\Omega)$ solve
\begin{equation}
 \int_\Omega B^{\rm hom}\nabla v_{\eps,B}\cdot\nabla\psi\,dx
 =\int_\Omega\widetilde f_\eps\psi\,dx
 \qquad\forall\psi\in H_0^1(\Omega).
 \label{eq:common-cut-hom-resolvent}
\end{equation}
The macroscopic regularity in \textup{(H4)} gives
\begin{equation}
 \|v_{\eps,B}\|_{H^2(\Omega)}
 \le C_B\|f\|_{L^2(\OmE)}.
 \label{eq:common-cut-H2}
\end{equation}
For every fixed $R>0$, the standard boundary-strip estimates for
$H^2(\Omega)\cap H_0^1(\Omega)$ yield
\begin{align}
 \|\nabla v_{\eps,B}\|_{L^2(\{d<R\eps\})}
 &\le C_{B,R}\eps^{1/2}
       \|f\|_{L^2(\OmE)},
 \label{eq:common-cut-strip-grad}\\
 \|v_{\eps,B}\|_{L^2(\{d<R\eps\})}
 &\le C_{B,R}\eps^{3/2}
       \|f\|_{L^2(\OmE)}.
 \label{eq:common-cut-strip-L2}
\end{align}
These are \eqref{eq:strip-grad-u}--\eqref{eq:strip-L2-u}.

We next choose the cut-off so that the oscillating corrector is switched off before the boundary cells.
Recall that every cell meeting $\partial\Omega$ is contained in
$\{d<C_Y\eps\}$, with $C_Y$ as in \eqref{eq:omitted-cells-strip}. Choose
\begin{equation}
 M_L>\max\{L+4,C_Y+4\}.
 \notag
\end{equation}
For sufficiently small $\eps$, let
$\kappa_{\eps,L}\in C_c^\infty(\Omega)$ satisfy
\begin{equation}
 0\le\kappa_{\eps,L}\le1,\qquad
 \kappa_{\eps,L}=0\ \hbox{in }\{d<M_L\eps\},\qquad
 \kappa_{\eps,L}=1\ \hbox{in }\{d>2M_L\eps\},\qquad
 |\nabla\kappa_{\eps,L}|\le C_L\eps^{-1},
 \notag
\end{equation}
and set
\begin{equation}
 K_{\eps,B,j}:=\mathcal F_\eps(\kappa_{\eps,L}\partial_jv_{\eps,B}),\quad 1\le j\le N,
 \qquad
 U_{\eps,B}^{(1)}:=v_{\eps,B}
   +\eps\sum_{j=1}^N(\widehat\chi_B^j)^\eps K_{\eps,B,j}.
 \notag
\end{equation}
Because the smoothing kernel is supported in a ball of radius $\eps/2$,
$K_{\eps,B,j}=0$ in $\{d<(M_L-1)\eps\}$. In particular the oscillating
corrector vanishes on every portion of a perforation meeting
$\partial\Omega$.

To enforce the additional Dirichlet condition on the cut perforations, we introduce a lifting supported in the boundary strip.
Choose $R_L$ with $C_Y+1<R_L<M_L-2$ and
$\vartheta_\eps\in W^{1,\infty}(\Omega)$ such that
\begin{equation}
 0\le\vartheta_\eps\le1,\qquad
 \vartheta_\eps=1\ \hbox{in }\{d<(C_Y+1)\eps\},\qquad
 \vartheta_\eps=0\ \hbox{in }\{d>R_L\eps\},\qquad
 |\nabla\vartheta_\eps|\le C_L\eps^{-1}.
 \notag
\end{equation}
Define on $\OmE$
\begin{equation}
 \zeta_{\eps,B}:=-\vartheta_\eps v_{\eps,B}.
 \notag
\end{equation}
On $\partial\Omega$ the zero trace of $v_{\eps,B}$ gives zero trace of
$U_{\eps,B}^{(1)}$. On a cut perforation one has
$\vartheta_\eps=1$ and $K_{\eps,B,j}=0$ for $1\le j\le N$, hence
$U_{\eps,B}^{(1)}+\zeta_{\eps,B}=0$ near every such piece. By the capacitary definition \eqref{eq:Veps}, this proves
\begin{equation}
 U_{\eps,B}:=U_{\eps,B}^{(1)}+\zeta_{\eps,B}
 \in V_\eps.
 \notag
\end{equation}
Using \eqref{eq:common-cut-strip-grad},
\eqref{eq:common-cut-strip-L2}, and
$|\nabla\vartheta_\eps|\le C_L\eps^{-1}$ gives
\begin{equation}
 \|\nabla\zeta_{\eps,B}\|_{L^2(\OmE)}
 \le C_{B,L}\eps^{1/2}
       \|f\|_{L^2(\OmE)}.
 \label{eq:common-cut-lifting-energy}
\end{equation}

We then estimate the residual of the truncated first-order expansion.
Lemma~\ref{lem:common-cutoff-stability}, applied with $M=M_L$ and
$v=v_{\eps,B}$, together with \eqref{eq:common-cut-H2} and
\eqref{eq:common-cut-strip-grad}, gives for every $\psi\in H_0^1(\Omega)$
\begin{equation}
 \left|\int_\Omega B^{\rm hom}\nabla v_{\eps,B}\cdot\nabla\psi\,dx
 -\int_\Omega B_0^\eps\nabla U_{\eps,B}^{(1)}\cdot\nabla\psi\,dx\right|
 \le C_{B,L}\eps^{1/2}\|f\|_2\|\nabla\psi\|_2,
 \label{eq:common-cut-flux-residual-rate}
\end{equation}
where $B_0=\ind_{\mathcal W} B$ and $\|f\|_2=\|f\|_{L^2(\OmE)}$.

The preceding estimates now yield a global $H^1$ bound for the corrector error.
Set
\begin{equation}
 W_{\eps,B}:=z_{\eps,B}-U_{\eps,B}
 \in V_\eps.
 \notag
\end{equation}
For arbitrary $\varphi\in V_\eps$, take
$\psi=P_\eps\varphi$ in
\eqref{eq:common-cut-hom-resolvent}. Since
$P_\eps\varphi=\varphi$ in
$\OmE$ and $\widetilde f_\eps=0$ outside this set,
\[
 \int_{\OmE}f\varphi\,dx
 =\int_\Omega B^{\rm hom}\nabla v_{\eps,B}\cdot
                  \nabla(P_\eps\varphi)\,dx.
\]
Moreover, since $\OmE=\Omega\cap\eps\mathcal W$,
\[
 \int_{\OmE}
 B^\eps\nabla U_{\eps,B}^{(1)}\cdot\nabla\varphi\,dx
 =\int_\Omega B_0^\eps\nabla U_{\eps,B}^{(1)}\cdot
                    \nabla(P_\eps\varphi)\,dx.
\]
Subtracting the variational equations, and then using
\eqref{eq:common-cut-flux-residual-rate},
\eqref{eq:common-cut-extension-gradient}, and
\eqref{eq:common-cut-lifting-energy}, yields
\begin{equation}
 \left|\int_{\OmE}
 B^\eps\nabla W_{\eps,B}\cdot\nabla\varphi\,dx\right|
 \le C_{B,L}\eps^{1/2}\|f\|_2
       \|\nabla\varphi\|_{L^2(\OmE)}.
 \notag
\end{equation}
Taking $\varphi=W_{\eps,B}$ and using ellipticity gives
\begin{equation}
 \|\nabla W_{\eps,B}\|_{L^2(\OmE)}
 \le C_{B,L}\eps^{1/2}\|f\|_2.
 \label{eq:common-cut-W-energy}
\end{equation}

Finally, we restrict the estimate to the target boundary strip.
By the choice of $M_L$ and the support of the smoothing kernel,
$K_{\eps,B,j}=0$ ($1\le j\le N$) on $\{d<L\eps\}$. Therefore in the target strip
\[
 z_{\eps,B}=W_{\eps,B}+v_{\eps,B}+\zeta_{\eps,B}.
\]
Combining \eqref{eq:common-cut-W-energy},
\eqref{eq:common-cut-strip-grad}, and
\eqref{eq:common-cut-lifting-energy} proves \eqref{eq:common-RBL}.
\end{proof}

This yields the exact Dirichlet layer.

\begin{lemma}
\label{lem:common-boundary-layer}
Assume \textup{(H4)} and $\Gamma\in C^{2,\alpha}$, and let $B$ satisfy
\eqref{eq:common-B-assumptions}, with correctors $\chi_B^j$ from
\eqref{eq:common-B-cell}. For $u\in H^2(\Omega)\cap H_0^1(\Omega)$ set
\begin{equation}
 R_{\eps,\mathrm{int}}^B u
 :=u+\eps\sum_{j=1}^N(\chi_B^j)^\eps\partial_j u
 \quad\hbox{in }\OmE.
 \notag
\end{equation}
Let $\beta_{\eps,B}[u]\in H^1(\OmE)$ be the unique
function satisfying
\begin{equation}
 \beta_{\eps,B}[u]+R_{\eps,\mathrm{int}}^B u
 \in V_\eps,
 \qquad
 \int_{\OmE}
 B^\eps\nabla\beta_{\eps,B}[u]\cdot\nabla v\,dx=0
 \quad\forall v\in V_\eps.
 \notag
\end{equation}
Thus the exact Dirichlet datum on all of $D_\eps$ is
\begin{equation}
 \beta_{\eps,B}[u]
 =-u-\eps\sum_{j=1}^N(\chi_B^j)^\eps\partial_j u
 \quad\hbox{on }D_\eps
 \notag
\end{equation}
in the sense of the capacitary definition \eqref{eq:Veps}; on the exterior
part $\partial\Omega$ this reduces to the familiar oscillatory datum because
$u=0$. Then, uniformly for sufficiently small $\eps$,
\begin{align}
 \|\nabla\beta_{\eps,B}[u]\|_{L^2(\OmE)}
 &\le C_B\eps^{1/2}\|u\|_{H^2(\Omega)},
 \label{eq:common-beta-energy}\\
 \|\beta_{\eps,B}[u]\|_{L^2(\OmE)}
 +\|d\nabla\beta_{\eps,B}[u]\|_{L^2(\OmE)}
 +\eps^{1/2}\|\beta_{\eps,B}[u]\|_{L^2(\GamE)}
 &\le C_B\eps\|u\|_{H^2(\Omega)}.
 \label{eq:common-BL}
\end{align}
In particular, under \textup{(H1)} and \textup{(H3)} the lemma applies to
$B=A$, and under \textup{(H1)--(H3)} and $m_{\mathrm{eff}}>0$ it applies to
$B=\widetilde A=(\theta_1^-)^2A$.
\end{lemma}

\begin{proof}
Every point of $D_\eps$ which does not belong to $\partial\Omega$ lies in
$\{d<C_Y\eps\}$, with $C_Y$ as in \eqref{eq:omitted-cells-strip}. Choose fixed numbers
$C_Y+1<R_0<R_1$ and a cut-off $\vartheta_\eps\in W^{1,\infty}(\Omega)$ such
that
\begin{equation}
 0\le\vartheta_\eps\le1,\qquad
 \vartheta_\eps=1\ \hbox{in }\{d<(C_Y+1)\eps\},\qquad
 \vartheta_\eps=0\ \hbox{in }\{d>R_0\eps\},\qquad
 |\nabla\vartheta_\eps|\le C\eps^{-1}.
 \notag
\end{equation}
Define the explicit lifting
\begin{equation}
 \zeta_{\eps,B}[u]
 :=-\vartheta_\eps R_{\eps,\mathrm{int}}^B u
 \quad\hbox{in }\OmE.
 \notag
\end{equation}
Since $\vartheta_\eps=1$ on the cut-hole part of $D_\eps$, while
$u=0$ on $\partial\Omega$ and the trace of the right-hand side is exactly
$-R_{\eps,\mathrm{int}}^B u$ there, one has
\begin{equation}
 \zeta_{\eps,B}[u]+R_{\eps,\mathrm{int}}^B u
 \in V_\eps.
 \notag
\end{equation}
The cell regularity encoded in \eqref{eq:common-B-assumptions} and
Lemma~\ref{lem:cell-regularity-H4} gives
$\chi_B^j,\nabla_y\chi_B^j\in L^\infty(Y^*)$. Using the strip estimates
\eqref{eq:strip-grad-u}--\eqref{eq:strip-L2-u} with $L=R_1$ and the identity
\[
 \nabla R_{\eps,\mathrm{int}}^B u
 =\nabla u
 +\sum_{j=1}^N(\nabla_y\chi_B^j)^\eps\partial_j u
 +\eps\sum_{j=1}^N(\chi_B^j)^\eps\nabla\partial_j u,
\]
we obtain
\begin{align}
 \|\nabla\zeta_{\eps,B}[u]\|_{L^2(\OmE)}
 &\le C_B\eps^{1/2}\|u\|_{H^2(\Omega)},
 \label{eq:common-zeta-energy}\\
 \|\zeta_{\eps,B}[u]\|_{L^2(\OmE)}
 +\|d\nabla\zeta_{\eps,B}[u]\|_{L^2(\OmE)}
 &\le C_B\eps^{3/2}\|u\|_{H^2(\Omega)}.
 \label{eq:common-zeta-bounds}
\end{align}
Set
\[
 \delta_{\eps,B}:=\beta_{\eps,B}[u]-\zeta_{\eps,B}[u]
 \in V_\eps.
\]
Then
\begin{equation}
 \int_{\OmE}
 B^\eps\nabla\delta_{\eps,B}\cdot\nabla v\,dx
 =-\int_{\OmE}
 B^\eps\nabla\zeta_{\eps,B}[u]\cdot\nabla v\,dx
 \qquad\forall v\in V_\eps.
 \label{eq:common-delta-equation}
\end{equation}
Taking $v=\delta_{\eps,B}$ and using ellipticity proves
\eqref{eq:common-beta-energy}.

For the improved $L^2$ estimate, let $z_{\eps,B}\in V_\eps$
solve \eqref{eq:common-B-resolvent} with
$f=\delta_{\eps,B}$. By symmetry and
\eqref{eq:common-delta-equation},
\[
 \|\delta_{\eps,B}\|_2^2
 =-\int_{\OmE}
 B^\eps\nabla\zeta_{\eps,B}[u]\cdot\nabla z_{\eps,B}\,dx.
\]
The support of $\nabla\zeta_{\eps,B}[u]$ is contained in
$\{d<R_1\eps\}$, so Lemma~\ref{lem:common-RBL} and
\eqref{eq:common-zeta-energy} give
\[
 \|\delta_{\eps,B}\|_2^2
 \le C_B
 \bigl(\eps^{1/2}\|u\|_{H^2}\bigr)
 \bigl(\eps^{1/2}\|\delta_{\eps,B}\|_2\bigr).
\]
Hence
\begin{equation}
 \|\delta_{\eps,B}\|_{L^2(\OmE)}
 \le C_B\eps\|u\|_{H^2(\Omega)}.
 \label{eq:common-delta-L2}
\end{equation}
Together with \eqref{eq:common-zeta-bounds} this yields the asserted
$L^2$ estimate for $\beta_{\eps,B}[u]$.

For the weighted gradient, $d$ is Lipschitz and
$d^2\delta_{\eps,B}\in V_\eps$. Testing
\eqref{eq:common-delta-equation} with this function and using
$|\nabla d|\le1$ a.e. gives, with
$X=\|d\nabla\delta_{\eps,B}\|_2$,
$X_1=\|\delta_{\eps,B}\|_2$, and
$X_2=\|\nabla\zeta_{\eps,B}[u]\|_2$,
\[
 \alpha_BX^2
 \le C_BX_1X+C_B\eps X_2X+C_B\eps X_2X_1,
\]
because $d\le R_1\eps$ on the support of
$\nabla\zeta_{\eps,B}[u]$. Young's inequality, together with
\eqref{eq:common-delta-L2} and \eqref{eq:common-zeta-energy}, gives
\begin{equation}
 \|d\nabla\delta_{\eps,B}\|_2
 \le C_B\eps\|u\|_{H^2(\Omega)}.
 \notag
\end{equation}
Combining this with \eqref{eq:common-zeta-bounds} proves the weighted part
of \eqref{eq:common-BL}.

Finally, by the definition of $I_\eps^{\rm int}$, every component
$\eps(k+\Gamma)$ of $\GamE$ is contained in a complete cell
$\eps(k+Y)\subset\Omega$. The rescaled cell trace inequality therefore
applies on each such cell with a constant depending only on $(Y,T)$.
Summing over the disjoint complete cells gives
\[
 \eps\|\beta_{\eps,B}[u]\|_{L^2(\GamE)}^2
 \le C\left(
 \|\beta_{\eps,B}[u]\|_{L^2(\OmE)}^2
 +\eps^2\|\nabla\beta_{\eps,B}[u]\|_{L^2(\OmE)}^2
 \right).
\]
Using \eqref{eq:common-beta-energy} and the $L^2$ estimate just proved
completes \eqref{eq:common-BL}.
\end{proof}

\begin{corollary}
\label{cor:common-boundary-layer-augmented}
Assume the hypotheses of Lemma~\ref{lem:common-boundary-layer}.  Let
$\chi_B^0\in W^{1,\infty}_\#(Y^*)$, let $|\Lambda|\le R$, and define
\[
 R_{\eps,\mathrm{int}}^{B,\Lambda}u
 :=u+\eps\sum_{j=1}^N(\chi_B^j)^\eps\partial_j u
   +\eps\Lambda(\chi_B^0)^\eps u.
\]
There is a unique $B^\eps$-harmonic layer
$\beta_{\eps,B}^{\Lambda}[u]$ such that
\[
 \beta_{\eps,B}^{\Lambda}[u]
 +R_{\eps,\mathrm{int}}^{B,\Lambda}u\in V_\eps,
\]
and, uniformly for $|\Lambda|\le R$,
\begin{align*}
 \|\nabla\beta_{\eps,B}^{\Lambda}[u]\|_{L^2(\OmE)}
 &\le C_{B,R}\eps^{1/2}\|u\|_{H^2(\Omega)},\\
 \|\beta_{\eps,B}^{\Lambda}[u]\|_{L^2(\OmE)}
 +\|d\nabla\beta_{\eps,B}^{\Lambda}[u]\|_{L^2(\OmE)}
 +\eps^{1/2}\|\beta_{\eps,B}^{\Lambda}[u]\|_{L^2(\GamE)}
 &\le C_{B,R}\eps\|u\|_{H^2(\Omega)}.
\end{align*}
\end{corollary}

\begin{proof}
By linearity,
\[
 \beta_{\eps,B}^{\Lambda}[u]
 =\beta_{\eps,B}[u]+\Lambda\gamma_\eps[u],
\]
where $\gamma_\eps[u]$ is $B^\eps$-harmonic and satisfies
\[
 \gamma_\eps[u]+\eps(\chi_B^0)^\eps u\in V_\eps.
\]
Set
\[
 \zeta_\eps^0:=-\vartheta_\eps\,\eps(\chi_B^0)^\eps u.
\]
Then $\zeta_\eps^0=-\eps(\chi_B^0)^\eps u$ on $D_\eps$,
$\operatorname{supp}\zeta_\eps^0\subset\{d<R_0\eps\}$, and
\begin{align*}
 \nabla\zeta_\eps^0
 &=-\eps(\chi_B^0)^\eps u\,\nabla\vartheta_\eps
   -\vartheta_\eps(\nabla_y\chi_B^0)^\eps u
   -\eps\vartheta_\eps(\chi_B^0)^\eps\nabla u.
\end{align*}
Using $\eps|\nabla\vartheta_\eps|\le C$ and
$\chi_B^0,\nabla_y\chi_B^0\in L^\infty$, we obtain
\begin{align*}
 \|\nabla\zeta_\eps^0\|_{L^2(\OmE)}
 &\le C\Bigl(
 \|u\|_{L^2(\{d<R_0\eps\})}
 +\eps\|\nabla u\|_{L^2(\{d<R_0\eps\})}
 \Bigr) \le C\Bigl(
 \eps^{3/2}\|u\|_{H^2(\Omega)}
 +\eps\,\eps^{1/2}\|u\|_{H^2(\Omega)}
 \Bigr)\\
 &\le C\eps^{3/2}\|u\|_{H^2(\Omega)},
\end{align*}
where Lemma~\ref{lem:strips} was used in the second line.
Similarly,
\[
 \|\zeta_\eps^0\|_{L^2(\OmE)}
 \le C\eps\|u\|_{L^2(\{d<R_0\eps\})}
 \le C\eps^{5/2}\|u\|_{H^2(\Omega)},\qquad \|d\nabla\zeta_\eps^0\|_{L^2(\OmE)}
 \le R_0\eps\|\nabla\zeta_\eps^0\|_{L^2(\OmE)}
 \le C\eps^{5/2}\|u\|_{H^2(\Omega)}.
\]

Let $\delta_\eps^0:=\gamma_\eps[u]-\zeta_\eps^0\in V_\eps$. Since
$\gamma_\eps[u]$ is $B^\eps$-harmonic,
\[
 \int_{\OmE}B^\eps\nabla\delta_\eps^0\cdot\nabla v\,dx
 =-\int_{\OmE}B^\eps\nabla\zeta_\eps^0\cdot\nabla v\,dx
 \qquad\forall v\in V_\eps.
\]
Taking $v=\delta_\eps^0$ and using ellipticity gives
\begin{align*}
 \alpha_B\|\nabla\delta_\eps^0\|_{L^2(\OmE)}^2
 &\le \beta_B
 \|\nabla\zeta_\eps^0\|_{L^2(\OmE)}
 \|\nabla\delta_\eps^0\|_{L^2(\OmE)},
\end{align*}
so that
\[
 \|\nabla\delta_\eps^0\|_{L^2(\OmE)}
 \le C\|\nabla\zeta_\eps^0\|_{L^2(\OmE)}
 \le C\eps^{3/2}\|u\|_{H^2(\Omega)}.
\]
Hence
\[
 \|\nabla\gamma_\eps[u]\|_{L^2(\OmE)}
 \le
 \|\nabla\delta_\eps^0\|_{L^2(\OmE)}
 +\|\nabla\zeta_\eps^0\|_{L^2(\OmE)} \le C\eps^{3/2}\|u\|_{H^2(\Omega)}.
\]
For the $L^2$ estimate, let $z_\eps^0\in V_\eps$ solve
\eqref{eq:common-B-resolvent} with $f=\delta_\eps^0$. By symmetry and the
preceding variational identity,
\[
 \|\delta_\eps^0\|_{L^2(\OmE)}^2
 =-\int_{\OmE}B^\eps\nabla\zeta_\eps^0\cdot\nabla z_\eps^0\,dx \le \beta_B
 \|\nabla\zeta_\eps^0\|_{L^2(\{d<R_0\eps\})}
 \|\nabla z_\eps^0\|_{L^2(\{d<R_0\eps\})}.
\]
Lemma~\ref{lem:common-RBL} gives
\[
 \|\nabla z_\eps^0\|_{L^2(\{d<R_0\eps\})}
 \le C\eps^{1/2}\|\delta_\eps^0\|_{L^2(\OmE)},
\]
and therefore
\[
 \|\delta_\eps^0\|_{L^2(\OmE)}^2
 \le C\eps^{3/2}\|u\|_{H^2}
          \eps^{1/2}\|\delta_\eps^0\|_{L^2(\OmE)} \le C\eps^2\|u\|_{H^2}
          \|\delta_\eps^0\|_{L^2(\OmE)}.
\]
Hence
\[
 \|\delta_\eps^0\|_{L^2(\OmE)}
 \le C\eps^2\|u\|_{H^2(\Omega)},
\]
and thus
\[
 \|\gamma_\eps[u]\|_{L^2(\OmE)}
 \le \|\delta_\eps^0\|_{L^2(\OmE)}
      +\|\zeta_\eps^0\|_{L^2(\OmE)} \le C(\eps^2+\eps^{5/2})\|u\|_{H^2}
 \le C\eps^2\|u\|_{H^2}.
\]

For the weighted gradient, test the variational identity with
$d^2\delta_\eps^0\in V_\eps$. Since $|\nabla d|\le1$ a.e. and
$d\le R_0\eps$ on $\operatorname{supp}\nabla\zeta_\eps^0$, setting
\[
 X:=\|d\nabla\delta_\eps^0\|_2,
 \qquad X_1:=\|\delta_\eps^0\|_2,
 \qquad X_2:=\|\nabla\zeta_\eps^0\|_2,
\]
gives
\[
 \alpha_BX^2
 \le C_BX_1X+C_B\eps X_2X+C_B\eps X_2X_1.
\]
Using
\[
 X_1\le C\eps^2\|u\|_{H^2},
 \qquad
 X_2\le C\eps^{3/2}\|u\|_{H^2},
\]
and Young's inequality,
\begin{align*}
 X^2
 &\le \frac12X^2
   +C\bigl(\eps^4+\eps^5+\eps^{9/2}\bigr)\|u\|_{H^2}^2,
\end{align*}
so
\[
 \|d\nabla\delta_\eps^0\|_2
 \le C\eps^2\|u\|_{H^2}.
\]
Consequently,
\[
 \|d\nabla\gamma_\eps[u]\|_2
 \le \|d\nabla\delta_\eps^0\|_2
      +\|d\nabla\zeta_\eps^0\|_2 \le C(\eps^2+\eps^{5/2})\|u\|_{H^2}
 \le C\eps^2\|u\|_{H^2}.
\]

Finally, the rescaled trace inequality on the complete cells yields
\[
 \eps\|\gamma_\eps[u]\|_{L^2(\GamE)}^2
 \le C\left(
 \|\gamma_\eps[u]\|_{L^2(\OmE)}^2
 +\eps^2\|\nabla\gamma_\eps[u]\|_{L^2(\OmE)}^2
 \right) \le C\left(\eps^4+\eps^5\right)\|u\|_{H^2}^2,
\]
whence
\[
 \eps^{1/2}\|\gamma_\eps[u]\|_{L^2(\GamE)}
 \le C\eps^2\|u\|_{H^2(\Omega)}.
\]
Finally, since $|\Lambda|\le R$,
\begin{align*}
 \|\nabla\beta_{\eps,B}^{\Lambda}[u]\|_{L^2}
 &\le \|\nabla\beta_{\eps,B}[u]\|_{L^2}
 +R\|\nabla\gamma_\eps[u]\|_{L^2}
 \le C_{B,R}\eps^{1/2}\|u\|_{H^2},\\
 \|\beta_{\eps,B}^{\Lambda}[u]\|_{L^2}
 +\|d\nabla\beta_{\eps,B}^{\Lambda}[u]\|_{L^2}
 +\eps^{1/2}\|\beta_{\eps,B}^{\Lambda}[u]\|_{L^2(\GamE)}
 &\le C_{B,R}(\eps+\eps^2)\|u\|_{H^2}
 \le C_{B,R}\eps\|u\|_{H^2}.
\end{align*}
\end{proof}

We also need compatible fixed-domain extensions.

\begin{lemma}
\label{lem:common-reconstructor-extension}
Assume the hypotheses of Lemma~\ref{lem:common-boundary-layer}.  Let
$u\in H^2(\Omega)\cap H_0^1(\Omega)$.  For $|\Lambda|\le R$, let
\[
 R_{\eps,\mathrm{int}}^{B,\Lambda}u
 =u+\eps\sum_{j=1}^N(\chi_B^j)^\eps\partial_j u
  +\eps\Lambda(\chi_B^0)^\eps u,
\]
where the last term may be omitted, and let
\[
 \mathcal R_{\eps,B}^{\Lambda}u
 :=R_{\eps,\mathrm{int}}^{B,\Lambda}u
   +\beta_{\eps,B}^{\Lambda}[u]\in V_\eps
\]
be the corresponding exact-layer reconstructor.  Then
\begin{equation}
 \|P_\eps(\mathcal R_{\eps,B}^{\Lambda}u)-u\|_{L^2(\Omega)}
 \le C_{B,R}\eps\|u\|_{H^2(\Omega)},
 \label{eq:common-reconstructor-extension}
\end{equation}
uniformly for $|\Lambda|\le R$.  For the non-augmented reconstructor the
same estimate holds with a constant independent of $R$.
\end{lemma}

\begin{proof}
Write $R_\eps=\mathcal R_{\eps,B}^{\Lambda}u$.  The corrector bounds and
Lemma~\ref{lem:common-boundary-layer}, or its augmented version, give
\begin{equation}
 \|R_\eps-u\|_{L^2(\OmE)}
 +\eps\|\nabla(R_\eps-u)\|_{L^2(\OmE)}
 \le C_{B,R}\eps\|u\|_{H^2(\Omega)}.
 \label{eq:common-reconstructor-solid-error}
\end{equation}

We estimate the extension separately on complete cells and on the boundary
cells.  Let $Y_{\eps,k}=\eps(k+Y)$ with
$k\in I_\eps^{\rm int}$.  Denote by $E_{\eps,k}$ the local cell extension
used in the construction of $P_\eps$.  Since it preserves constants,
scaled Poincar\'e and \eqref{eq:common-cut-extension-local-L2} imply
\begin{equation}
 \|E_{\eps,k}(R_\eps)-u\|_{L^2(Y_{\eps,k})}
 \le C\Bigl(
 \|R_\eps-u\|_{L^2(Y_{\eps,k}\cap\OmE)}
 +\eps\|\nabla(R_\eps-u)\|_{L^2(Y_{\eps,k}\cap\OmE)}\Bigr) +C\eps\|\nabla u\|_{L^2(Y_{\eps,k})}.
 \notag
\end{equation}
Indeed, the last term is obtained by applying the constant-preserving local
extension to $u|_{Y_{\eps,k}\cap\OmE}$ and comparing it with $u$ on the
whole cell modulo the cell average.  Squaring, summing over the disjoint
complete cells, and using \eqref{eq:common-reconstructor-solid-error} gives
an $O(\eps)\|u\|_{H^2}$ bound on their union.

All cells not in $I_\eps^{\rm int}$ which meet $\Omega$ are contained in a
strip $S_{C_Y\eps}:=\{d<C_Y\eps\}$, with $C_Y$ as in \eqref{eq:omitted-cells-strip}.
On their physical solid part, $P_\eps R_\eps=R_\eps$; inside their cut
Dirichlet perforation pieces the chosen continuation is zero.  Hence
\[
 \|P_\eps R_\eps-u\|_{L^2(S_{C_Y\eps})}
 \le \|R_\eps-u\|_{L^2(\OmE\cap S_{C_Y\eps})}
      +\|u\|_{L^2(S_{C_Y\eps})} \le C_{B,R}\eps\|u\|_{H^2(\Omega)},
\]
where the last inequality uses \eqref{eq:common-reconstructor-solid-error}
and Hardy's inequality for $u\in H_0^1(\Omega)$.  Combining the two regions
proves \eqref{eq:common-reconstructor-extension}.
\end{proof}

For later use we introduce the energy and mass forms
\begin{align}
 a_\eps(u,v)
 &:=\int_{\OmE}A^\eps\nabla u\cdot\nabla v\,dx,
 \label{eq:ae}\\
 b_\eps(u,v)
 &:=\int_{\OmE}q^\eps uv\,dx
 +\eps\int_{\GamE}\rho^\eps uv\,dS_x,
 \label{eq:be}
\end{align}
for $u,v\in V_\eps$.
The variational form of the microscopic problem is
\begin{equation}
 a_\eps(u_\eps,v)=\lambda_\eps b_\eps(u_\eps,v)
 \qquad\forall v\in V_\eps.
 \label{eq:spectral-micro}
\end{equation}
The form $a_\eps$ is coercive on $V_\eps$, while \eqref{eq:Poincare} and \eqref{eq:scaled-trace} give
\begin{equation}
 |b_\eps(u,v)|\le C\norm{u}_{V_\eps}\norm{v}_{V_\eps}.
 \label{eq:b-bounded}
\end{equation}
For each fixed $\eps$, the operator $\Keps$ defined below is compact. Indeed,
$v\mapsto v|_{\OmE}$ is compact from $V_\eps$ into $L^2(\OmE)$, since it
factors through $P_\eps:V_\eps\to H_0^1(\Om)$ and the Rellich embedding
$H_0^1(\Om)\hookrightarrow L^2(\Om)$; and the trace $V_\eps\to L^2(\GamE)$
is compact because $\GamE$ is the union of finitely many surfaces
$\eps(k+\Gam)$, $k\in I_\eps^{\rm int}$, each of which is the boundary of the
hole of a complete cell, so that the trace can be taken in the Lipschitz
domain $\eps(k+\Ys)$, where it is compact. No regularity of $\OmE$ near
$\partial\Om$ is involved.

\subsection{Spectral setting and two-sided microscopic spectrum}
\label{subsec:prelim-spectral-existence}

Equip $V_\eps$ with the scalar product $a_\eps$. By the Riesz representation theorem, define
\begin{equation}
 a_\eps(\Keps f,v)=b_\eps(f,v),
 \qquad f,v\in V_\eps.
 \notag
\end{equation}
The operator $\Keps$ is compact and self-adjoint in $(V_\eps,a_\eps)$. If $\tau_\eps\ne0$ is an eigenvalue of $\Keps$, then
\begin{equation}
 \lambda_\eps=\tau_\eps^{-1}
 \label{eq:prelim-lambda-tau}
\end{equation}
is an eigenvalue of \eqref{eq:spectral-micro}, with the same eigenspace and multiplicity.

Recall the cell mass form $\mathfrak b_Y$ from \eqref{eq:intro-cell-mass-form}. We first verify the two-sided cell spectrum.

\begin{proposition}
\label{prop:prelim-two-sign-spectrum}
Assume \textup{(H2)}. Then $\mathfrak b_Y$ is indefinite: there exist
$\psi_+,\psi_-\in H^1_\#(\Ys)$ such that
\begin{equation}
 \mathfrak b_Y(\psi_+,\psi_+)>0,
 \qquad
 \mathfrak b_Y(\psi_-,\psi_-)<0.
 \notag
\end{equation}
If, in addition, \textup{(H1)} holds, then for every
$\eps\in(0,1)$ such that $I_\eps^{\rm int}\ne\varnothing$, the spectrum of
\eqref{eq:spectral-micro} consists of two infinite real sequences, counted with
multiplicity,
\begin{equation}
 0<\lambda_{\eps,1}^+\le\lambda_{\eps,2}^+\le\cdots\longrightarrow+\infty,
 \qquad
 0>\lambda_{\eps,1}^-\ge\lambda_{\eps,2}^-\ge\cdots\longrightarrow-\infty.
 \label{eq:double-spectrum}
\end{equation}
If $B(x_0,r_\Om)\subset\Om$, then $I_\eps^{\rm int}\ne\varnothing$ for
\begin{equation}
 0<\eps<\eps_\Om:=r_\Om/\sqrt N,
 \label{eq:eps-Omega}
\end{equation}
since every point of the closed cell $\eps(k+\overline Y)$ containing $x_0$
lies at distance at most $\eps\sqrt N<r_\Om$ from $x_0$.
\end{proposition}

\begin{proof}
We prove a slightly stronger statement that is convenient for the microscopic
problem: for every $m\ge1$ and each sign $\sigma\in\{-1,+1\}$ there are
functions
\[
 \psi_{\sigma,1},\ldots,\psi_{\sigma,m}\in H^1_\#(\Ys)
\]
with pairwise disjoint supports within the fixed sign family, all supported
away from $\partial Y$, such that
\begin{equation}
 \sigma\,\mathfrak b_Y(\psi_{\sigma,\ell},
 \psi_{\sigma,\ell})>0,
 \qquad \ell=1,\ldots,m.
 \label{eq:cell-many-sign-directions}
\end{equation}

Fix $\sigma\in\{-1,+1\}$. By \textup{(H2)}, at least one of the sets
\[
 E_\sigma^q:=\{\sigma q>0\}\subset\Ys,
 \qquad
 E_\sigma^\Gamma:=\{\sigma\rho>0\}\subset\Gam
\]
has positive measure in its natural dimension.

If $|E_\sigma^q|>0$, there exists $\delta_\sigma>0$ such that
$\{\sigma q\ge\delta_\sigma\}$ has positive measure. Since Lebesgue
measure is nonatomic, split this set into $m$ pairwise disjoint measurable
pieces of positive measure. By inner regularity choose pairwise disjoint
compact subsets $K_{\sigma,\ell}\Subset\Ys$ of positive measure. Choose
pairwise disjoint open neighborhoods compactly contained in $\Ys$ and
approximate the corresponding characteristic functions in $L^2$ by
functions $\psi_{\sigma,\ell}\in C_c^\infty(\Ys)$. Their traces on
$\Gam$ vanish and, for sufficiently accurate approximations,
\[
 \sigma\,\mathfrak b_Y(\psi_{\sigma,\ell},
 \psi_{\sigma,\ell})
 =\sigma\int_{\Ys} q(y)\,|\psi_{\sigma,\ell}(y)|^2\,dy>0.
\]

If instead $\mathcal H^{N-1}(E_\sigma^\Gamma)>0$, choose
$\delta_\sigma>0$ such that
$\{\sigma\rho\ge\delta_\sigma\}$ has positive surface measure. Surface
measure on the Lipschitz manifold $\Gam$ is nonatomic, so we may split this
set into $m$ pairwise disjoint measurable pieces of positive surface measure
and, by inner regularity, choose compact subsets
$K_{\sigma,\ell}^\Gamma$ of positive surface measure. Surround them by
pairwise disjoint relatively open sets $U_{\sigma,\ell}\subset\Gam$.
Because traces on the Lipschitz boundary $\Gam$ of ambient
$C_c^\infty$ functions are dense in $L^2(\Gam)$, and the approximation can
be localized in the relatively open set $U_{\sigma,\ell}$, we may choose
$h_{\sigma,\ell}\in H^{1/2}(\Gam)$ as such a smooth ambient trace, with
$\operatorname{supp}h_{\sigma,\ell}\Subset U_{\sigma,\ell}$, so that
\begin{equation}
 \sigma\int_\Gam\rho(y)\,|h_{\sigma,\ell}(y)|^2\,dS_y>0.
 \label{eq:surface-sign-traces}
\end{equation}
Let $E_\Gamma:H^{1/2}(\Gam)\to H^1(\Ys)$ be a bounded right inverse of the
trace map. Because $\overline T\Subset Y$, choose pairwise disjoint ambient
neighborhoods of the $U_{\sigma,\ell}$ whose closures stay away from
$\partial Y$, together with Lipschitz cutoffs $\xi_{\sigma,\ell}$ supported
in these neighborhoods and equal to $1$ near the supports of the
corresponding traces. Let
$d_\Gamma(y)=\operatorname{dist}(y,\Gam)$ and fix
$\eta\in W^{1,\infty}([0,\infty))$ with $\eta(0)=1$ and
$\eta(t)=0$ for $t\ge1$. Set
\[
 \psi_{\sigma,\ell}^{\delta}(y)
 :=\xi_{\sigma,\ell}(y)E_\Gamma h_{\sigma,\ell}(y)
   \eta\!\left(\frac{d_\Gamma(y)}{\delta}\right).
\]
For every $\delta>0$ this belongs to $H^1_\#(\Ys)$, is supported away from
$\partial Y$, and has trace $h_{\sigma,\ell}$ on $\Gam$. Moreover,
\[
 \|\psi_{\sigma,\ell}^{\delta}\|_{L^2(\Ys)}\longrightarrow0
 \qquad(\delta\downarrow0),
\]
because its support is contained in a shrinking solid-side neighborhood of
$\Gam$. Since $q\in L^\infty(\Ys)$,
\[
 \int_{\Ys} q|\psi_{\sigma,\ell}^{\delta}|^2\,dy=o(1),
\]
whereas the surface term is the fixed quantity in
\eqref{eq:surface-sign-traces}. Hence, for sufficiently small $\delta$,
\eqref{eq:cell-many-sign-directions} holds. The ambient neighborhoods are
disjoint, so the supports are pairwise disjoint within this sign family. We
fix such a $\delta$ and drop it from the notation.

Applying this construction separately to $\sigma=+1$ and $\sigma=-1$
proves the indefiniteness of $\mathfrak b_Y$. Notice that the two signs may
come from different components of the combined measure. In particular, the
argument includes both limiting cases $\rho\equiv0$ with sign-changing $q$
and $q\equiv0$ with sign-changing $\rho$, as well as the case where the two
signs are carried by different components of the combined measure.

We now pass to the microscopic problem. Fix $m\ge1$ and $\eps$ with
$I_\eps^{\rm int}\ne\varnothing$, and choose $k\in I_\eps^{\rm int}$. The
cell functions are supported in a compact subset of $Y$ disjoint from
$\partial Y$; the rescaled functions below therefore vanish near
$\partial(\eps(k+Y))$, hence near $\Dep$, and belong to $V_\eps$. For the cell functions constructed above define
\[
 z_{\eps,\sigma,\ell}(x)
 :=\begin{cases}
 \psi_{\sigma,\ell}(x/\eps-k),&x\in\eps(k+\Ys),\\
 0,&x\in\OmE\setminus\eps(k+\Ys).
 \end{cases}
\]
A change of variables gives
\begin{equation}
 b_\eps(z_{\eps,\sigma,\ell},z_{\eps,\sigma,\ell})
 =\eps^N\mathfrak b_Y(\psi_{\sigma,\ell},\psi_{\sigma,\ell}),
 \notag
\end{equation}
and disjoint supports within each sign family give vanishing cross terms.
Thus $b_\eps$ is positive definite on one $m$-dimensional span and negative
definite on the other. Since $m$ is arbitrary, the positive and negative
Morse indices of the compact self-adjoint operator $\Keps$ are both
infinite. The spectral theorem yields infinitely many eigenvalues of each
sign for $\Keps$, accumulating only at $0$; taking reciprocals gives
\eqref{eq:double-spectrum}.
\end{proof}

\begin{remark}[Large $\eps$]
\label{rem:large-eps}
The condition $I_\eps^{\rm int}\ne\varnothing$ cannot be dropped. If $\Om$
contains no complete cell, then $\GamE=\varnothing$, $\Dep$ contains every
hole, and $b_\eps(v,v)=\int_{\OmE}q^\eps v^2\,dx$. A sign carried only by $\rho$
then produces no eigenvalue of that sign; in the purely Steklov case
$q\equiv0$ the problem has no eigenvalue at all. All results of the paper
are asymptotic and are stated for $\eps<\eps_\Om$ or smaller.
\end{remark}

We next record normalization and exclusion of zero.

\begin{proposition}
\label{prop:prelim-signed-normalization}
Assume \textup{(H1)}, and let
$q\in L^\infty_\#(\Ys)$ and $\rho\in L^\infty_\#(\Gam)$ be real; no sign
condition on the combined spectral measure is required in this proposition.
Let $(\lambda_\eps,u_\eps)$ be any nonzero eigencouple of
\eqref{eq:spectral-micro}. Then $b_\eps(u_\eps,u_\eps)\ne0$ and $u_\eps$ may be normalized by
\begin{equation}
 b_\eps(u_\eps,u_\eps)=\operatorname{sign}(\lambda_\eps).
 \label{eq:prelim-signed-normalization}
\end{equation}
With this normalization,
\begin{align}
 a_\eps(u_\eps,u_\eps)&=|\lambda_\eps|,
 \label{eq:prelim-energy-identity}\\
 \frac{|\lambda_\eps|}{\beta_0}
 \le\|u_\eps\|_{V_\eps}^2
 &\le\frac{|\lambda_\eps|}{\alpha_0},
 \label{eq:prelim-eigenfunction-energy-estimate}\\
 \|u_\eps\|_{L^2(\OmE)}
 +\eps^{1/2}\|u_\eps\|_{L^2(\GamE)}
 &\le C|\lambda_\eps|^{1/2}.
 \label{eq:prelim-eigenfunction-L2-trace-estimate}
\end{align}
There exists $c_*>0$, independent of sufficiently small $\eps$, such that every nonzero microscopic eigenvalue satisfies
\begin{equation}
 |\lambda_\eps|\ge c_*.
 \notag
\end{equation}
\end{proposition}

\begin{proof}
Testing \eqref{eq:spectral-micro} with $u_\eps$ gives
\[
 a_\eps(u_\eps,u_\eps)=\lambda_\eps b_\eps(u_\eps,u_\eps).
\]
Coercivity of $a_\eps$ shows that $b_\eps(u_\eps,u_\eps)$ cannot vanish. Multiplication of $u_\eps$ by a scalar gives \eqref{eq:prelim-signed-normalization}, and then \eqref{eq:prelim-energy-identity}. Ellipticity yields \eqref{eq:prelim-eigenfunction-energy-estimate}, while \eqref{eq:mass-energy-bound} gives \eqref{eq:prelim-eigenfunction-L2-trace-estimate}.

From \eqref{eq:b-bounded} and ellipticity,
\[
 |b_\eps(u,v)|
 \le C_0a_\eps(u,u)^{1/2}a_\eps(v,v)^{1/2}
\]
with $C_0$ independent of $\eps$. Hence $\|\Keps\|_{\mathcal L(V_\eps,a_\eps)}\le C_0$. If $\tau_\eps\ne0$ is an eigenvalue of $\Keps$, then $|\tau_\eps|\le C_0$, and \eqref{eq:prelim-lambda-tau} gives $|\lambda_\eps|\ge C_0^{-1}=:c_*$.
\end{proof}

All fixed-index bounds use the following min--max principle.

\begin{lemma}
\label{lem:signed-courant-fischer}
Let $(H,a)$ be a real Hilbert space and let $c$ be a bounded symmetric
bilinear form on $H$ whose associated operator $\mathcal K$,
$a(\mathcal Kf,v)=c(f,v)$ for $f,v\in H$, is compact. Order the eigenvalues of
$a(u,v)=\Lambda c(u,v)$ for all $v\in H$, with $u\in H\setminus\{0\}$, as
$0<\Lambda_1^+\le\Lambda_2^+\le\cdots$ and
$0>\Lambda_1^-\ge\Lambda_2^-\ge\cdots$. Then, for $\sigma\in\{-1,+1\}$ and
every $k$ for which $\Lambda_k^\sigma$ exists,
\begin{equation}
 \sigma\Lambda_k^\sigma
 =\min_{\substack{L\subset H,\ \dim L=k\\ \sigma c|_L>0}}
   \ \max_{0\ne v\in L}\frac{a(v,v)}{\sigma c(v,v)}.
 \notag
\end{equation}
In particular, if $L\subset H$ is $k$-dimensional and
$\sigma c(v,v)\ge\gamma\|v\|^2_*$, $a(v,v)\le M\|v\|_*^2$ on $L$ for some
norm $\|\cdot\|_*$, then $\sigma\Lambda_k^\sigma\le M/\gamma$.
\end{lemma}

\begin{proof}
This is the Courant--Fischer principle for the positive eigenvalues of the
compact self-adjoint operator $\sigma\mathcal K$ on $(H,a)$, combined with
$\Lambda=1/\tau$.
\end{proof}

\subsection{A common fixed-cluster principle on varying spaces}
\label{subsec:common-rank-principle}

The approximation of isolated eigenvalue clusters of compact self-adjoint operators, including multiplicity, spectral projectors and eigenspaces, is classical; see, for instance, Kato \cite{Kato1976}, Osborn \cite{Osborn1975}, and Babu\v{s}ka--Osborn \cite{BabuskaOsborn1991}. In the present problem, however, both the Hilbert spaces and the operators depend on $\eps$. We therefore establish a fixed-cluster principle for this varying-space setting. It combines non-pollution, finite-dimensional quasimodes and Gram-matrix control to yield preservation of multiplicity together with quantitative estimates for eigenvalues, spectral projectors and aligned eigenspaces.

\begin{theorem}
\label{thm:abstract-fixed-cluster}
For $0<\eps<1$, let $(H_\eps,(\cdot,\cdot)_\eps)$ be a real
Hilbert space and let $\mathcal A_\eps$ be compact and self-adjoint.  Let
$E_0\subset X$ be an $r$-dimensional subspace of a fixed Hilbert space $X$,
let $(\mathcal I_\eps)_{0<\eps<1}$ satisfy
$\mathcal I_\eps:H_\eps\to X$, and fix
$\tau_0\ne0$ and $d_0$ such that $0<d_0<|\tau_0|$, and set
$I_0=(\tau_0-d_0,\tau_0+d_0)$.  Set
$\Pi_\eps=\ind_{I_0}(\mathcal A_\eps)$.  Assume:

\begin{enumerate}[label=\textup{(\roman*)},leftmargin=1.0cm]
\item \emph{Non-pollution.} For every sequence $\eps_n\to0$ and every
family $p_{\eps_n,1},\ldots,p_{\eps_n,s}$ of $(\cdot,\cdot)_{\eps_n}$-orthonormal
eigenvectors of $\mathcal A_{\eps_n}$ with eigenvalues in $I_0$, there is a
subsequence along which each $\mathcal I_{\eps_n}p_{\eps_n,i}$ converges
strongly in $X$, and the limits are linearly independent elements of $E_0$.

\item There are families of linear maps
$(J_\eps)_{0<\eps<1}$, $J_\eps:E_0\to H_\eps$, and numbers
$(\delta_\eps)_{0<\eps<1}$ with $\delta_\eps\to0$, together
with a norm $\|\cdot\|_0$ on $E_0$ and constants $c_0,C_0>0$, such that
\begin{equation}
 c_0\|\phi\|_0\le \|J_\eps\phi\|_\eps\le C_0\|\phi\|_0,\qquad \|(\mathcal A_\eps-\tau_0I)J_\eps\phi\|_\eps
 \le\delta_\eps\|J_\eps\phi\|_\eps.
 \label{eq:abstract-rank-residual}
\end{equation}
\end{enumerate}
Assume in addition that $E_0$ carries an inner product
$(\cdot,\cdot)_0$ and
\begin{equation}
 |(J_\eps\phi,J_\eps\psi)_\eps-(\phi,\psi)_0|
 \le C_G\delta_\eps\|\phi\|_0\|\psi\|_0.
 \label{eq:abstract-Gram-assumption}
\end{equation}
Then, for all small $\eps$,
\begin{equation}
 \operatorname{rank}\Pi_\eps=r,
 \qquad
 \|\Pi_\eps-Q_\eps\|_{\mathcal L(H_\eps)}
 \le \frac{\delta_\eps}{d_0},
 \notag
\end{equation}
where $Q_\eps$ is the orthogonal projector onto $J_\eps E_0$.  If
$\tau_{\eps,1},\ldots,\tau_{\eps,r}$ are the eigenvalues of
$\mathcal A_\eps$ in $I_0$, then
\begin{equation}
 \max_{1\le i\le r}|\tau_{\eps,i}-\tau_0|\le2\delta_\eps.
 \label{eq:abstract-complete-rates}
\end{equation}
Consequently, for every $C^1$ function $g$ near $\overline I_0$,
\begin{equation}
 \max_{1\le i\le r}|g(\tau_{\eps,i})-g(\tau_0)|\le C_g\delta_\eps.
 \label{eq:abstract-transformed-eigenvalues}
\end{equation}
If $(\phi_i)_{i=1}^r$ is an orthonormal basis of $E_0$ and
$(p_{\eps,i})_{i=1}^r$ an orthonormal eigenbasis of $\Ran\Pi_\eps$, there
exists a family $(\mathsf O_\eps)_{0<\eps<1}\subset O(r)$ such that, with
$\Phi_{\eps,i}=\sum_{j=1}^r(\mathsf O_\eps)_{ij}\phi_j$,
\begin{equation}
 \max_{1\le i\le r}\|p_{\eps,i}-J_\eps\Phi_{\eps,i}\|_\eps
 \le C\delta_\eps.
 \label{eq:abstract-orthogonal-alignment}
\end{equation}
The same conclusions hold qualitatively if the two-sided bound in
condition \textup{(ii)} remains uniform, while its relative residual is
$o(1)$ uniformly on the unit sphere of $E_0$ and the Gram defect in
\eqref{eq:abstract-Gram-assumption} is $o(1)$ there.  In that case the
projector, eigenvalue, and orthogonal-alignment errors above are all $o(1)$.

Let, moreover, $\Sigma_0\subset\mathbb R\setminus\{0\}$ be a closed
discrete subset of $\mathbb R\setminus\{0\}$ (the limit spectrum) with
$\tau_0\in\Sigma_0$, let $\mathcal U$ be the open half-line containing
$\tau_0$, and let $\mathcal J$ be a bounded open interval such that
\[
 \overline{\mathcal J}\subset\mathcal U,\qquad
 \overline{\mathcal J}\cap\Sigma_0=\{\tau_0\}.
\]
If condition \textup{(i)} holds, with $E_0=\{0\}$, for every open
interval $I$ such that $\overline I\subset\mathcal U\setminus\Sigma_0$,
then
\[
 \operatorname{rank}\ind_{\mathcal J}(\mathcal A_\eps)=r
\]
for all sufficiently small $\eps$.
\end{theorem}

\begin{proof}
Condition (i) excludes more than $r$ eigenvectors in $I_0$ for small
$\eps$: otherwise $r+1$ orthonormal eigenvectors would have $r+1$ linearly
independent limits in the $r$-dimensional space $E_0$.  Conversely,
if $\operatorname{rank}\Pi_\eps<r$, some nonzero
$z\in J_\eps E_0$ satisfies $\Pi_\eps z=0$, and the spectral theorem gives
\[
 \|(\mathcal A_\eps-\tau_0I)z\|_\eps\ge d_0\|z\|_\eps,
\]
contradicting condition (ii) for small $\eps$.  Thus the rank is exactly $r$.  The
same spectral estimate gives
$\|(I-\Pi_\eps)Q_\eps\|\le\delta_\eps/d_0$; since the two projectors are
orthogonal and have the same finite rank, both this quantity and
$\|\Pi_\eps-Q_\eps\|$ equal the sine of the largest principal angle
between their ranges.

Applying the preceding rank argument to
$(\tau_0-2\delta_\eps,\tau_0+2\delta_\eps)$ shows that all $r$
eigenvalues in $I_0$ lie in that smaller interval, proving
\eqref{eq:abstract-complete-rates}; \eqref{eq:abstract-transformed-eigenvalues}
follows from the mean-value theorem.  Finally,
\eqref{eq:abstract-Gram-assumption} gives a frame
$J_\eps\phi_i$ with Gram matrix $I_r+O(\delta_\eps)$.  After the standard
$G_\eps^{-1/2}$ orthonormalization, the projector estimate bounds the
principal angles between this frame and $\Ran\Pi_\eps$ by
$O(\delta_\eps)$.  The polar factor of the cross-Gram matrix then yields
$\mathsf O_\eps\in O(r)$ and \eqref{eq:abstract-orthogonal-alignment}.  The same
argument with $o(1)$ errors gives the qualitative version.

For the last assertion, choose an interval $I_0$ as above with
$\overline I_0\subset\mathcal J$.  By the preceding argument, the rank of
$\ind_{I_0}(\mathcal A_\eps)$ is $r$ for small $\eps$.
The compact set $\overline{\mathcal J}\setminus I_0$ is contained in
$\mathcal U\setminus\Sigma_0$, which is open; hence it is covered by
finitely many open intervals $I_1,\ldots,I_M$ with
$\overline{I_m}\subset\mathcal U\setminus\Sigma_0$.  Condition
\textup{(i)} with $E_0=\{0\}$ gives $\ind_{I_m}(\mathcal A_\eps)=0$
for small $\eps$ and $1\le m\le M$.  Hence no additional eigenvalue
enters $\mathcal J$, and therefore the rank of
$\ind_{\mathcal J}(\mathcal A_\eps)$ equals $r$.
\end{proof}

\begin{remark}[Boundary-compatible diagonal recoveries]
\label{rem:boundary-compatible-diagonal}
Whenever a qualitative recovery uses profiles
$\phi_i^{(m)}\in C_c^\infty(\Omega)$, the diagonal $m=m(\eps)\to\infty$ is
chosen sufficiently slowly that
\[
 \eps\,\Big/\min_{1\le i\le r}\operatorname{dist}
 (\operatorname{supp}\phi_i^{(m(\eps))},\partial\Omega)\longrightarrow0.
\]
Then the profile and all its derivatives vanish throughout every boundary
cell for sufficiently small $\eps$, so the corresponding interior recovery
belongs to $V_\eps$.  This additional diagonal constraint is compatible
with all finitely many approximation and Gram-matrix requirements used
below.
\end{remark}

\section{Noncritical regimes: regular and factorized branches}
\label{sec:noncritical-regimes}

This section proves Theorems~\ref{thm:intro-regular} and \ref{thm:intro-factorized} and treats the case $m_{\rm eff}<0$.

\subsection{Regular positive branch: qualitative homogenization}
\label{subsec:regular-positive-qualitative}

We begin with regular-branch fixed-index bounds.

\begin{proposition}
\label{prop:prelim-positive-fixed-index}
Assume \textup{(H1)--(H2)} and $m_{\mathrm{eff}}>0$. For every fixed $k\ge1$ there exist $C_k>0$ and $\eps_k>0$ such that
\begin{equation}
 c_*\le\lambda_{\eps,j}^+\le C_k,
 \qquad j=1,\ldots,k,
 \quad 0<\eps<\eps_k,
 \notag
\end{equation}
where $c_*$ is the constant in Proposition~\ref{prop:prelim-signed-normalization}. For eigenfunctions normalized by
$b_\eps(u_{\eps,j}^+,u_{\eps,j}^+)=1$,
\begin{equation}
 \|u_{\eps,j}^+\|_{V_\eps}
 +\|P_\eps u_{\eps,j}^+\|_{H_0^1(\Om)}
 +\eps^{1/2}\|u_{\eps,j}^+\|_{L^2(\GamE)}
 \le C_k,
 \qquad j=1,\ldots,k.
 \notag
\end{equation}
\end{proposition}

\begin{proof}
Choose a $k$-dimensional subspace $L_k\subset C_c^\infty(\Om)$. Since the union of the supports of a fixed basis of $L_k$ has positive distance from $\partial\Omega$, for all sufficiently small $\eps$ the restrictions of all functions in $L_k$ vanish throughout every boundary cell and therefore belong to $V_\eps$. By Proposition~\ref{prop:prelim-trace-mass-limit}, uniformly on the unit sphere of $L_k$,
\[
 b_\eps(\phi|_{\OmE},\psi|_{\OmE})
 \longrightarrow m_{\mathrm{eff}}\int_\Om\phi\psi\,dx.
\]
Thus the mass matrix on $L_k$ is positive definite for all sufficiently small $\eps$, with its smallest eigenvalue bounded below independently of $\eps$. On the same finite-dimensional space, ellipticity and smoothness of the basis imply a uniform upper bound for the energy matrix. Lemma~\ref{lem:signed-courant-fischer} with $\sigma=+1$ and $L=L_k$ gives $\lambda_{\eps,k}^+\le C_k$, and monotonicity gives the same bound for $1\le j\le k$. Proposition~\ref{prop:prelim-signed-normalization} supplies the lower bound $c_*$. The eigenfunction estimate follows from \eqref{eq:prelim-energy-identity}, the extension estimate \eqref{eq:extension}, and the scaled trace estimate.
\end{proof}

These bounds yield the qualitative limit.

\begin{theorem}
\label{thm:qual-positive}
Assume \textup{(H1)--(H2)} and $m_{\mathrm{eff}}>0$. For $j=1,\ldots,N$, let
$\chi^j\in H^1_\#(\Ys)/\R$ solve
\begin{equation}
 \int_{\Ys}A(y)(e_j+\nabla_y\chi^j)\cdot\nabla_y\psi\,dy=0
 \qquad\forall\psi\in H^1_\#(\Ys)/\R,
 \label{eq:cell-chi}
\end{equation}
and define
\begin{equation}
 B^j=A(e_j+\nabla_y\chi^j),
 \qquad
 \Ah e_j=\int_{\Ys}B^j(y)\,dy.
 \label{eq:Ah}
\end{equation}
Then $\Ah$ is symmetric positive definite. Let
\begin{equation}
 -\Div(\Ah\nabla\phi)=\mu\phi\quad\text{in }\Om,
 \qquad \phi=0\quad\text{on }\partial\Om,
 \label{eq:hom-Dirichlet-mu}
\end{equation}
have eigenvalues
$0<\mu_1^0\le\mu_2^0\le\cdots\to+\infty$, repeated according to multiplicity, and set
\begin{equation}
 \lambda_k^0:=\frac{\mu_k^0}{m_{\mathrm{eff}}}.
 \notag
\end{equation}
Then, for every fixed $k\ge1$,
\begin{equation}
 \lambda_{\eps,k}^+\longrightarrow\lambda_k^0.
 \label{eq:qual-lambda}
\end{equation}

Suppose that $\mu_k^0=\cdots=\mu_{k+r-1}^0=: \mu^0$ and set
\begin{equation}
 E_k^0:=\ker\bigl(-\Div(\Ah\nabla\cdot)-\mu^0I\bigr),
 \qquad \dim E_k^0=r.
 \notag
\end{equation}
Let $\mathcal J_k^+\Subset(0,+\infty)$ be any open interval such that
\begin{equation}
 \lambda_k^0\in \mathcal J_k^+,
 \qquad
 \overline{\mathcal J_k^+}\cap\{\lambda_j^0:j\ge1\}=\{\lambda_k^0\}.
 \notag
\end{equation}
Choose a basis $\{\phi_k,\ldots,\phi_{k+r-1}\}$ of $E_k^0$ such that
\begin{equation}
 m_{\mathrm{eff}}\int_\Om\phi_{k+i-1}\phi_{k+j-1}\,dx=\delta_{ij},
 \qquad 1\le i,j\le r,
 \notag
\end{equation}
and microscopic eigenfunctions such that
\begin{equation}
 b_\eps(u_{\eps,k+i-1}^+,u_{\eps,k+j-1}^+)=\delta_{ij},
 \qquad 1\le i,j\le r.
 \notag
\end{equation}
For all sufficiently small $\eps$, exactly $r$ positive microscopic
eigenvalues, counted with multiplicity, belong to $\mathcal J_k^+$. Equivalently,
the isolated microscopic cluster converging to $\lambda_k^0$ has rank
$r$ and no additional eigenvalue enters it. Moreover, there exists a family
of orthogonal matrices
\begin{equation}
 (\mathsf O_{\eps,k}^+)_{0<\eps<1},\qquad
 \mathsf O_{\eps,k}^+=(\mathsf o_{ij}^{\eps,+})_{1\le i,j\le r}\in O(r)
 \notag
\end{equation}
such that
\begin{equation}
 \max_{1\le i\le r}
 \left\|P_\eps u_{\eps,k+i-1}^+
 -\sum_{j=1}^{r}\mathsf o_{ij}^{\eps,+}\phi_{k+j-1}\right\|_{L^2(\Om)}
 \longrightarrow0.
 \notag
\end{equation}
After extraction of a subfamily, without relabeling, suppose that
$\mathsf O_{\eps,k}^+\to(\mathsf o_{ij}^+)\in O(r)$ and
$\phi_i^*:=\sum_{j=1}^{r}\mathsf o_{ij}^+\phi_{k+j-1}$, then
 \begin{align}
 \ind_{\OmE}P_\eps u_{\eps,k+i-1}^+
 &\stackrel{2}{\rightharpoonup}
 \ind_{Y^*}\phi_i^*,\notag\\
 \ind_{\OmE}\nabla P_\eps u_{\eps,k+i-1}^+
 &\stackrel{2}{\rightharpoonup}
 \ind_{Y^*}\left(\nabla_x\phi_i^*
 +\sum_{\ell=1}^N\nabla_y\chi^\ell\,\partial_\ell\phi_i^*\right),\notag
 \end{align}
for $i=1,\ldots,r$.
\end{theorem}

\begin{proof}
We first establish compactness for eigenpairs with bounded index.
Proposition~\ref{prop:prelim-positive-fixed-index} gives, for every fixed $k$, constants independent of $\eps$ such that the first $k$ positive eigenvalues stay in a compact subset of $(0,+\infty)$. Let
$(\lambda_\eps,u_\eps)_{0<\eps<1}$ be any family of such
eigencouples with index bounded by $k$, normalized by
$b_\eps(u_\eps,u_\eps)=1$. The spectral identity yields
$a_\eps(u_\eps,u_\eps)=\lambda_\eps$, so
$(P_\eps u_\eps)_{0<\eps<1}$ is bounded in
$H_0^1(\Om)$. Theorem~\ref{thm:two-scale-compactness-prelim} gives, after extraction,
\[
 P_\eps u_\eps\rightharpoonup u\ \hbox{in }H_0^1(\Om),\qquad
 P_\eps u_\eps\to u\ \hbox{in }L^2(\Om),
\]
and, after zero extension from the solid phase,
\[
 \ind_{\OmE}\nabla P_\eps u_\eps
 \stackrel{2}{\rightharpoonup}
 \ind_{Y^*}(\nabla_xu+\nabla_yu_1)
 \quad\text{in }L^2(\Om\times Y)^N.
\]

We next identify the problem satisfied by every limit pair. Let
$\varphi_0\in C_c^\infty(\Omega)$ and
$\varphi_1\in C_c^\infty(\Omega;C_\#^\infty(Y^*))$.  Their common
macroscopic support has positive distance from $\partial\Omega$; hence, for
all sufficiently small $\eps$, the oscillating test
$\varphi_0(x)+\eps\varphi_1(x,x/\eps)$ vanishes throughout every boundary
cell and belongs to $V_\eps$.  Testing with this function and using
Proposition~\ref{prop:prelim-trace-mass-limit} yields
\[
 \int_{\Om\times\Ys}A(y)(\nabla_xu+\nabla_yu_1)\cdot
 (\nabla_x\varphi_0+\nabla_y\varphi_1)\,dx\,dy =\lambda m_{\mathrm{eff}}\int_\Om u\varphi_0\,dx.
\]
The choice $\varphi_0=0$ gives $u_1=\sum_{j=1}^N\chi^j\partial_ju$, and then $\varphi_1=0$ gives
\[
 \int_\Om\Ah\nabla u\cdot\nabla\varphi_0\,dx
 =\lambda m_{\mathrm{eff}}\int_\Om u\varphi_0\,dx.
\]
The mass normalization and Proposition~\ref{prop:prelim-trace-mass-limit} give
$1=m_{\mathrm{eff}}\int_\Om|u|^2\,dx$, so $u\ne0$.

We finally establish spectral counting and convergence of the associated eigenspaces.
Put $\tau_k^+=1/\lambda_k^0$, $\Sigma_0^+:=\{m_{\mathrm{eff}}/\mu_j^0:j\ge1\}$
and $I_0^+:=(\tau_k^+-d_0,\tau_k^++d_0)$, where
$0<d_0<\dist\bigl(\tau_k^+,(\{0\}\cup\Sigma_0^+)\setminus\{\tau_k^+\}\bigr)$;
thus $\overline{I_0^+}\subset(0,\infty)$ and
$\overline{I_0^+}\cap\Sigma_0^+=\{\tau_k^+\}$. Let
$(p_{\eps,1},\ldots,p_{\eps,s})_{0<\eps<1}$ be a family
of $a_\eps$-orthonormal eigenvectors of $\Keps$ with eigenvalues
$\tau_{\eps,i}\in I_0^+$. After extraction,
$\tau_{\eps,i}\to\tau_i^*\in\overline{I_0^+}$ and, by the preceding
compactness argument, $P_\eps p_{\eps,i}\to p_{0,i}$ in $L^2(\Om)$.
Moreover,
\begin{equation}
 b_\eps(p_{\eps,i},p_{\eps,j})
 =a_\eps(\Keps p_{\eps,i},p_{\eps,j})
 =\tau_{\eps,i}\delta_{ij},
 \label{eq:qual-positive-cross-mass-micro}
\end{equation}
and Proposition~\ref{prop:prelim-trace-mass-limit} yields
\begin{equation}
 m_{\mathrm{eff}}\int_\Om p_{0,i}p_{0,j}\,dx
 =\tau_i^*\delta_{ij}.
 \label{eq:qual-positive-cross-mass-limit}
\end{equation}
Since $\tau_i^*>0$, each $p_{0,i}$ is nonzero, and the preceding limit
identification, applied with $\lambda=1/\tau_i^*$, shows that $p_{0,i}$ is
an eigenfunction of $-\Div(\Ah\nabla\cdot)$ in $H_0^1(\Om)$ with
eigenvalue $m_{\mathrm{eff}}/\tau_i^*$. Hence
$\tau_i^*\in\Sigma_0^+\cap\overline{I_0^+}=\{\tau_k^+\}$ and
$p_{0,i}\in E_k^0$. Thus the limit vectors are linearly independent and
$s\le r$, which verifies condition (i).
For the abstract principle we take explicitly
\[
 X=L^2(\Om),\qquad \mathcal I_\eps=P_\eps.
\]

For the lower bound, we use the basis of the statement and write
$\phi_i:=\phi_{k+i-1}$, $1\le i\le r$. For each $i$ choose
$\phi_i^{(m)}\in C_c^\infty(\Om)$ such that
$\phi_i^{(m)}\to\phi_i$ in $H_0^1(\Om)$.  If
$\phi=\sum_{i=1}^rc_i\phi_i$, put
\[
 \phi^{(m)}:=\sum_{i=1}^rc_i\phi_i^{(m)},
 \qquad
 R_{\eps,m}^+\phi
 :=\left.\left(\phi^{(m)}
 +\eps\sum_{j=1}^N(\chi^j)^\eps\partial_j\phi^{(m)}\right)\right|_{\OmE}.
\]
For each fixed $m$, the map $R_{\eps,m}^+:E_k^0\to V_\eps$ is linear for all sufficiently small $\eps$; along the diagonal below this admissibility is ensured by Remark~\ref{rem:boundary-compatible-diagonal}.
We define the variational defect
\[
 \mathcal D_{\eps,m,\phi}(v)
 :=a_\eps(R_{\eps,m}^+\phi,v)
 -\lambda_k^0 b_\eps(R_{\eps,m}^+\phi,v).
\]
The required smallness is in the full dual energy norm, not merely against
fixed test functions. Indeed,
\begin{equation}
 \limsup_{\eps\to0}
 \sup_{\|v\|_{a_\eps}=1}
 |\mathcal D_{\eps,m,\phi}(v)|
 \le C\|\phi^{(m)}-\phi\|_{H_0^1(\Om)}.
 \label{eq:qual-positive-dual-defect}
\end{equation}
To prove this estimate, otherwise choose a family
$(v_\eps)_{0<\eps<1}$ with
$\|v_\eps\|_{a_\eps}=1$ that violates the bound. The extension estimate
makes $(P_\eps v_\eps)_{0<\eps<1}$ bounded in
$H_0^1(\Om)$; after extraction it converges
weakly in $H_0^1(\Om)$ and strongly in $L^2(\Om)$, while its gradients admit
the usual weak two-scale limit. Since the recovery gradient has the strong
two-scale limit determined by $\phi^{(m)}$, the cell equation
\eqref{eq:cell-chi} and Proposition~\ref{prop:prelim-trace-mass-limit} give
\[
 \lim_{\eps\to0}\mathcal D_{\eps,m,\phi}(v_\eps)
 =\int_\Om\Ah\nabla\phi^{(m)}\cdot\nabla v_0\,dx
 -\lambda_k^0m_{\mathrm{eff}}\int_\Om\phi^{(m)}v_0\,dx.
\]
Because $\phi$ satisfies the limit eigenvalue equation, the right-hand side
contains only $\phi^{(m)}-\phi$ and is bounded by
$C\|\phi^{(m)}-\phi\|_{H_0^1(\Om)}$. This proves
\eqref{eq:qual-positive-dual-defect}. Because the basis is fixed and
$\max_{1\le i\le r}\|\phi_i^{(m)}-\phi_i\|_{H_0^1}\to0$, the estimate is uniform on
the unit sphere of $E_k^0$. For each fixed $m$, the two-scale calculation
also gives convergence of every entry of the Gram matrix of
$\{R_{\eps,m}^+\phi_i\}_{i=1}^r$. A single diagonal choice $m=m(\eps)\to\infty$, common to all basis vectors and all matrix entries and satisfying Remark~\ref{rem:boundary-compatible-diagonal}, therefore gives linear maps
\[
 J_{\eps,k}^+:=R_{\eps,m(\eps)}^+:E_k^0\longrightarrow V_\eps
\]
whose Gram matrices converge to the positive-definite limit energy Gram
matrix. Hence $J_{\eps,k}^+$ is injective for all sufficiently small
$\eps$, and
\[
 \|(\Keps-\tau_k^+I)J_{\eps,k}^+\phi\|_{a_\eps}
 =o(1)\|J_{\eps,k}^+\phi\|_{a_\eps}.
\]
Moreover,
\[
 \|J_{\eps,k}^+\phi-\phi\|_{L^2(\OmE)}\longrightarrow0,
 \qquad
 P_\eps J_{\eps,k}^+\phi-\phi
 \quad\hbox{is bounded in }H^1_0(\Om).
\]
If, along a subsequence,
$P_\eps J_{\eps,k}^+\phi-\phi\to h$ strongly in $L^2(\Om)$, then periodic
averaging of $\ind_{\OmE}$ and the first convergence give
\[
 |Y^*|\int_\Om |h|^2\,dx
 =\lim_{\eps\to0}
   \int_{\OmE}|P_\eps J_{\eps,k}^+\phi-\phi|^2\,dx=0.
\]
Thus $P_\eps J_{\eps,k}^+\phi\to\phi$ strongly in $L^2(\Om)$.
Theorem~\ref{thm:abstract-fixed-cluster}, applied with the inner product
$(\phi,\psi)_0:=\int_\Om\Ah\nabla\phi\cdot\nabla\psi\,dx$ on $E_k^0$,
for which $\{\phi_i/(\lambda_k^0)^{1/2}\}_{i=1}^r$ is orthonormal by the
mass normalization of the statement, now gives exactly $r$ eigenvalues in
$I_0^+$, together with gap convergence and an orthogonal alignment of the
$a_\eps$-normalized eigenvectors
$u_{\eps,k+i-1}^+/(\lambda_{\eps,k+i-1}^+)^{1/2}$ with the vectors
$\phi_i/(\lambda_k^0)^{1/2}$. Since
$\lambda_{\eps,k+i-1}^+\to\lambda_k^0$, undoing these normalizations
preserves the $o(1)$ alignment. The same argument as for condition (i) shows that an open
interval whose closure lies in $(0,\infty)$ and does not meet
$\Sigma_0^+$ contains no eigenvalue of $\Keps$ for small $\eps$ (a limit
point would belong to $\Sigma_0^+$), so the final assertion of
Theorem~\ref{thm:abstract-fixed-cluster} gives
the count in an arbitrary isolating interval $\mathcal J_k^+$. Applying the
argument successively to the isolated levels proves \eqref{eq:qual-lambda}
with multiplicities; the aligned $L^2$ and two-scale statements follow from
the same recoveries and the preceding identification of the two-scale limit.
\end{proof}

The proof above is coefficient-generic: it uses only uniform ellipticity,
the extension and trace bounds, convergence of the bulk--surface mass with
positive mean, the periodic corrector problem, and
Theorem~\ref{thm:abstract-fixed-cluster}. We shall use this observation
below after the ground-state transform.

For the quantitative analysis we use the reciprocal parameters
$\tau_{\eps,k}^+:=1/\lambda_{\eps,k}^+$ and
$\tau_k^+:=1/\lambda_k^0=m_{\mathrm{eff}}/\mu_k^0$.

\subsection{Quantitative regular positive branch}
\label{sec:mass}

The first quantitative ingredient is an exact representation of the
centered bulk--surface mass. The same construction, with $A$ in place of
the identity and $m_{\mathrm{eff}}=0$, gives the compensation identity of the critical regime
(Section~\ref{subsec:critical-compensation}).

Define on the torus $Y$ the function
\begin{equation}
 q^\circ(y):=
 \begin{cases}
 q(y),&y\in\Ys,\\
 0,&y\in T.
 \end{cases}
 \label{eq:qcirc}
\end{equation}
The extension by zero is a device on the torus; no coefficient is assigned
inside the holes.

Since $\int_{\Ys}q\,dy+\int_\Gam\rho\,dS_y=m_{\mathrm{eff}}$, the linear form
\begin{equation}
 \mathfrak m(\psi)
 :=\int_{\Ys}\Bigl(q-\frac{m_{\mathrm{eff}}}{|\Ys|}\Bigr)\psi\,dy
 +\int_\Gam\rho\psi\,dS_y,
 \qquad \psi\in H^1_\#(\Ys),
 \notag
\end{equation}
satisfies
\begin{equation}
 \mathfrak m(1)=0.
 \notag
\end{equation}
The trace theorem shows that $\mathfrak m$ is continuous on
$H^1_\#(\Ys)$, and the Lax--Milgram lemma gives a unique
\begin{equation}
 \chi^{\mathrm m}\in H^1_\#(\Ys)/\R,
 \qquad
 \int_{\Ys}\nabla_y\chi^{\mathrm m}\cdot\nabla_y\psi\,dy
 =\mathfrak m(\psi)
 \qquad\forall\psi\in H^1_\#(\Ys).
 \label{eq:chi-mass}
\end{equation}
By Lemma~\ref{lem:cell-regularity-H4}\textup{(iii)}, applied with $A$
replaced by the identity, $g=q-m_{\mathrm{eff}}/|\Ys|\in L^\infty_\#(\Ys)$
and $h=\rho$, we have $\nabla_y\chi^{\mathrm m}\in L^\infty(\Ys)^N$.
Next, since $|Y|=1$, the function $\ind_{\Ys}/|\Ys|-1$ is bounded, periodic
and has zero mean on $Y$, so Lemma~\ref{lem:periodic-potentials}\textup{(b)} gives
$H^{\mathrm m}\in L^\infty_\#(Y)^N$ with
$\Div_yH^{\mathrm m}=\ind_{\Ys}/|\Ys|-1$. Set
\begin{equation}
 F^{\mathrm m}:=\ind_{\Ys}\nabla_y\chi^{\mathrm m}
 -m_{\mathrm{eff}}H^{\mathrm m}\in L^\infty_\#(Y)^N.
 \notag
\end{equation}
Since the zero extension of $\nabla_y\chi^{\mathrm m}$ satisfies
$-\Div_y(\ind_{\Ys}\nabla_y\chi^{\mathrm m})
=(q^\circ-m_{\mathrm{eff}}\ind_{\Ys}/|\Ys|)\,dy+\rho\,dS_\Gam$ on the torus,
by \eqref{eq:chi-mass}, we obtain, in the sense of periodic distributions,
\begin{equation}
 -\Div_yF^{\mathrm m}
 =(q^\circ-m_{\mathrm{eff}})\,dy
 +\rho\,dS_\Gam,
 \label{eq:Fmass-div}
\end{equation}
and
\begin{equation}
 F^{\mathrm m}\in L^\infty(Y)^N.
 \label{eq:Fmass-bounded}
\end{equation}
These two properties are the only ones used below.

Rescaling \eqref{eq:Fmass-div} produces a surface term on every periodic copy
of $\Gam$ meeting $\Om$; the copies not in $\GamE$ belong to $\Dep$ and do not
contribute when one factor vanishes there. Hence, if $U=P_\eps v$ and
$V=P_\eps w$ with $v,w\in V_\eps$,
\begin{equation}
\eps\int_\Om(F^{\mathrm m})^\eps\cdot\nabla(UV)\,dx
 =\int_\Om\bigl((q^\circ)^\eps-m_{\mathrm{eff}}\bigr)UV\,dx +\eps\int_{\GamE}\rho^\eps UV\,dS_x.
\label{eq:mass-div-scaled}
\end{equation}

Now take $v,w\in V_\eps$ and set
\[
 U=\Peps v,\qquad V=\Peps w.
\]
Inside the perforations, $(q^\circ)^\eps=0$, whereas in $\OmE$ it equals
$q^\eps$. Hence \eqref{eq:mass-div-scaled} rearranges into the exact global identity
\begin{equation}
\begin{aligned}
 \be(v,w)-m_{\mathrm{eff}}\int_\Om UV\,dx
 ={}\eps\int_\Om(F^{\mathrm m})^\eps\cdot\nabla(UV)\,dx.
\end{aligned}
 \label{eq:mass-exact}
\end{equation}
The product $UV$ belongs to $W_0^{1,1}(\Om)$, which justifies the density
argument when $U,V$ are merely in $H_0^1(\Om)$.

The right-hand side of \eqref{eq:mass-exact} is controlled by
\eqref{eq:Fmass-bounded}, Poincar\'e's inequality, and the uniform extension:
\begin{equation}
 \eps\abs{\int_\Om(F^{\mathrm m})^\eps\cdot\nabla(UV)\,dx}
 \le C\eps\bigl(
 \norm{U}_2\norm{\nabla V}_2
 +\norm{V}_2\norm{\nabla U}_2\bigr) \le C\eps\norm{v}_{V_\eps}\norm{w}_{V_\eps}.
 \label{eq:mass-F-bound}
\end{equation}
The preceding identity gives quantitative averaging.

\begin{proposition}
\label{prop:mass-average}
Assume \textup{(H2)} and, for some $\alpha\in(0,1)$,
$\Gamma\in C^{2,\alpha}$ and $\rho\in C^{0,\alpha}_\#(\Gamma)$. Then
\begin{equation}
 \abs{
 \be(v,w)
 -m_{\mathrm{eff}}\int_\Om(\Peps v)(\Peps w)\,dx}
 \le C\eps\norm{v}_{V_\eps}\norm{w}_{V_\eps}.
 \notag
\end{equation}
The constant is independent of $\eps$.
\end{proposition}

\begin{proof}
This is \eqref{eq:mass-exact} combined with \eqref{eq:mass-F-bound}.
\end{proof}

The same rescaled identity remains exact with $u\in H_0^1(\Om)$ as the first
factor, because the second factor $v\in V_\eps$ vanishes on every omitted
microscopic surface. This gives the version used in the residual estimate.

\begin{corollary}
\label{cor:mass-smooth-u}
Under the assumptions of Proposition~\ref{prop:mass-average}, for $u\in H_0^1(\Om)$,
$v\in V_\eps$, and $V=\Peps v$,
\begin{equation}
 \abs{
 \be(u|_{\OmE},v)
 -m_{\mathrm{eff}}\int_\Om uV\,dx}
 \le C\eps\norm{u}_{H^1(\Om)}\norm{v}_{V_\eps}.
 \label{eq:mass-smooth-u}
\end{equation}
\end{corollary}

For the cut-geometry reconstruction, assume
\eqref{hyp:cell-regularity}--\eqref{hyp:macro-regularity}.

For $u\in H^2(\Om)\cap H_0^1(\Om)$, set
\begin{equation}
 R_{\eps,\mathrm{int}}^+u
 :=u+\eps\sum_{j=1}^N(\chi^j)^\eps\partial_ju
 \quad\text{in }\OmE.
 \label{eq:Rint-plus}
\end{equation}
The functions $\chi^j$ are periodically defined in $\mathcal W$; no extension into
$T$ is needed. We have
\begin{equation}
 \nabla R_{\eps,\mathrm{int}}^+u
 =\sum_{j=1}^N(e_j+\nabla_y\chi^j)^\eps\partial_ju
 +\eps\sum_{j=1}^N(\chi^j)^\eps\nabla\partial_ju,
 \label{eq:gradRint-plus}
\end{equation}
and the last term is bounded in $L^2(\OmE)$ by
$C\eps\norm{u}_{H^2(\Om)}$.

The Dirichlet trace of the interior corrector on the whole set $D_\eps$
is removed by the common exact layer. In
Lemma~\ref{lem:common-boundary-layer}, take $B=A$, so that
$\chi_A^j=\chi^j$, and write $\beta_\eps[u]:=\beta_{\eps,A}[u]$. Thus
\begin{equation}
 \beta_\eps[u]+R_{\eps,\mathrm{int}}^+u\in V_\eps,\qquad \aep(\beta_\eps[u],v)=0
 \quad \forall v\in V_\eps,
 \notag
\end{equation}
so, in the capacitary sense,
$\beta_\eps[u]=-R_{\eps,\mathrm{int}}^+u$ on $D_\eps$; on
$\partial\Omega$ this reduces to the familiar oscillatory datum because
$u=0$ there.
The layer estimates give
\begin{align}
 \norm{\nabla\beta_\eps[u]}_{L^2(\OmE)}
 &\le C\eps^{1/2}\norm{u}_{H^2(\Om)},
 \label{eq:beta-energy}\\
 \norm{\beta_\eps[u]}_{L^2(\OmE)}
 +\eps^{1/2}\norm{\beta_\eps[u]}_{L^2(\GamE)}
 &\le C\eps\norm{u}_{H^2(\Om)}.
 \label{eq:BL}
\end{align}

We set
\begin{equation}
 \mathcal R_\eps^+u
 :=R_{\eps,\mathrm{int}}^+u
 +\beta_\eps[u].
 \notag
\end{equation}
By construction $\mathcal R_\eps^+u\in V_\eps$, and \eqref{eq:Rint-plus}, \eqref{eq:beta-energy}, \eqref{eq:BL} give
\begin{align}
 \norm{\mathcal R_\eps^+u-u}_{L^2(\OmE)}
 \le C\eps\norm{u}_{H^2(\Om)},
 \label{eq:Rplus-L2}\\
 \norm{\mathcal R_\eps^+u}_{V_\eps}
 \le C\norm{u}_{H^2(\Om)}.
 \label{eq:Rplus-H1-bound}
\end{align}
The relevant small quantity is the variational residual, not the raw gradient
difference, which contains the order-one oscillatory corrector gradient.

For the residual estimate, set
\[
 e_\eps[u]:=\mathcal R_\eps^+u-u|_{\OmE}.
\]
By \eqref{eq:Rplus-L2}--\eqref{eq:Rplus-H1-bound} and the rescaled trace
inequality \eqref{eq:scaled-trace},
\begin{equation}
 \left(\eps\int_{\GamE}|e_\eps[u]|^2\,dS_x\right)^{1/2}
 \le C\left(
 \norm{e_\eps[u]}_2
 +\eps\norm{\nabla e_\eps[u]}_2\right) \le C\eps\norm{u}_{H^2}.
 \notag
\end{equation}
For $v\in V_\eps$, \eqref{eq:Poincare} and \eqref{eq:scaled-trace} likewise
give
\begin{equation}
 \norm{v}_2
 +\left(\eps\int_{\GamE}|v|^2\,dS_x\right)^{1/2}
 \le C\norm{v}_{V_\eps}.
 \notag
\end{equation}
The boundedness of $q$ and $\rho$ therefore yields
\begin{equation}
 \abs{\be(e_\eps[u],v)}
 \le C\eps\norm{u}_{H^2}\norm{v}_{V_\eps}.
 \label{eq:b-delta}
\end{equation}

By \eqref{eq:b-delta} and \eqref{eq:mass-smooth-u},
\begin{equation}
 \abs{
 \be(\mathcal R_\eps^+u,v)
 -m_{\mathrm{eff}}\int_\Om uV\,dx}
 \le C\eps\norm{u}_{H^2}\norm{v}_{V_\eps},
 \qquad V=\Peps v.
 \label{eq:Mplus}
\end{equation}

For the diffusion residual, extend $B^j=A(e_j+\nabla_y\chi^j)$ by zero into $T$:
\begin{equation}
 B_0^j(y):=
 \begin{cases}
 B^j(y),&y\in\Ys,\\
 0,&y\in T.
 \end{cases}
 \notag
\end{equation}
Since $\Div_yB^j=0$ in $\Ys$ and $B^j\cdot n=0$ on $\Gam$,
\begin{equation}
 \Div_yB_0^j=0\quad\hbox{in }\mathcal D'_\#(Y),
 \qquad
 \int_YB_0^j\,dy=\Ah e_j.
 \notag
\end{equation}
Under $\mathrm{(H3}_A\mathrm{)}$ the fields $B_0^j$ are bounded, so
Lemma~\ref{lem:periodic-potentials}\textup{(a)} applies to
$B_0^j-\Ah e_j$.

We next control the diffusion residual.

\begin{proposition}
\label{prop:diffusion-residual}
Assume \textup{(H1)} and \textup{(H4)}, and suppose, for
some $\alpha\in(0,1)$, that
$\Gamma\in C^{2,\alpha}$ and
$A\in C^{0,\alpha}_\#(\overline{Y^*})^{N\times N}$.
For $u\in H^2(\Om)\cap H_0^1(\Om)$, $v\in V_\eps$, and every extension
$V\in H_0^1(\Om)$ such that $V=v$ in $\OmE$,
\begin{equation}
 \abs{
 \aep(\mathcal R_\eps^+u,v)
 -\int_\Om\Ah\nabla u\cdot\nabla V\,dx}
 \le C\eps\norm{u}_{H^2(\Om)}\norm{V}_{H^1(\Om)}.
 \label{eq:Aplus}
\end{equation}
\end{proposition}

\begin{proof}
Formula \eqref{eq:gradRint-plus} gives
\begin{equation}
 \aep(R_{\eps,\mathrm{int}}^+u,v)
 =\sum_{j=1}^N\int_{\OmE}(B^j)^\eps\partial_ju\cdot\nabla v\,dx
 +\mathcal E_{\eps,1}(u,v),\qquad \abs{\mathcal E_{\eps,1}(u,v)}
 \le C\eps\norm{u}_{H^2(\Om)}\norm{v}_{V_\eps}.
 \label{eq:a-Rint-plus}
\end{equation}
Since $B_0^j=0$ in $T$ and $V=v$ in $\OmE$,
\begin{equation}
 \sum_{j=1}^N\int_{\OmE}(B^j)^\eps\partial_ju\cdot\nabla v\,dx
 =\sum_{j=1}^N\int_\Om(B_0^j)^\eps\partial_ju\cdot\nabla V\,dx.
 \label{eq:flux-extension-exact}
\end{equation}

Lemma~\ref{lem:periodic-potentials}\textup{(c)}, applied for each $j$ to
$G=B_0^j-\Ah e_j$ with $w=\partial_ju$ and $\psi=V$, and summed over
$j=1,\ldots,N$, gives
\begin{equation}
 \abs{
 \sum_{j=1}^N\int_\Om(B_0^j)^\eps\partial_ju\cdot\nabla V\,dx
 -\int_\Om\Ah\nabla u\cdot\nabla V\,dx}
 \le C\eps\norm{u}_{H^2(\Om)}\norm{V}_{H^1(\Om)}.
 \label{eq:B-estimate}
\end{equation}
Finally, $\aep(\beta_\eps[u],v)=0$. Combining
\eqref{eq:a-Rint-plus}, \eqref{eq:flux-extension-exact}, and
\eqref{eq:B-estimate} gives \eqref{eq:Aplus}. For $V=\Peps v$,
\eqref{eq:extension} replaces the last factor by
$C\norm{v}_{V_\eps}$.
\end{proof}

Together these estimates give the regular-branch residual.

\begin{proposition}
\label{prop:Qstar-plus}
Assume \textup{(H1)--(H4)}.
Let $u\in H^2(\Om)\cap H_0^1(\Om)$, $v\in V_\eps$,
$V=\Peps v$, and $|\lambda|\le R$. Then
\begin{equation}
\aep(\mathcal R_\eps^+u,v)
 -\lambda\be(\mathcal R_\eps^+u,v) =
 \int_\Om\Ah\nabla u\cdot\nabla V\,dx
 -\lambda m_{\mathrm{eff}}\int_\Om uV\,dx
 +\mathcal E_\eps^+(u,v),
 \label{eq:Qstar-plus}
\end{equation}
where
\begin{equation}
 \abs{\mathcal E_\eps^+(u,v)}
 \le C_R\eps\norm{u}_{H^2(\Om)}\norm{v}_{V_\eps}.
 \notag
\end{equation}
\end{proposition}

\begin{proof}
Subtract $\lambda$ times \eqref{eq:Mplus} from \eqref{eq:Aplus}.
\end{proof}

If $u\in E_k^0$ and
\begin{equation}
 \lambda=\lambda_k^0=\frac{\mu_k^0}{m_{\mathrm{eff}}},
 \notag
\end{equation}
the macroscopic term in \eqref{eq:Qstar-plus} vanishes. We obtain the
quasimode residual
\begin{equation}
 \abs{
 \aep(\mathcal R_\eps^+u,v)
 -\lambda_k^0\be(\mathcal R_\eps^+u,v)}
 \le C_k\eps\norm{u}_{H^2}\norm{v}_{V_\eps}.
 \label{eq:Qplus}
\end{equation}

For the quasimode argument, use the energy norm
\begin{equation}
 \norm{v}_{a_\eps}:=a_\eps(v,v)^{1/2}.
 \notag
\end{equation}
By ellipticity and Poincar\'e's inequality, $\norm{\cdot}_{a_\eps}$ and
$\norm{\cdot}_{V_\eps}$ are uniformly equivalent.

Let $\mu^0$ be a distinct eigenvalue of \eqref{eq:hom-Dirichlet-mu}, of
multiplicity $r$, and let $E_k^0$ be its eigenspace. Set
\begin{equation}
 \lambda_k^0=\frac{\mu^0}{m_{\mathrm{eff}}},
 \qquad
 \tau_k^+=\frac1{\lambda_k^0}
 =\frac{m_{\mathrm{eff}}}{\mu^0}.
 \notag
\end{equation}
Fix a basis $\{\phi_1,\ldots,\phi_r\}$ of $E_k^0$ normalized by
\begin{equation}
 m_{\mathrm{eff}}\int_\Om\phi_i\phi_j\,dx=\delta_{ij}.
 \label{eq:cluster-limit-mass-normalization}
\end{equation}
It then satisfies
\begin{equation}
 \int_\Om\Ah\nabla\phi_i\cdot\nabla\phi_j\,dx
 =\lambda_k^0\delta_{ij}.
 \notag
\end{equation}
Introduce the basis normalized for the limit energy,
\begin{equation}
 \varphi_j:=\frac{\phi_j}{\sqrt{\lambda_k^0}},
 \qquad j=1,\ldots,r.
 \notag
\end{equation}
Then $a_0(\varphi_i,\varphi_j)=\delta_{ij}$.

Define the quasimode space
\begin{equation}
 \mathcal Q_{\eps,k}^+
 :=\mathcal R_\eps^+E_k^0.
 \notag
\end{equation}
Set
\begin{equation}
 r_{\eps,j}:=\mathcal R_\eps^+\varphi_j,
 \qquad
 G_{\eps,k}^+
 :=\bigl(\aep(r_{\eps,i},r_{\eps,j})\bigr)_{1\le i,j\le r}.
 \notag
\end{equation}

We next verify uniform reconstruction nondegeneracy.

\begin{lemma}
\label{lem:R-injective}
Assume \textup{(H1)}, \textup{(H2)}, $\mathrm{(H3}_A\mathrm{)}$, and
\textup{(H4)}, and $m_{\mathrm{eff}}>0$, with the eigenspace $E_k^0$ fixed as above. There
exists $C_k>0$ such that
\begin{equation}
 \norm{G_{\eps,k}^+-I_r}_{\R^{r\times r}}
 \le C_k\eps.
 \label{eq:Gram-rate}
\end{equation}
There exist $\eps_k>0$ and constants $0<c_k<C_k$ such that, for
$0<\eps<\eps_k$ and $\phi\in E_k^0$,
\begin{equation}
 c_k\norm{\phi}_{H^2(\Om)}
 \le\norm{\mathcal R_\eps^+\phi}_{a_\eps}
 \le C_k\norm{\phi}_{H^2(\Om)}.
 \label{eq:R-equivalence}
\end{equation}
In particular,
\begin{equation}
 \dim\mathcal Q_{\eps,k}^+=r.
 \notag
\end{equation}
\end{lemma}

\begin{proof}
Fix $j$ and use the extension of the reconstructor,
\begin{equation}
 \widehat r_{\eps,j}:=P_\eps r_{\eps,j}\in H_0^1(\Omega).
 \notag
\end{equation}
Lemma~\ref{lem:common-reconstructor-extension}, applied with $B=A$ and
$\Lambda=0$, together with \eqref{eq:Rplus-H1-bound} and the uniform
extension estimate, gives
\begin{equation}
 \norm{\widehat r_{\eps,j}-\varphi_j}_{L^2(\Om)}\le C_k\eps,
 \qquad
 \norm{\widehat r_{\eps,j}}_{H^1(\Om)}\le C_k.
 \label{eq:special-extension-Gram-bounds}
\end{equation}

Apply Proposition~\ref{prop:diffusion-residual} with
$u=\varphi_i$, $v=r_{\eps,j}$, and the extension $V=\widehat r_{\eps,j}$. We get
\begin{equation}
 \aep(r_{\eps,i},r_{\eps,j})
 =\int_\Om\Ah\nabla\varphi_i\cdot\nabla \widehat r_{\eps,j}\,dx
 +O_k(\eps) =\mu^0\int_\Om\varphi_i\widehat r_{\eps,j}\,dx+O_k(\eps).
 \notag
\end{equation}
The second equality uses the limit eigenvalue equation. By
\eqref{eq:special-extension-Gram-bounds} and the energy normalization,
\begin{equation}
 \mu^0\int_\Om\varphi_i\widehat r_{\eps,j}\,dx
 =\mu^0\int_\Om\varphi_i\varphi_j\,dx+O_k(\eps)
 =\delta_{ij}+O_k(\eps).
 \label{eq:Gram-limit}
\end{equation}
The $r^2$ identities \eqref{eq:Gram-limit} yield \eqref{eq:Gram-rate}.

If $u=\sum_jc_j\varphi_j$, then
\begin{equation}
 \norm{\mathcal R_\eps^+u}_{a_\eps}^2
 =c^TG_{\eps,k}^+c.
 \notag
\end{equation}
The eigenvalues of $G_{\eps,k}^+$ lie in
$[1-C_k\eps,1+C_k\eps]$. Since $E_k^0$ is finite dimensional, the Euclidean
norm of the coefficients, the $H^2(\Om)$ norm, and the limit energy norm are
equivalent. This proves \eqref{eq:R-equivalence}, injectivity, and
$\dim\mathcal Q_{\eps,k}^+=r$.
\end{proof}

For $\phi\in E_k^0$, \eqref{eq:Qplus} and the definition of $\Keps$ give
\begin{equation}
 \norm{
 \mathcal R_\eps^+\phi
 -\lambda_k^0\Keps\mathcal R_\eps^+\phi}_{a_\eps}
 \le C_k\eps\norm{\phi}_{H^2}.
 \notag
\end{equation}
Since $\tau_k^+=1/\lambda_k^0$, Lemma~\ref{lem:R-injective} gives, for
every $z\in\mathcal Q_{\eps,k}^+$,
\begin{equation}
 \norm{(\Keps-\tau_k^+I)z}_{a_\eps}
 \le C_k\eps\norm{z}_{a_\eps}.
 \label{eq:quasimode}
\end{equation}

For the fixed-cluster estimates, set
\[
 \Sigma_0^+=\{m_{\mathrm{eff}}/\mu_j^0:j\ge1\},\qquad
 d_k^+=\dist\bigl(\tau_k^+,(\{0\}\cup\Sigma_0^+)\setminus\{\tau_k^+\}\bigr)>0,
\]
and choose
$I_k^+=(\tau_k^+-d_k^+/3,\tau_k^++d_k^+/3)\Subset(0,+\infty)$.
Let $\Pi_{\eps,k}^+=\ind_{I_k^+}(\Keps)$ and let $Q_{\eps,k}^+$ be the
$a_\eps$-orthogonal projector onto $\mathcal Q_{\eps,k}^+$.

These estimates yield the $O(\eps)$ regular-cluster result.

\begin{theorem}
\label{thm:quantitative-positive-cluster}
Assume \textup{(H1)--(H4)}, $m_{\mathrm{eff}}>0$, and let
$\lambda_k^0=\cdots=\lambda_{k+r-1}^0=\mu_k^0/m_{\mathrm{eff}}$ be an
isolated regular positive cluster. Choose the limit basis as in
\eqref{eq:cluster-limit-mass-normalization} and microscopic eigenfunctions
with $b_\eps(u_{\eps,k+i-1}^+,u_{\eps,k+j-1}^+)=\delta_{ij}$. Then, for all
small $\eps$, the cluster has rank $r$ and there exists a family
$(\mathsf O_{\eps,k}^+)_{0<\eps<1}\subset O(r)$, with
$\mathsf O_{\eps,k}^+=(\mathsf o_{ij}^{\eps,+})$, such that
\begin{align}
 \|\Pi_{\eps,k}^+-Q_{\eps,k}^+\|_{\mathcal L(V_\eps,a_\eps)}
 &\le C_k\eps,\qquad \max_{1\le j\le r}|\lambda_{\eps,k+j-1}^+-\lambda_k^0|
 \le C_k\eps,\notag\\
 \max_{1\le i\le r}
 \left\|u_{\eps,k+i-1}^+-\sum_{j=1}^r\mathsf o_{ij}^{\eps,+}
 \mathcal R_\eps^+\phi_{k+j-1}\right\|_{a_\eps}
 &\le C_k\eps,\qquad \max_{1\le i\le r}
 \left\|u_{\eps,k+i-1}^+-\sum_{j=1}^r\mathsf o_{ij}^{\eps,+}
 \phi_{k+j-1}\right\|_{L^2(\OmE)}
 \le C_k\eps.\notag
\end{align}
For the extensions $P_\eps u_{\eps,k+i-1}^+$ one has
\begin{equation}
 \max_{1\le i\le r}
 \left\|P_\eps u_{\eps,k+i-1}^+
 -\sum_{j=1}^r\mathsf o_{ij}^{\eps,+}\phi_{k+j-1}\right\|_{L^2(\Om)}
 \le C_k\eps.
 \label{eq:eigenfunction-extension-L2-rate}
\end{equation}
The constant $C_k$ depends only on the fixed cluster data and its
isolation distance.
\end{theorem}

\begin{proof}
Lemma~\ref{lem:R-injective} gives injectivity and the Gram estimate, while
\eqref{eq:quasimode} gives the uniform residual on
$\mathcal Q_{\eps,k}^+$. Non-pollution is
\eqref{eq:qual-positive-cross-mass-micro}--\eqref{eq:qual-positive-cross-mass-limit}.
Apply Theorem~\ref{thm:abstract-fixed-cluster} with
$H_\eps=(V_\eps,a_\eps)$, $\mathcal A_\eps=\Keps$,
$J_\eps=\mathcal R_\eps^+|_{E_k^0}$ and
$\tau_0=1/\lambda_k^0$. It gives the rank, projector bound, reciprocal
eigenvalue estimate and orthogonal alignment. Returning from the
energy-normalized bases gives the third estimate. Poincar\'e's inequality and
\eqref{eq:Rplus-L2} give the physical-domain $L^2$ estimate.  Applying
$P_\eps$ to the energy error and using Lemma~\ref{lem:common-reconstructor-extension}
for the aligned reconstructors gives \eqref{eq:eigenfunction-extension-L2-rate}.
\end{proof}

The fast negative branch is treated by removing its $\eps^{-2}$ scale
through the first negative cell mode.


\subsection{Negative ground state and exact factorization}
\label{subsec:negative-ground-state}

Factorization through a cell ground state is a classical device in spectral
homogenization
\cite{AllaireCapdeboscqPiatnitskiSiessVanninathan2004,AllairePiatnitski2005,NazarovPankratovaPiatnitski2011}.
We extend it here to the combined bulk--surface mass, for which the cell
ground state must satisfy a spectral Robin condition on $\Gam$.

Assume now $m_{\mathrm{eff}}>0$. Consider the combined cell problem
\begin{equation}
 \int_{\Ys}A\nabla_y\theta\cdot\nabla_y\psi\,dy
 =\lambda\mathfrak b_Y(\theta,\psi)
 \qquad\forall\psi\in H^1_\#(\Ys).
 \label{neg:eq:cell-spectral}
\end{equation}
Equivalently,
\begin{equation}
\left\{
\begin{aligned}
 -\Div_y(A\nabla_y\theta)&=\lambda q\theta&&\text{in }\Ys,\\
 A\nabla_y\theta\cdot n&=\lambda\rho\theta&&\text{on }\Gam,\\
 \theta&\text{ is }Y\text{-periodic.}
\end{aligned}
\right.
 \notag
\end{equation}
The pair $(0,1)$ is an eigencouple, and Proposition~\ref{prop:prelim-two-sign-spectrum} provides a negative direction.
Since \cite{HessKato1980} does not cover spectral Robin mass, we first establish the positivity needed for factorization.

\begin{lemma}
\label{neg:lem:strict-positive-cell}
Assume \textup{(H1)--(H3)}. Let
$\lambda\in\mathbb R$ and let
$0\le\theta\in H^1_\#(\Ys)$, $\theta\not\equiv0$, satisfy
\begin{equation}
 \int_{\Ys}A\nabla\theta\cdot\nabla\psi\,dy
 =\lambda\left(\int_{\Ys}q\theta\psi\,dy
 +\int_\Gamma\rho\theta\psi\,dS_y\right)
 \qquad\forall\psi\in H^1_\#(\Ys).
 \label{neg:eq:positive-cell-generic}
\end{equation}
Then, for every $0<\alpha'<\alpha_Y$,
$\theta\in C^{1,\alpha'}(\overline{\Ys})$ and
\begin{equation}
 \theta>0\quad\hbox{on }\overline{\Ys}.
 \label{neg:eq:strict-positive-cell}
\end{equation}
Consequently there are constants $c_\theta,C_\theta>0$ such that
$c_\theta\le\theta\le C_\theta$ on $\overline{\Ys}$.
\end{lemma}

\begin{proof}
Lemma~\ref{lem:cell-regularity-H4}\textup{(ii)} gives
$\theta\in C^{1,\alpha'}(\overline{\Ys})$.  In the interior it satisfies
$-\operatorname{div}(A\nabla\theta)=V\theta$ with
$V=\lambda q\in L^\infty$.  The Harnack inequality for nonnegative solutions
with bounded zeroth-order coefficient, followed along a Harnack chain in the
connected periodic solid, implies $\theta>0$ in $\Ys$; see, for example,
\cite[Chapters~3 and~8]{GilbargTrudinger2001}.

Suppose that $\theta(y_0)=0$ at some $y_0\in\Gamma$.  Put
$V=\lambda q$ and write $V=V^+-V^-$, with $V^\pm\ge0$. Define
\[
 L\theta:=\operatorname{div}(A\nabla\theta)-V^-\theta.
\]
Then the interior equation gives
\[
 L\theta=-V^+\theta\le0\qquad\text{in }Y^*,
\]
so the zeroth-order coefficient of $L$ is $-V^-\le0$. Since
$\theta>0$ in $Y^*$ and $\theta(y_0)=0$, the point $y_0$ is a strict
boundary minimum. After flattening the $C^{2,\alpha_Y}$ boundary, the Hopf
boundary point lemma for weak uniformly elliptic divergence-form equations
with H\"older continuous principal coefficients applies to $L$ and yields a
strictly nonzero outward normal derivative at $y_0$; see \cite[Theorem~4.1]{Rosales2019}
and the survey \cite{ApushkinskayaNazarov2022}, where H\"older (indeed Dini)
continuity of the principal coefficients is shown to suffice for
divergence-form operators.  Since
$\theta|_\Gamma\ge0$ attains its minimum $0$ at $y_0$ and
$\theta\in C^1(\overline{\Ys})$, the tangential gradient of $\theta$
vanishes at $y_0$. Hence $\nabla\theta(y_0)$ is normal
to $\Gamma$, and uniform ellipticity implies
\[
 A\nabla\theta(y_0)\cdot n(y_0)\ne0.
\]
On the other hand, the Robin condition contained in
\eqref{neg:eq:positive-cell-generic} gives
\[
 A\nabla\theta(y_0)\cdot n(y_0)
 =\lambda\rho(y_0)\theta(y_0)=0,
\]
a contradiction.  Hence \eqref{neg:eq:strict-positive-cell} holds; the
uniform upper and lower bounds follow from continuity and compactness of the
periodic cell closure.
\end{proof}

We can now select the negative cell ground state.

\begin{proposition}
\label{neg:prop:ground-state}
Assume \textup{(H1)--(H3)}, and
$m_{\mathrm{eff}}>0$. Then the cell problem
\eqref{neg:eq:cell-spectral} possesses a first negative eigenvalue
$\lambdaone<0$, characterized by
\begin{equation}
 -\lambdaone
 =\min_{\substack{\phi\in H^1_\#(\Ys)\\ \mathfrak b_Y(\phi,\phi)<0}}
 \frac{\displaystyle\int_{\Ys}A\nabla\phi\cdot\nabla\phi\,dy}
 {\displaystyle-\mathfrak b_Y(\phi,\phi)}.
 \label{neg:eq:cell-Rayleigh}
\end{equation}
It is simple and its eigenfunction may be chosen strictly positive. We normalize it by
\begin{equation}
 \int_{\Ys}(\thetaone)^2\,dy=1.
 \label{neg:eq:theta-normalization}
\end{equation}
There exist $c_\theta,C_\theta>0$ such that
\begin{equation}
 0<c_\theta\le\thetaone(y)\le C_\theta<\infty
 \qquad\text{on }\overline{\Ys}.
 \label{neg:eq:theta-bounds}
\end{equation}
Moreover,
\begin{equation}
 \int_{\Ys}q(\thetaone)^2\,dy
 +\int_\Gam\rho(\thetaone)^2\,dS_y
 =\frac1{\lambdaone}\int_{\Ys}A\nabla\thetaone\cdot\nabla\thetaone\,dy<0.
 \label{neg:eq:theta-energy}
\end{equation}
\end{proposition}

\begin{proof}
Set
\[
 a_Y(\phi,\psi):=\int_{\Ys}A\nabla\phi\cdot\nabla\psi\,dy
 \quad(\phi,\psi\in H^1_\#(\Ys)),
 \qquad b_0:=\mathfrak b_Y(1,1)=m_{\mathrm{eff}}>0.
\]

We first show that the Rayleigh minimum is attained. Normalize a minimizing sequence by
$\mathfrak b_Y(\phi_n,\phi_n)=-1$. Then $a_Y(\phi_n,\phi_n)$ is bounded.
Write
\[
 \phi_n=m_n+w_n,\qquad
 m_n:=\frac1{|\Ys|}\int_{\Ys}\phi_n\,dy,\qquad
 \int_{\Ys}w_n\,dy=0.
\]
The periodic Poincar\'e inequality and ellipticity show that $(w_n)$ is
bounded in $H^1_\#(\Ys)$. Boundedness of the volume and trace parts of
$\mathfrak b_Y$ gives
\[
 -1=b_0m_n^2+2m_n\mathfrak b_Y(1,w_n)
       +\mathfrak b_Y(w_n,w_n).
\]
Consequently,
\[
 b_0m_n^2\le 1+C|m_n|\|w_n\|_{H^1(\Ys)}
                    +C\|w_n\|_{H^1(\Ys)}^2,
\]
and Young's inequality bounds $(m_n)$. Hence $(\phi_n)$ is bounded in
$H^1_\#(\Ys)$. Up to a subsequence, it converges weakly in $H^1$ and
strongly both in $L^2(\Ys)$ and, by compactness of the trace, in
$L^2(\Gam)$. Thus $\mathfrak b_Y(\phi_n,\phi_n)$ converges to the mass of the
weak limit, while $a_Y$ is weakly lower semicontinuous. The limit
$\thetaone$ attains the minimum and may be rescaled later to satisfy
\eqref{neg:eq:theta-normalization}.

The minimum is strictly positive. Indeed, zero energy would force
$\thetaone$ to be constant on the connected set $\Ys$, which is incompatible
with $\mathfrak b_Y(\thetaone,\thetaone)=-1$ because $b_0>0$. The
Euler--Lagrange equation on this constraint therefore reads
\[
 a_Y(\thetaone,\psi)=\lambdaone\mathfrak b_Y(\thetaone,\psi)
 \quad\forall\psi\in H^1_\#(\Ys),
 \qquad \lambdaone=-a_Y(\thetaone,\thetaone)<0,
\]
and proves the variational characterization \eqref{neg:eq:cell-Rayleigh}.

We next show that the minimizer is strictly positive. Since
$a_Y(|\phi|,|\phi|)=a_Y(\phi,\phi)$ and
$\mathfrak b_Y(|\phi|,|\phi|)=\mathfrak b_Y(\phi,\phi)$, the minimizer may
be chosen nonnegative.  Lemma~\ref{neg:lem:strict-positive-cell}, applied to
the Euler--Lagrange equation obtained above, gives
$\thetaone\in C^{1,\alpha'}(\overline{\Ys})$ for every
$0<\alpha'<\alpha_Y$, strict positivity on the closed cell, and hence
\eqref{neg:eq:theta-bounds}.  In particular, both $\thetaone$ and
$1/\thetaone$ are bounded, which is the precise property needed by the
global ground-state transform below.

Next, simplicity follows from Picone's identity. Let $w$ be any real eigenfunction associated with $\lambdaone$. The same
elliptic regularity gives
$w\in C^{1,\alpha'}(\overline{\Ys})$ for every
$0<\alpha'<\alpha_Y$; together
with the lower bound in \eqref{neg:eq:theta-bounds}, this makes
$w^2/\thetaone$ an admissible $H^1_\#(\Ys)$ test function. Testing
the equation for $\thetaone$ with $w^2/\thetaone$ and using
\(\thetaone(w^2/\thetaone)=w^2\) in both the volume and surface mass
terms gives the exact identity
\begin{equation}
 a_Y(w,w)-\lambdaone\mathfrak b_Y(w,w)
 =\int_{\Ys}(\thetaone)^2
 A\nabla\!\left(\frac{w}{\thetaone}\right)\cdot
   \nabla\!\left(\frac{w}{\thetaone}\right)\,dy.
 \notag
\end{equation}
The left-hand side vanishes because $w$ has the same eigenvalue.
Ellipticity and connectedness imply
$\nabla(w/\thetaone)=0$, hence $w=C\thetaone$. Thus the eigenspace is
one-dimensional; the surface mass cancels in the Picone test through the
Robin part of $\mathfrak b_Y$.

Finally, testing \eqref{neg:eq:cell-spectral} with $\thetaone$ gives
\eqref{neg:eq:theta-energy}; its mass factor is negative because the energy
is positive and $\lambdaone<0$.
\end{proof}

Set
\begin{equation}
 \Theta_\eps:=(\thetaone)^\eps.
 \notag
\end{equation}

The ground state yields exact factorization.

\begin{theorem}
\label{neg:thm:global-factorization}
Assume \textup{(H1)--(H3)} and $m_{\mathrm{eff}}>0$. For all $v,w\in V_\eps$,
\begin{equation}
 a_\eps(\Theta_\eps v,\Theta_\eps w)
 -\frac{\lambdaone}{\eps^2} b_\eps(\Theta_\eps v,\Theta_\eps w)
 =\int_{\OmE}((\thetaone)^2A)^\eps\nabla v\cdot\nabla w\,dx.
 \tag{GF}\label{neg:eq:global-factorization}
\end{equation}
Define
\begin{equation}
 \tA=(\thetaone)^2A,
 \qquad \tq=(\thetaone)^2q,
 \qquad \trho=(\thetaone)^2\rho,
 \notag
\end{equation}
and
\begin{equation}
 q^-:=-\tq,\qquad \rho^-:=-\trho,
 \qquad
 \mminus:=-\left(\int_{\Ys}\tq\,dy+\int_\Gam\trho\,dS_y\right)>0.
 \label{neg:eq:minus-data}
\end{equation}
The factorized effective mass is
\begin{equation}
 \tmeff:=\int_{\Ys}\tq\,dy+\int_\Gam\trho\,dS_y=-\mminus<0.
 \notag
\end{equation}
Both $\tA$ and $\tmeff$ scale like the square of the normalization constant
in \eqref{neg:eq:theta-normalization}; hence the homogenized tensor of $\tA$
and $\tmeff$ scale in the same way, and the limit ratios
$\widetilde\mu_k^0/\tmeff$ below do not depend on this normalization.
If
\begin{align}
 \widetilde a_\eps(v,w)&:=\int_{\OmE}\tA^\eps\nabla v\cdot\nabla w\,dx,\qquad \widetilde b_\eps(v,w)
 :=b_\eps(\Theta_\eps v,\Theta_\eps w)
 =\int_{\OmE}\tq^\eps vw\,dx
  +\eps\int_{\GamE}\trho^\eps vw\,dS_x,\notag\\
 b_\eps^-(v,w)&:=-\widetilde b_\eps(v,w)
 =\int_{\OmE}(q^-)^\eps vw\,dx
  +\eps\int_{\GamE}(\rho^-)^\eps vw\,dS_x,\notag
\end{align}
for $v,w\in V_\eps$, then multiplication $u=\Theta_\eps v$ establishes a bijection, preserving multiplicities, between negative eigenvalues $\lambda_\eps$ of the original problem and positive eigenvalues $\eta_\eps$ of
\begin{equation}
 \widetilde a_\eps(v_\eps,w)=\eta_\eps b_\eps^-(v_\eps,w)
 \qquad\forall w\in V_\eps.
 \tag{$\mathscr P_\eps^-$}\label{neg:eq:positive-transformed}
\end{equation}
The parameters satisfy
\begin{equation}
 \lambda_\eps=\frac{\lambdaone}{\eps^2}-\eta_\eps,
 \notag
\end{equation}
so that
\begin{equation}
 \operatorname{Spec}(\mathscr P_\eps)\cap(\frac{\lambdaone}{\eps^2},0)=\varnothing,
 \qquad
 \lambda_{\eps,k}^-=\frac{\lambdaone}{\eps^2} -\eta_{\eps,k}^+,
 \quad k\ge1.
 \label{neg:eq:spectral-correspondence}
\end{equation}
\end{theorem}

\begin{proof}
It suffices to prove (GF) for $v,w\in\mathcal D_\eps$. Expanding the
gradients gives
\[
 a_\eps(\Theta_\eps v,\Theta_\eps w)
 =\int_{\OmE}\Theta_\eps^2A^\eps\nabla v\cdot\nabla w\,dx
 +\int_{\OmE}A^\eps\nabla\Theta_\eps\cdot\nabla(\Theta_\eps vw)\,dx.
\]
Rescaling \eqref{neg:eq:cell-spectral}, the periodic function $\Theta_\eps$
satisfies $-\Div(A^\eps\nabla\Theta_\eps)=\lambdaone\eps^{-2}q^\eps\Theta_\eps$
in the whole periodic solid $\eps\mathcal W$, with
$A^\eps\nabla\Theta_\eps\cdot n=\lambdaone\eps^{-1}\rho^\eps\Theta_\eps$ on
every periodic hole. The test function $\Theta_\eps vw$ vanishes in a
neighborhood of $\Dep$, which contains $\partial\Om\cap\overline{\OmE}$ and
every hole not in $\GamE$; hence only the holes of $\GamE$ contribute, and
the last integral equals $\lambdaone\eps^{-2}b_\eps(\Theta_\eps v,\Theta_\eps w)$.
This is \eqref{neg:eq:global-factorization} for $v,w\in\mathcal D_\eps$;
density of $\mathcal D_\eps$ in $V_\eps$, together with bounded multiplication
by $\Theta_\eps$ for fixed $\eps$, extends it to $V_\eps$.

The bounds \eqref{neg:eq:theta-bounds} and the cell regularity imply that
$\Theta_\eps,\Theta_\eps^{-1}\in W^{1,\infty}(\OmE)$ for each fixed
$\eps$.  Hence multiplication by $\Theta_\eps$ is a linear isomorphism
of $V_\eps$, so the change of unknown cannot alter eigenspace dimensions.

The identity \eqref{neg:eq:theta-energy} gives
\[
 \left(\int_{\Ys}\tq\,dy+\int_\Gam\trho\,dS_y\right)
 =\frac1{\lambdaone}\int_{\Ys}A\nabla\thetaone\cdot\nabla\thetaone\,dy<0,
\]
hence $\mminus>0$. If $u_\eps=\Theta_\eps v_\eps$ and $\lambda_\eps=\frac{\lambdaone}{\eps^2}+\xi_\eps$, \eqref{neg:eq:global-factorization} yields
\[
 \widetilde a_\eps(v_\eps,w)=\xi_\eps\widetilde b_\eps(v_\eps,w)
 =-\xi_\eps b_\eps^-(v_\eps,w)
 \qquad\forall w\in V_\eps.
\]
Thus $\eta_\eps=-\xi_\eps$. Conversely, any positive eigencouple of \eqref{neg:eq:positive-transformed} gives an original eigencouple with $\lambda_\eps=\frac{\lambdaone}{\eps^2} -\eta_\eps<0$. If an original eigenvalue is negative, testing with its eigenfunction shows $b_\eps(u_\eps,u_\eps)<0$; applying \eqref{neg:eq:global-factorization} then forces $\lambda_\eps<\frac{\lambdaone}{\eps^2}$. Hence no negative eigenvalue is omitted and the correspondence is complete.
\end{proof}

The factorization yields the following fixed-index bounds.

\begin{proposition}
\label{prop:prelim-negative-fixed-index-bound}
Assume \textup{(H1)--(H3)} and $m_{\mathrm{eff}}>0$.
For every fixed $k\ge1$ there exist
$0<c_{-,*}\le C_{-,k}<\infty$ and $\eps_k>0$ such that, for
$j=1,\ldots,k$ and $0<\eps<\eps_k$,
\begin{equation}
 c_{-,*}
 \le \frac{\lambdaone}{\eps^2}-\lambda_{\eps,j}^-
 =\eta_{\eps,j}^+
 \le C_{-,k}.
 \label{eq:prelim-negative-shift-bound}
\end{equation}
Equivalently,
\begin{equation}
 \frac{\lambdaone}{\eps^2}-C_{-,k}
 \le\lambda_{\eps,j}^-
 \le\frac{\lambdaone}{\eps^2}-c_{-,*}.
 \label{eq:prelim-negative-location}
\end{equation}
If $v_{\eps,j}$ is normalized by $b_\eps^-(v_{\eps,j},v_{\eps,j})=1$ and $u_{\eps,j}^-=\Theta_\eps v_{\eps,j}$, then
\begin{equation}
 \|v_{\eps,j}\|_{V_\eps}
 +\|P_\eps v_{\eps,j}\|_{H_0^1(\Om)}
 +\eps^{1/2}\|v_{\eps,j}\|_{L^2(\GamE)}
 \le C_{-,k},
 \label{eq:prelim-negative-factorized-bound}
\end{equation}
and, for the shifted energy norm
\begin{equation}
 \|z\|_{\Theta,\eps}^2
 :=a_\eps(z,z)-\frac{\lambdaone}{\eps^2}b_\eps(z,z),
 \notag
\end{equation}
one has
\begin{equation}
 \|u_{\eps,j}^-\|_{\Theta,\eps}^2
 =\eta_{\eps,j}^+\le C_{-,k}.
 \label{eq:prelim-negative-original-shifted-bound}
\end{equation}
\end{proposition}

\begin{proof}
By \eqref{neg:eq:theta-bounds}, $\thetaone$ and its reciprocal are bounded; hence $\tA$ is uniformly elliptic. The transformed effective mass is $\mminus>0$. Apply the fixed-index argument of Proposition~\ref{prop:prelim-positive-fixed-index}, that is, Lemma~\ref{lem:signed-courant-fischer} with macroscopic trial spaces, to the transformed forms $(\widetilde a_\eps,b_\eps^-)$: the same finite-dimensional macroscopic trial spaces give $\eta_{\eps,j}^+\le C_{-,k}$ for $j\le k$, while uniform boundedness of the transformed compact operator excludes zero and gives $\eta_{\eps,j}^+\ge c_{-,*}$. The exact correspondence then gives \eqref{eq:prelim-negative-shift-bound}--\eqref{eq:prelim-negative-location}. The normalized transformed energy identity $\widetilde a_\eps(v_{\eps,j},v_{\eps,j})=\eta_{\eps,j}^+$, ellipticity, extension, and trace estimates give \eqref{eq:prelim-negative-factorized-bound}. Finally, \eqref{neg:eq:global-factorization} with $u_{\eps,j}^-=\Theta_\eps v_{\eps,j}$ gives \eqref{eq:prelim-negative-original-shifted-bound}.
\end{proof}

\subsection{Factorized branch: homogenization and quantitative transfer}
\label{neg:sec:homogenized}

The matrix $\tA=(\thetaone)^2A$ is symmetric and uniformly elliptic in
$\mathcal W$. For $j=1,\ldots,N$, let
$\widetilde\chi^j\in H^1_\#(\Ys)/\R$ solve
\begin{equation}
 \int_{\Ys}\tA(y)
 (e_j+\nabla_y\widetilde\chi^j)\cdot\nabla_y\psi\,dy=0
 \qquad\forall\psi\in H^1_\#(\Ys)/\R.
 \notag
\end{equation}
Set
\begin{equation}
 \widetilde B^j:=\tA(e_j+\nabla_y\widetilde\chi^j).
 \notag
\end{equation}
Then
\begin{equation}
 \Div_y\widetilde B^j=0\quad\hbox{in }\Ys,
 \qquad
 \widetilde B^j\cdot n=0\quad\hbox{on }\Gam.
 \notag
\end{equation}
The homogenized matrix is
\begin{equation}
 \tAh e_j
 =\int_{\Ys}\widetilde B^j(y)\,dy.
 \notag
\end{equation}
It is symmetric and positive definite.

The limit problem for \eqref{neg:eq:positive-transformed} is
\begin{equation}
 \int_\Om\tAh\nabla u\cdot\nabla\varphi\,dx
 =\eta\mminus\int_\Om u\varphi\,dx
 \qquad\forall\varphi\in H_0^1(\Om).
 \notag
\end{equation}
Let
\begin{equation}
 0<\widetilde\mu_1^0\le\widetilde\mu_2^0\le\cdots\to+\infty
 \notag
\end{equation}
be the spectrum of
\begin{equation}
 -\Div(\tAh\nabla u)=\widetilde\mu u\quad\hbox{in }\Om,
 \qquad u=0\quad\hbox{on }\partial\Om.
 \label{neg:eq:tilde-hom-dirichlet}
\end{equation}
Then
\begin{equation}
 \eta_k^0=\frac{\widetilde\mu_k^0}{\mminus},
 \qquad
 \xi_k^0=-\eta_k^0
 =\frac{\widetilde\mu_k^0}{\tmeff}.
 \notag
\end{equation}

Let $\widetilde\mu_k^0$ be a \emph{distinct} eigenvalue of
\eqref{neg:eq:tilde-hom-dirichlet}, of multiplicity $r$, occupying the
indices $k,\ldots,k+r-1$. We write
\begin{equation}
 E_k^-:=\left\{\phi\in H_0^1(\Om):
 -\Div(\tAh\nabla\phi)=\widetilde\mu_k^0\phi\right\},
 \qquad \dim E_k^-=r,
 \notag
\end{equation}
and fix a basis $\{\phi_1,\ldots,\phi_r\}$ such that
\begin{equation}
 \mminus\int_\Om\phi_i\phi_j\,dx=\delta_{ij}.
 \notag
\end{equation}
It then satisfies
\begin{equation}
 \int_\Om\tAh\nabla\phi_i\cdot\nabla\phi_j\,dx
 =\eta_k^0\delta_{ij},
 \qquad
 \eta_k^0=\frac{\widetilde\mu_k^0}{\mminus}.
 \notag
\end{equation}

We now homogenize the transformed problem.

\begin{theorem}
\label{neg:thm:qualitative}
Assume \textup{(H1)--(H3)} and $m_{\mathrm{eff}}>0$. Then, for every fixed $k\ge1$,
\begin{equation}
 \eta_{\eps,k}^+\longrightarrow
 \eta_k^0=\frac{\widetilde\mu_k^0}{\mminus}.
 \notag
\end{equation}
Consequently,
\begin{equation}
 \lambda_{\eps,k}^-
 -\frac{\lambdaone}{\eps^2}
 \longrightarrow
 \frac{\widetilde\mu_k^0}{\tmeff}
 =-\frac{\widetilde\mu_k^0}{\mminus}.
 \label{neg:eq:lambda-qualitative}
\end{equation}
More precisely, let $\mathcal J_k^-\Subset(-\infty,0)$ be any open interval such that
\begin{equation}
 -\eta_k^0\in \mathcal J_k^-,
 \qquad
 \overline{\mathcal J_k^-}\cap\{-\eta_j^0:j\ge1\}=\{-\eta_k^0\}.
 \notag
\end{equation}
For all sufficiently small $\eps$, exactly $r$ values in the shifted
negative spectrum
\(
 \{\lambda_{\eps,j}^- -\lambdaone/\eps^2:j\ge1\}
\), counted with multiplicity, belong to $\mathcal J_k^-$. Thus multiplicities are
preserved, and no additional shifted negative eigenvalue enters the isolated
limit cluster.

For the cluster $k,\ldots,k+r-1$, choose transformed eigenfunctions
$v_{\eps,k},\ldots,v_{\eps,k+r-1}$ such that
\begin{equation}
 b_\eps^-(v_{\eps,k+i-1},v_{\eps,k+j-1})=\delta_{ij},
 \qquad 1\le i,j\le r,
 \notag
\end{equation}
and set
\begin{equation}
 u_{\eps,k+i-1}^-:=\Theta_\eps v_{\eps,k+i-1}.
 \notag
\end{equation}
Thus
$b_\eps(u_{\eps,k+i-1}^-,u_{\eps,k+j-1}^-)=-\delta_{ij}$.
There exists a family of orthogonal matrices
$(\mathsf O_{\eps,k}^-)_{0<\eps<1}\subset O(r)$, with
$\mathsf O_{\eps,k}^-=(\mathsf o_{ij}^{\eps,-})_{1\le i,j\le r}$, such that
\begin{equation}
 \max_{1\le i\le r}\left\|P_\eps v_{\eps,k+i-1}-\sum_{j=1}^{r}\mathsf o_{ij}^{\eps,-}\phi_j\right\|_{L^2(\Om)}
 \longrightarrow0,\qquad \max_{1\le i\le r}\left\|u_{\eps,k+i-1}^--\Theta_\eps\sum_{j=1}^{r}\mathsf o_{ij}^{\eps,-}\phi_j\right\|_{L^2(\OmE)}
 \longrightarrow0.
 \label{neg:eq:qual-eigenvectors-original}
\end{equation}
After extraction of a subfamily, without relabeling, assume
$\mathsf O_{\eps,k}^-\to(\mathsf o_{ij}^-)$, and define
$\phi_i^*=\sum_{j=1}^{r}\mathsf o_{ij}^-\phi_j$.
\begin{equation}
 \ind_{\OmE}\nabla P_\eps v_{\eps,k+i-1}
 \stackrel{2}{\rightharpoonup}
 \ind_{Y^*}\left(\nabla_x\phi_i^*
 +\sum_{\ell=1}^N\nabla_y\widetilde\chi^\ell\,
 \partial_\ell\phi_i^*\right),\qquad \ind_{\OmE}P_\eps u_{\eps,k+i-1}^-
 \stackrel{2}{\rightharpoonup}
 \ind_{Y^*}\thetaone\phi_i^*.
 \label{neg:eq:qual-original-two-scale-cluster}
\end{equation}
\end{theorem}

\begin{proof}
By \eqref{neg:eq:theta-bounds} and \eqref{neg:eq:minus-data}, the transformed triple
$(\widetilde A,q^-,\rho^-)$ has uniformly elliptic diffusion and positive
combined mass $m_->0$. The coefficient-generic proof of
Theorem~\ref{thm:qual-positive}, applied to the transformed problem
\eqref{neg:eq:positive-transformed}, therefore gives
$\eta_{\eps,k}^+\to\widetilde\mu_k^0/m_-$, exact multiplicity, orthogonal
alignment and the two-scale gradient formula with the correctors
$\widetilde\chi^j$. The exact spectral correspondence
\eqref{neg:eq:spectral-correspondence} yields
\eqref{neg:eq:lambda-qualitative}. Finally,
$u_{\eps,k}^-=\Theta_\eps v_{\eps,k}$ and the bounded periodic multiplier
$\Theta_\eps$ give \eqref{neg:eq:qual-eigenvectors-original} and
\eqref{neg:eq:qual-original-two-scale-cluster}.
\end{proof}

For the quantitative transfer, under \textup{(H3)} the positive ground state
$\thetaone$ is regular, so $(\widetilde A,q^-,\rho^-)$ satisfies the same
quantitative hypotheses as the regular positive-mass problem (with any fixed
$0<\alpha'<\alpha_Y$), while $m_->0$ by
\eqref{neg:eq:minus-data}. Thus the complete positive-mass template applies under
\begin{equation}
 (A,q,\rho,m_{\mathrm{eff}},\chi^j,\Ah,a_\eps,b_\eps,\Keps)
 \longmapsto
 (\tA,q^-,\rho^-,\mminus,\widetilde\chi^j,\tAh,
  \widetilde a_\eps,b_\eps^-,\widetilde{\mathcal K}_\eps),
 \label{neg:eq:template-substitution}
\end{equation}
where $\widetilde a_\eps(\widetilde{\mathcal K}_\eps f,v)=b_\eps^-(f,v)$.
The corresponding exact-layer reconstructor is
\begin{equation}
 \mathcal R_\eps^-u
 :=u+\eps\sum_{j=1}^N(\widetilde\chi^j)^\eps\partial_j u
   +\widetilde\beta_\eps[u],
 \notag
\end{equation}
with $\widetilde\beta_\eps[u]$ given by
Lemma~\ref{lem:common-boundary-layer} for $B=\widetilde A$. Consequently,
for $u\in E_k^-$,
\begin{equation}
 \|\mathcal R_\eps^-u-u\|_{L^2(\OmE)}
 \le C\eps\|u\|_{H^2(\Om)},\qquad \left|\widetilde a_\eps(\mathcal R_\eps^-u,v)
 -\eta_k^0 b_\eps^-(\mathcal R_\eps^-u,v)\right|
 \le C_k\eps\|u\|_{H^2(\Om)}\|v\|_{V_\eps},
 \notag
\end{equation}
and the transferred Gram estimate gives injectivity on
$\mathcal R_\eps^-E_k^-$ for small $\eps$.

Let
\[
 \widetilde\tau_k^0=\frac1{\eta_k^0}=\frac{\mminus}{\widetilde\mu_k^0},\qquad
 d_k^-=\dist\!\left(\widetilde\tau_k^0,
 \bigl(\{0\}\cup\{\mminus/\widetilde\mu_j^0:j\ge1\}\bigr)
 \setminus\{\widetilde\tau_k^0\}\right),
\]
set
$I_k^-=(\widetilde\tau_k^0-d_k^-/3,\widetilde\tau_k^0+d_k^-/3)\Subset(0,\infty)$,
let $\Pi_{\eps,k}^-=\ind_{I_k^-}(\widetilde{\mathcal K}_\eps)$, and let $Q_{\eps,k}^-$ be the
$\widetilde a_\eps$-orthogonal projector onto $\mathcal R_\eps^-E_k^-$.

The transformed estimates yield the cluster rate.

\begin{theorem}
\label{neg:thm:quantitative-cluster}
Assume \textup{(H1)--(H4)}, $m_{\mathrm{eff}}>0$, and let
$\widetilde\mu_k^0=\cdots=\widetilde\mu_{k+r-1}^0$ be an isolated transformed limit eigenvalue cluster of
multiplicity $r$. Then, for all small $\eps$,
\begin{align}
 \operatorname{rank}\Pi_{\eps,k}^-&=r,
 &\|\Pi_{\eps,k}^--Q_{\eps,k}^-\|_{\mathcal L(V_\eps,\widetilde a_\eps)}
 &\le C_k\eps,\notag\\
 \max_{1\le j\le r}
 \left|\eta_{\eps,k+j-1}^+-\eta_k^0\right|
 &\le C_k\eps,
 &\max_{1\le j\le r}
 \left|\lambda_{\eps,k+j-1}^-
 -\frac{\lambdaone}{\eps^2}
 -\frac{\widetilde\mu_k^0}{\tmeff}\right|
 &\le C_k\eps.\notag
\end{align}
If $\{\phi_1,\ldots,\phi_r\}$ is a transformed limit basis and the
microscopic amplitudes are mass-normalized, there exists a family
$(\mathsf O_{\eps,k}^-)_{0<\eps<1}\subset O(r)$, with
$\mathsf O_{\eps,k}^-=(\mathsf o_{ij}^{\eps,-})$, such that
\begin{equation}
 \max_{1\le i\le r}\left\|v_{\eps,k+i-1}
 -\sum_{j=1}^r\mathsf o_{ij}^{\eps,-}\mathcal R_\eps^-\phi_j
 \right\|_{\widetilde a_\eps}
 \le C_k\eps,\qquad \max_{1\le i\le r}\left\|u_{\eps,k+i-1}^-
 -\Theta_\eps\sum_{j=1}^r\mathsf o_{ij}^{\eps,-}\phi_j
 \right\|_{L^2(\OmE)}
 \le C_k\eps.
 \notag
\end{equation}
The extensions of the transformed amplitudes satisfy
\begin{equation}
 \max_{1\le i\le r}\left\|P_\eps v_{\eps,k+i-1}
 -\sum_{j=1}^r\mathsf o_{ij}^{\eps,-}\phi_j\right\|_{L^2(\Om)}
 \le C_k\eps,
 \label{neg:eq:amplitude-extension-L2-rate}
\end{equation}
and multiplication by $\Theta_\eps$ transfers the energy estimate
to the shifted original energy norm. The constant $C_k$ depends only on the fixed transformed cluster data
and its isolation distance.
\end{theorem}

\begin{proof}
The substitution \eqref{neg:eq:template-substitution} transfers the
centering identity, the exact Dirichlet-layer estimate, the flux-corrector
residual, the Gram bound and the operator quasi-mode estimate from the
regular positive template. Non-pollution and qualitative multiplicity are
already contained in Theorem~\ref{neg:thm:qualitative}.
Theorem~\ref{thm:abstract-fixed-cluster} applied to $\widetilde{\mathcal K}_\eps$ therefore
gives the rank, projector, reciprocal transformed-eigenvalue rate and
orthogonal alignment. Since
$\lambda_{\eps,j}^-=\lambdaone/\eps^2-\eta_{\eps,j}^+$ and
$-\eta_k^0=\widetilde\mu_k^0/\tmeff$, the original eigenvalue estimate follows.
Finally, \eqref{neg:eq:global-factorization} and the uniform upper and lower
bounds on $\Theta_\eps$ transfer the aligned amplitude estimates to the
original eigenfunctions.  Applying $P_\eps$ to the transformed energy error
and using Lemma~\ref{lem:common-reconstructor-extension} with
$B=\widetilde A$ yields \eqref{neg:eq:amplitude-extension-L2-rate}.
\end{proof}

\subsection[Negative effective mass by exact spectral symmetry]
{The case of negative effective mass by exact spectral symmetry}
\label{subsec:negative-effective-mass}

For $m_{\rm eff}<0$, sign symmetry transfers the analysis.

\begin{proposition}
\label{sym:prop:exact-involution}
Assume \textup{(H1)}, with real
$q\in L^\infty_\#(\Ys)$ and $\rho\in L^\infty_\#(\Gam)$. For every fixed
$\eps$,
\begin{equation}
 (\lambda,u)\text{ is an eigencouple for }(q,\rho)
 \quad\Longleftrightarrow\quad
 (-\lambda,u)\text{ is an eigencouple for }(-q,-\rho),
 \notag
\end{equation}
with the same eigenspace and multiplicity. Under
\textup{(H2)}, the ordered branches therefore satisfy
\begin{equation}
 \lambda_{\eps,k}^{+}(q,\rho)=-\lambda_{\eps,k}^{-}(-q,-\rho),
 \qquad
 \lambda_{\eps,k}^{-}(q,\rho)=-\lambda_{\eps,k}^{+}(-q,-\rho).
 \label{sym:eq:ordered-branch-exchange}
\end{equation}
\end{proposition}

\begin{proof}
Since the variational equation is $a_\eps(u,v)=\lambda b_\eps(u,v)$,
replacing $(q,\rho,\lambda)$ by $(-q,-\rho,-\lambda)$ leaves it unchanged.
The identity map preserves eigenspaces and multiplicities, and negation
exchanges the two ordered branches.
\end{proof}

\begin{corollary}[The case $m_{\mathrm{eff}}<0$]
\label{sym:cor:negative-effective-mass}
Assume \textup{(H1)--(H3)} and $m_{\mathrm{eff}}<0$. Then the regular
branch is negative, and the fast positive branch is generated by
\[
 \lambda_1^+:=-\lambda_1^-(-q,-\rho)>0,
 \qquad \theta_1^+:=\theta_1^-(-q,-\rho)>0.
\]
If \textup{(H4)} also holds, then for every isolated cluster
$\{k,\ldots,k+r-1\}$ of $(\mu_j^0)_{j\ge1}$,
\begin{equation}
 \lambda_{\eps,k+j-1}^{-}
 =\frac{\mu_k^0}{m_{\mathrm{eff}}}+O_k(\eps),
 \qquad 1\le j\le r,
 \notag
\end{equation}
and, with $\widetilde A^+=(\theta_1^+)^2A$,
$\widetilde m_{\mathrm{eff}}^+:=\int_{\Ys}(\theta_1^+)^2q\,dy
+\int_\Gam(\theta_1^+)^2\rho\,dS_y>0$ and corresponding homogenized
eigenvalues $\widetilde\mu_j^{0,+}$, for every isolated cluster
$\{k,\ldots,k+r-1\}$ of $(\widetilde\mu_j^{0,+})_{j\ge1}$,
\begin{equation}
 \lambda_{\eps,k+j-1}^{+}
 =\frac{\lambda_1^+}{\eps^2}
  +\frac{\widetilde\mu_k^{0,+}}{\widetilde m_{\mathrm{eff}}^+}
  +O_k(\eps),
 \qquad 1\le j\le r.
 \notag
\end{equation}
The projector, multiplicity and aligned-eigenspace conclusions transfer by
\eqref{sym:eq:ordered-branch-exchange}; the fast original modes carry the
factor $(\theta_1^+)^\eps$.
\end{corollary}


\section[Critical zero-effective-mass regime]
{Critical regime with zero effective mass}\label{sec:critical-regime}

\subsection{Critical compensation and fixed-index bounds}
\label{subsec:critical-compensation}

Throughout this section we assume $m_{\mathrm{eff}}=0$. Let $\ell$ be the
linear form on $H^1_\#(\Ys)$ defined by
\begin{equation}
 \ell(\psi):=\int_{\Ys}q\psi\,dy+\int_\Gam\rho\psi\,dS_y
 =\mathfrak b_Y(1,\psi),
 \qquad \psi\in H^1_\#(\Ys).
 \notag
\end{equation}
Then $\ell(1)=m_{\mathrm{eff}}$, so the assumption $m_{\mathrm{eff}}=0$ means
that $\ell$ vanishes on constants and is therefore well defined on
$H^1_\#(\Ys)/\R$. This compatibility condition permits us to define the critical cell corrector $\chi^0\in H^1_\#(\Ys)/\R$ by
\begin{equation}
 \int_{\Ys}A\nabla_y\chi^0\cdot\nabla_y\psi\,dy
 =\ell(\psi)
 \qquad\forall\psi\in H^1_\#(\Ys)/\R,
 \label{crit:eq:cell-chi0}
\end{equation}
with $\int_{\Ys}\chi^0\,dy=0$. Set
\begin{equation}
 F^0:=A\nabla_y\chi^0,
 \qquad
 \nu^2:=\int_{\Ys}A\nabla_y\chi^0\cdot\nabla_y\chi^0\,dy
 =\ell(\chi^0).
 \label{crit:eq:F0-nu}
\end{equation}

When $m_{\mathrm{eff}}=0$, hypothesis \textup{(H2)} reduces to
$(q,\rho)\not\equiv(0,0)$: a nonzero signed measure with vanishing total
mass has nontrivial positive and negative parts.

The critical corrector is nontrivial by the following observation.

\begin{lemma}
\label{crit:lem:nu-positive}
Assume \textup{(H1)}, $m_{\mathrm{eff}}=0$ and $(q,\rho)\not\equiv(0,0)$.
Then $\nu^2>0$.
\end{lemma}

\begin{proof}
If $\nu=0$, ellipticity gives $\nabla_y\chi^0=0$. Since
$m_{\mathrm{eff}}=\ell(1)=0$, equation \eqref{crit:eq:cell-chi0} then implies
\[
 \ell(\psi)=0\qquad\forall\psi\in H^1_\#(\Ys).
\]
Testing first with $C_c^\infty(\Ys)$ gives $q=0$ a.e. in $\Ys$. Hence
\[
 \int_\Gam\rho\,\gamma\psi\,dS_y=0
 \qquad\forall\psi\in H^1_\#(\Ys),
\]
where $\gamma$ is the trace operator. The trace space
$H^{1/2}(\Gam)$ is dense in $L^2(\Gam)$ and the trace map admits a bounded
right inverse; therefore $\rho=0$ a.e. on $\Gam$. Thus both spectral
densities vanish, a contradiction.
\end{proof}

Extend $F^0$ by zero into $T$ and denote the extension by $\mathring F^0$. With $q^\circ$ as in \eqref{eq:qcirc},
\begin{equation}
 -\Div_y\mathring F^0=q^\circ\,dy+\rho\,dS_\Gam
 \quad\text{in }\mathcal D'_\#(Y).
 \label{crit:eq:F0-div}
\end{equation}
By Lemma~\ref{lem:cell-regularity-H4}\textup{(iii)} with $g=q$ and $h=\rho$,
$\mathring F^0\in L^\infty_\#(Y)^N$. For $u,v\in V_\eps$,
$U=P_\eps u$, $V=P_\eps v$, rescaling the distributional identity gives
\begin{equation}
 b_\eps(u,v)
 =\eps\int_\Om(\mathring F^0)^\eps\cdot\nabla(UV)\,dx.
 \notag
\end{equation}
Define
\begin{equation}
 c_\eps(u,v):=\eps^{-1}b_\eps(u,v),\qquad u,v\in H^1(\OmE).
 \notag
\end{equation}
Then, for $u,v\in V_\eps$,
\begin{equation}
 c_\eps(u,v)
 =\int_\Om(\mathring F^0)^\eps\cdot\bigl(U\nabla V+V\nabla U\bigr)\,dx,
 \label{crit:eq:ce-F0}
\end{equation}
and therefore
\begin{equation}
 |c_\eps(u,v)|\le C\|u\|_{V_\eps}\|v\|_{V_\eps}.
 \label{crit:eq:ce-bound}
\end{equation}
For later boundary-layer calculations we also need one unconstrained
factor. If $r\in H^1(\OmE)$ and $v\in V_\eps$, then
\begin{equation}
 c_\eps(r,v)=\int_{\OmE}(F^0)^\eps\cdot\nabla(rv)\,dx.
 \label{crit:eq:ce-F0-boundary}
\end{equation}
For $v\in\mathcal D_\eps$ this follows from the rescaled cell equation
\eqref{crit:eq:cell-chi0}, which holds in the whole periodic solid
$\eps\mathcal W$ with the conormal condition on every periodic hole: the test
function $rv$ vanishes near $\Dep$, so only the holes of $\GamE$ contribute.
Since $F^0\in L^\infty$, density of $\mathcal D_\eps$ in $V_\eps$ gives the
formula for arbitrary $v$.

The critical spectral equation is
\begin{equation}
 a_\eps(u_\eps,v)=\Lambda_\eps c_\eps(u_\eps,v)
 \quad\forall v\in V_\eps,
 \qquad
 \Lambda_\eps:=\eps\lambda_\eps.
 \label{crit:eq:renormalized-spectral}
\end{equation}
Define the compact self-adjoint operator $\Keps^{\rm c}$ on $(V_\eps,a_\eps)$ by
\begin{equation}
 a_\eps(\Keps^{\rm c}f,v)=c_\eps(f,v),
 \qquad f,v\in V_\eps.
 \label{crit:eq:Teps}
\end{equation}
Its nonzero eigenvalues are $\tau_\eps=1/\Lambda_\eps$.

We first control the rescaled critical spectrum.

\begin{proposition}
\label{prop:prelim-critical-estimates}
Assume \textup{(H1)--(H2)} and $m_{\mathrm{eff}}=0$. For every sufficiently small $\eps>0$, the renormalized problem \eqref{crit:eq:renormalized-spectral} has two infinite sequences
\begin{equation}
 0<\Lambda_{\eps,1}^+\le\Lambda_{\eps,2}^+\le\cdots\to+\infty,
 \qquad
 0>\Lambda_{\eps,1}^-\ge\Lambda_{\eps,2}^-\ge\cdots\to-\infty,
 \label{eq:prelim-critical-sequences}
\end{equation}
where $\Lambda_{\eps,j}^\sigma=\eps\lambda_{\eps,j}^\sigma$. If, in
addition, the cell regularity conditions in \eqref{hyp:cell-regularity} hold,
then for every fixed $k\ge1$ there exist $c_{0,*}>0$, $C_{0,k}>0$, and
$\eps_k>0$ such that
\begin{equation}
 c_{0,*}\le |\Lambda_{\eps,j}^\sigma|\le C_{0,k},
 \qquad j=1,\ldots,k,\quad \sigma\in\{-1,+1\},
 \quad 0<\eps<\eps_k.
 \notag
\end{equation}
If the corresponding eigenfunctions are normalized by
\begin{equation}
 c_\eps(u_{\eps,j}^\sigma,u_{\eps,j}^\sigma)=\sigma,
 \label{eq:prelim-critical-normalization}
\end{equation}
then
\begin{equation}
 \|u_{\eps,j}^\sigma\|_{V_\eps}
 +\|P_\eps u_{\eps,j}^\sigma\|_{H_0^1(\Om)}
 +\eps^{1/2}\|u_{\eps,j}^\sigma\|_{L^2(\GamE)}
 \le C_{0,k}.
 \label{eq:prelim-critical-eigenfunction-bound}
\end{equation}
\end{proposition}

\begin{proof}
Because multiplication by the positive number $\eps^{-1}$ does not change signs, $c_\eps=\eps^{-1}b_\eps$ has the same positive and negative subspaces as $b_\eps$. Proposition~\ref{prop:prelim-two-sign-spectrum} therefore gives the two infinite renormalized sequences \eqref{eq:prelim-critical-sequences}. We now assume the cell regularity conditions in \eqref{hyp:cell-regularity}. Estimate \eqref{crit:eq:ce-bound} implies
\[
 \|\Keps^{\rm c}\|_{\mathcal L(V_\eps,a_\eps)}\le C,
\]
so every nonzero $\tau_\eps=1/\Lambda_\eps$ satisfies $|\tau_\eps|\le C$. Thus $|\Lambda_\eps|\ge c_{0,*}>0$ uniformly.

It remains to prove the fixed-index upper bound. Fix a $k$-dimensional space $L_k\subset C_c^\infty(\Om)$. The union of the supports of a fixed basis has positive distance from $\partial\Omega$; hence, for all sufficiently small $\eps$, every function below vanishes throughout the boundary cells and the corrected trial spaces are contained in $V_\eps$. For $\sigma\in\{-1,+1\}$ define the corrected trial map
\begin{equation}
 J_{\eps}^\sigma\phi
 :=\phi+\sigma\eps(\chi^0)^\eps\phi
 \quad\text{in }\OmE,
 \qquad \phi\in L_k.
 \notag
\end{equation}
For small $\eps$ this map is injective. We claim that, uniformly for $\phi,\psi$ in bounded subsets of $L_k$,
\begin{equation}
 c_\eps(J_\eps^\sigma\phi,J_\eps^\sigma\psi)
 \longrightarrow
 2\sigma\nu^2\int_\Om\phi\psi\,dx.
 \label{eq:prelim-critical-trial-mass-limit}
\end{equation}
Indeed,
\[
 c_\eps(J_\eps^\sigma\phi,J_\eps^\sigma\psi)
 =c_\eps(\phi,\psi)
 +\sigma b_\eps(\phi,(\chi^0)^\eps\psi)
 +\sigma b_\eps((\chi^0)^\eps\phi,\psi) +\eps b_\eps((\chi^0)^\eps\phi,(\chi^0)^\eps\psi).
\]
The first term tends to zero by \eqref{crit:eq:ce-F0} and periodic weak convergence of $(\mathring F^0)^\eps$, because the integral of a constant vector against $\nabla(\phi\psi)$ vanishes for compactly supported $\phi\psi$. The last term tends to zero by boundedness. The two middle terms converge, by the volume--surface Riemann-sum formula, to
\[
 \ell(\chi^0)\int_\Om\phi\psi\,dx
 =\nu^2\int_\Om\phi\psi\,dx,
\]
which proves \eqref{eq:prelim-critical-trial-mass-limit}. Hence $c_\eps$ is positive definite on $J_\eps^+L_k$ and negative definite on $J_\eps^-L_k$, with a uniform lower bound for the absolute value of the mass matrix.

Furthermore,
\[
 \nabla J_\eps^\sigma\phi
 =\nabla\phi+\sigma(\nabla_y\chi^0)^\eps\phi
 +\sigma\eps(\chi^0)^\eps\nabla\phi,
\]
so the cell regularity conditions in \eqref{hyp:cell-regularity} and finite
dimensionality of $L_k$ give constants $\gamma_k,M_k>0$, independent of
small $\eps$, such that
\begin{equation}
 \sigma c_\eps(J_\eps^\sigma\phi,J_\eps^\sigma\phi)
 \ge\gamma_k\|\phi\|_{L^2(\Om)}^2,
 \qquad
 a_\eps(J_\eps^\sigma\phi,J_\eps^\sigma\phi)
 \le M_k\|\phi\|_{L^2(\Om)}^2.
 \notag
\end{equation}
In particular, $J_\eps^\sigma$ is injective and
$J_\eps^\sigma L_k$ is a $k$-dimensional subspace on which
$\sigma c_\eps$ is positive definite. Lemma~\ref{lem:signed-courant-fischer},
applied on $(V_\eps,a_\eps)$ with $c=c_\eps$ and $L=J_\eps^\sigma L_k$, gives
$\sigma\Lambda_{\eps,k}^\sigma\le M_k/\gamma_k$.
Monotonicity of each signed sequence yields
$|\Lambda_{\eps,j}^\sigma|\le C_{0,k}$ for every $j\le k$.
Finally, under normalization \eqref{eq:prelim-critical-normalization},
testing the critical equation with the eigenfunction yields
\[
 a_\eps(u_{\eps,j}^\sigma,u_{\eps,j}^\sigma)
 =|\Lambda_{\eps,j}^\sigma|.
\]
Ellipticity, extension, and trace estimates give \eqref{eq:prelim-critical-eigenfunction-bound}.
\end{proof}

The classical correctors $\chi^j$, fluxes $B^j$, and homogenized tensor $\Ah$
remain those introduced in \eqref{eq:cell-chi}--\eqref{eq:Ah}.  With the
critical compensation and fixed-index bounds now in place, we turn to the
macroscopic limit.

\subsection[Qualitative homogenization at zero effective mass]
{Qualitative homogenization at zero effective mass}
\label{crit:sec:qualitative}

Let $(\lambda_\eps,u_\eps)_{0<\eps<1}$ be critical eigencouples with
\[
 \Lambda_\eps=\eps\lambda_\eps\longrightarrow\Lambda\ne0.
\]
Set $\sigma:=\operatorname{sign}\Lambda\in\{-1,+1\}$. Since
$\Lambda_\eps\to\Lambda\ne0$, for all sufficiently small
$0<\eps<1$ one has
$\operatorname{sign}\Lambda_\eps=\sigma$. We normalize by
\begin{equation}
 \ce(u_\eps,u_\eps)=\sigma.
 \notag
\end{equation}
The spectral identity then gives
\[
 \aep(u_\eps,u_\eps)=\Lambda_\eps \sigma=|\Lambda_\eps|,
\]
hence $(u_\eps)_{0<\eps<1}$ is uniformly bounded in
\(V_\eps\). After extraction,
\begin{align}
 \Peps u_\eps&\rightharpoonup u_0
 &&\text{weakly in }H_0^1(\Om),\qquad \Peps u_\eps\longrightarrow u_0
 \quad \text{strongly in }L^2(\Om),\notag\\
 \ind_{\OmE}\nabla P_\eps u_\eps
 &\overset{2}{\rightharpoonup}
 \ind_{Y^*}(\nabla_xu_0+\nabla_yu_1)
 &&\text{in }L^2(\Om\times Y)^N,\notag
\end{align}
where
\[
 u_1\in L^2\bigl(\Om;H^1_\#(\Ys)/\R\bigr).
\]
No surface two-scale convergence is needed here, since the compensated form
\eqref{crit:eq:ce-F0} contains no surface integral. We test with
$v_0(x)+\eps v_1(x,x/\eps)$, where $v_0\in C_c^\infty(\Omega)$ and
$v_1\in C_c^\infty(\Omega;C_\#^\infty(Y^*))$; for small $\eps$ these
functions vanish throughout every boundary cell and belong to $V_\eps$,
and the general formulation below follows by density.

Set
\[
 \mathscr H
 :=H_0^1(\Om)\times
 L^2\bigl(\Om;H^1_\#(\Ys)/\R\bigr).
\]
For \(U=(u_0,u_1)\) and \(V=(v_0,v_1)\) in $\mathscr H$, define
\begin{align}
 \mathscr A(U,V)
 &: =
 \int_{\Om\times\Ys}
 A(y)(\nabla_xu_0+\nabla_yu_1)\cdot
 (\nabla_xv_0+\nabla_yv_1)\,dx\,dy,\notag\\
 \mathscr M(U,V)
 &: =
 \int_{\Om\times\Ys}F^0(y)\cdot
 \Bigl[
 u_0(\nabla_xv_0+\nabla_yv_1)
 +v_0(\nabla_xu_0+\nabla_yu_1)
 \Bigr]\,dx\,dy.\notag
\end{align}
Identity \eqref{crit:eq:ce-F0}, strong $L^2$ convergence, and two-scale
convergence of the gradients give
\begin{equation}
 \mathscr A(U,V)=\Lambda\mathscr M(U,V)
 \qquad\forall V\in\mathscr H.
 \label{crit:eq:two-scale-pencil}
\end{equation}

Taking \(v_0=0\) in \eqref{crit:eq:two-scale-pencil}, and using
\eqref{eq:cell-chi} and \eqref{crit:eq:cell-chi0}, gives
\begin{equation}
 u_1(x,y)
 =\sum_{j=1}^N\chi^j(y)\partial_ju_0(x)
 +\Lambda u_0(x)\chi^0(y).
 \label{crit:eq:u1}
\end{equation}
The term \(\Lambda u_0\chi^0\) is specific to the critical regime and
changes sign between the two branches. Set
\begin{equation}
 D_j:=F^0\cdot(e_j+\nabla_y\chi^j),
 \qquad
 \mathbf c:=\int_{\Ys}F^0\,dy.
 \label{crit:eq:Dj}
\end{equation}
The symmetry of \(A\) and the cell equation for \(\chi^j\), tested with
\(\chi^0\), give
\begin{equation}
 \int_{\Ys}D_j\,dy
 =\int_{\Ys}A(e_j+\nabla_y\chi^j)\cdot\nabla_y\chi^0\,dy=0,
 \qquad\text{that is,}\qquad
 \int_{\Ys}F^0\cdot\nabla_y\chi^j\,dy=-\mathbf c_j.
 \label{crit:eq:Dmean0}
\end{equation}
Now take \(v_1=0\) in \eqref{crit:eq:two-scale-pencil} and insert
\eqref{crit:eq:u1}. Since
\(\int_{\Ys}A\nabla_y\chi^0\,dy=\mathbf c\) and
\(\int_{\Ys}F^0\cdot\nabla_y\chi^0\,dy=\nu^2\),
\begin{align*}
 \mathscr A(U,V)
 &=\int_\Om\Ah\nabla u_0\cdot\nabla v_0\,dx
 +\Lambda\int_\Om u_0\,\mathbf c\cdot\nabla v_0\,dx,\\
 \mathscr M(U,V)
 &=\int_\Om\mathbf c\cdot\bigl(u_0\nabla v_0+v_0\nabla u_0\bigr)\,dx
 +\int_\Om v_0\Bigl(-\mathbf c\cdot\nabla u_0+\Lambda\nu^2u_0\Bigr)\,dx =\int_\Om u_0\,\mathbf c\cdot\nabla v_0\,dx
 +\Lambda\nu^2\int_\Om u_0v_0\,dx.
\end{align*}
The terms linear in \(\Lambda\) cancel in
\(\mathscr A(U,V)-\Lambda\mathscr M(U,V)\), and
\eqref{crit:eq:two-scale-pencil} becomes the quadratic pencil
\begin{equation}
 \int_\Om\Ah\nabla u_0\cdot\nabla v\,dx
 =\Lambda^2\nu^2\int_\Om u_0v\,dx
 \qquad\forall v\in H_0^1(\Om).
 \label{crit:eq:quadratic-weak}
\end{equation}
Thus
\begin{equation}
 \left\{
 \begin{aligned}
 -\Div(\Ah\nabla u_0)&=\Lambda^2\nu^2u_0&&\text{in }\Om,\\
 u_0&=0&&\text{on }\partial\Om.
 \end{aligned}
 \right.
 \notag
\end{equation}

Let
\begin{equation}
 -\Div(\Ah\nabla\phi_k)=\mu_k^0\phi_k,
 \qquad
 0<\mu_1^0\le\mu_2^0\le\cdots\to+\infty.
 \notag
\end{equation}
The two critical roots are
\begin{equation}
 \Lambda_k^\pm
 =\pm\frac{\sqrt{\mu_k^0}}{\nu}.
 \notag
\end{equation}
The macroscopic profile is the same on both branches, but
\begin{equation}
 u_{1,k}^{\pm}
 =\sum_{j=1}^N\chi^j\partial_j\phi_k+\Lambda_k^\pm\phi_k\chi^0
 \notag
\end{equation}
depends on the sign.

Fix a macroscopic eigenvalue cluster
\begin{equation}
 \mu_k^0=\cdots=\mu_{k+r-1}^0=:\mu^0,
 \qquad
 E_k^0:=\ker(-\Div(\Ah\nabla\cdot)-\mu^0I),
 \quad\dim E_k^0=r,
 \label{crit:eq:E0}
\end{equation}
and, for \(\sigma\in\{-1,+1\}\), set
\begin{equation}
 \Lambda_k^\sigma:=\sigma\frac{\sqrt{\mu^0}}{\nu}.
 \notag
\end{equation}
Thus \(\Lambda_{k+j-1}^\sigma=\Lambda_k^\sigma\) for \(j=1,\ldots,r\).

The natural microscopic normalization is
\begin{equation}
 \ce(u_{\eps,k+i-1}^{\sigma},u_{\eps,k+j-1}^{\sigma})
 =\sigma\delta_{ij},
 \qquad 1\le i,j\le r.
 \label{crit:eq:critical-normalization-micro}
\end{equation}
It is equivalent to
\(b_\eps(u_{\eps,k+i-1}^{\sigma},u_{\eps,k+j-1}^{\sigma})
=\sigma\eps\delta_{ij}\).

The factor \(2\) is intrinsic to the symmetric form
\eqref{crit:eq:ce-F0}. For \(U=(u_0,u_1)\), with \(u_1\) given by
\eqref{crit:eq:u1}, \eqref{crit:eq:Dmean0} yields
\begin{equation}
 \mathscr M(U,U)
 =2\Lambda\nu^2\int_\Om|u_0|^2\,dx.
 \notag
\end{equation}
Hence, under \eqref{crit:eq:critical-normalization-micro},
\begin{equation}
 \int_\Om|u_0|^2\,dx
 =\frac{\sigma}{2\Lambda\nu^2}
 =\frac1{2|\Lambda|\nu^2}.
 \notag
\end{equation}
We therefore choose a basis \(\{\phi_1,\ldots,\phi_r\}\) of \(E_k^0\) such that
\begin{equation}
 2|\Lambda_k^\sigma|\nu^2
 \int_\Om\phi_i\phi_j\,dx=\delta_{ij}.
 \label{crit:eq:limit-critical-normalization}
\end{equation}
This normalization is the same for the two signs because
\(|\Lambda_k^+|=|\Lambda_k^-|\).

These identities identify the critical clusters with multiplicity.

\begin{theorem}
\label{crit:thm:qual-critical}
Assume \textup{(H1)--(H3)}, $m_{\mathrm{eff}}=0$, and
\eqref{crit:eq:E0}. Then, for \(\sigma\in\{-1,+1\}\),
\begin{equation}
 \eps\lambda_{\eps,k+j-1}^{\sigma}
 \longrightarrow
 \Lambda_k^\sigma,
 \qquad j=1,\ldots,r.
 \label{crit:eq:qual-eigenvalue}
\end{equation}
More precisely, let $\mathcal J_k^\sigma$ be any open interval with compact closure
in the half-line $\{\sigma t:t>0\}$ and such that
\begin{equation}
 \Lambda_k^\sigma\in \mathcal J_k^\sigma,
 \qquad
 \overline{\mathcal J_k^\sigma}\cap
 \{\Lambda_j^\varsigma:j\ge1,\ \varsigma\in\{-1,+1\}\}
 =\{\Lambda_k^\sigma\}.
 \label{crit:eq:qual-isolating-interval}
\end{equation}
For all sufficiently small $\eps$, exactly $r$ values in the rescaled
signed spectrum
\(
 \{\eps\lambda_{\eps,j}^\sigma:j\ge1\}
\), counted with multiplicity, belong to $\mathcal J_k^\sigma$. This is the critical
macroscopic cluster isolated in the spectrum of \(\Keps^{\rm c}\); modes
on a deeper scale are sent to \(0\) by the transformation
\(\tau=(\eps\lambda)^{-1}\).

Choose the microscopic eigenfunctions according to
\eqref{crit:eq:critical-normalization-micro} and the limit basis $\{\phi_1,\ldots,\phi_r\}$ according to \eqref{crit:eq:limit-critical-normalization}. There exists a family of orthogonal matrices
$(\mathsf O_{\eps,k}^{\sigma})_{0<\eps<1}\subset O(r)$, with
$\mathsf O_{\eps,k}^{\sigma}=(\mathsf o_{ij}^{\eps,\sigma})_{1\le i,j\le r}$, such that
\begin{equation}
 \max_{1\le i\le r}\left\|P_\eps u_{\eps,k+i-1}^{\sigma}-\sum_{j=1}^{r}\mathsf o_{ij}^{\eps,\sigma}\phi_j\right\|_{L^2(\Om)}\longrightarrow0.
 \label{crit:eq:qual-eigenvectors-L2}
\end{equation}
After extraction of a subfamily, without relabeling, such that
$\mathsf O_{\eps,k}^{\sigma}\to(\mathsf o_{ij}^{\sigma})$, set
$\phi_i^*=\sum_{j=1}^{r}\mathsf o_{ij}^{\sigma}\phi_j$. Then, with the
physical gradients extended by zero from the solid phase,
\begin{equation}
 \ind_{\OmE}\nabla P_\eps u_{\eps,k+i-1}^\sigma
 \overset{2}{\rightharpoonup}
 \ind_{Y^*}\left[
 \nabla_x\phi_i^*+\nabla_y\!\left(
 \sum_{j=1}^N\chi^j\partial_j\phi_i^*+\Lambda_k^\sigma\phi_i^*\chi^0\right)
 \right]
 \quad\text{in }L^2(\Om\times Y)^N,
 \qquad 1\le i\le r.
 \label{crit:eq:qual-eigenvectors-gradient}
\end{equation}
\end{theorem}

\begin{proof}
Set
\[
 \tau_k^\sigma=\frac1{\Lambda_k^\sigma},\qquad
 I_{0,k}^\sigma=(\tau_k^\sigma-d_0,\tau_k^\sigma+d_0),
\]
where, with $\Sigma_0:=\{\pm\nu/\sqrt{\mu_j^0}:j\ge1\}$,
$0<d_0<\dist\bigl(\tau_k^\sigma,(\{0\}\cup\Sigma_0)\setminus\{\tau_k^\sigma\}\bigr)$;
thus $\overline{I_{0,k}^\sigma}$ lies in the half-line containing
$\tau_k^\sigma$ and $\overline{I_{0,k}^\sigma}\cap\Sigma_0=\{\tau_k^\sigma\}$.

We first verify condition (i).  Let
$(p_{\eps,1},\ldots,p_{\eps,s})_{0<\eps<1}$ be a family of
$a_\eps$-orthonormal eigenvectors of $\Keps^{\rm c}$ with eigenvalues in
$I_{0,k}^\sigma$.  After a common extraction,
$\tau_{\eps,i}\to\tau_i^*\in\overline{I_{0,k}^\sigma}$, and the
compactness and two-scale identification
\eqref{crit:eq:two-scale-pencil}--\eqref{crit:eq:quadratic-weak}, applied
with $\Lambda=\Lambda_i^*:=1/\tau_i^*$, give
$P_\eps p_{\eps,i}\to p_{0,i}$ in $L^2(\Omega)$, where $p_{0,i}$
satisfies the limit pencil with $\Lambda_i^*$. Moreover,
\begin{equation}
 c_\eps(p_{\eps,i},p_{\eps,j})
 =\tau_{\eps,i}\delta_{ij},
 \label{crit:eq:cross-mass-micro}
\end{equation}
while the limit of the compensated form gives
$c_\eps(p_{\eps,i},p_{\eps,i})\to2\Lambda_i^*\nu^2\int_\Omega p_{0,i}^2\,dx$.
Hence $\int_\Omega p_{0,i}^2\,dx=(\tau_i^*)^2/(2\nu^2)>0$, so
$p_{0,i}\ne0$ and $(\Lambda_i^*)^2\nu^2$ is an eigenvalue of
$-\Div(\Ah\nabla\cdot)$; that is,
$\tau_i^*\in\Sigma_0\cap\overline{I_{0,k}^\sigma}=\{\tau_k^\sigma\}$
and $p_{0,i}\in E_k^0$. With $\Lambda_i^*=\Lambda_k^\sigma$ for all $i$,
the bilinear limit of the compensated form gives
\begin{equation}
 \int_\Omega p_{0,i}p_{0,j}\,dx
 =\frac{\delta_{ij}}{2(\Lambda_k^\sigma)^2\nu^2},
 \label{crit:eq:limit-independent-family}
\end{equation}
so the limits are linearly independent; in particular $s\le r$.
We use $X=L^2(\Omega)$ and $\mathcal I_\eps=P_\eps$ in
Theorem~\ref{thm:abstract-fixed-cluster}.

To verify condition (ii), choose $\phi_i^{(m)}\in C_c^\infty(\Omega)$ converging to the
normalized basis $\phi_i$ in $H_0^1(\Omega)$ and set
\[
 R_{\eps,m}^\sigma\phi_i
 =\left.\left(\phi_i^{(m)}
 +\eps\sum_{j=1}^N(\chi^j)^\eps\partial_j\phi_i^{(m)}
 +\eps\Lambda_k^\sigma(\chi^0)^\eps\phi_i^{(m)}\right)\right|_{\Omega_\eps}.
\]
For fixed $m$ and all sufficiently small $\eps$, this is the interior
reconstruction $R_{\eps,\rm int}^{\Lambda_k^\sigma}\phi_i^{(m)}$ of
Section~\ref{crit:sec:reconstructor}, and it belongs to $V_\eps$ because
$\phi_i^{(m)}$ has compact support. The interior computation
\eqref{crit:eq:core-expansion}--\eqref{crit:eq:CQ-Rint} uses only
\textup{(H1)--(H3)}, Lemma~\ref{lem:periodic-potentials} and the
compensated identity, and neither the boundary layer nor \textup{(H4)}.
Applied with $u=\phi_i^{(m)}$ and $\Lambda=\Lambda_k^\sigma$, and since
$(\Lambda_k^\sigma)^2\nu^2=\mu^0$, it yields a bound that is uniform over
$\|v\|_{a_\eps}=1$:
\[
 \limsup_{\eps\to0}\sup_{\|v\|_{a_\eps}=1}
 |a_\eps(R_{\eps,m}^\sigma\phi_i,v)
 -\Lambda_k^\sigma c_\eps(R_{\eps,m}^\sigma\phi_i,v)|
 \le C\| -\Div(\Ah\nabla\phi_i^{(m)})-\mu^0\phi_i^{(m)}\|_{H^{-1}}.
\]
The right-hand side tends to zero with $m$. Choose the same diagonal index $m=m(\eps)$ for all elements of the fixed basis, sufficiently slowly to satisfy Remark~\ref{rem:boundary-compatible-diagonal}, and extend the resulting
construction linearly to $E_k^0$: if
$\phi=\sum_{i=1}^r c_i\phi_i$, set
\[
 J_{\eps,k}^\sigma\phi
 :=\sum_{i=1}^r c_i R_{\eps,m(\eps)}^\sigma\phi_i.
\]
The finite dimensionality of $E_k^0$ makes the preceding estimates uniform
on its unit sphere, and hence
\begin{equation}
 a_\eps(J_{\eps,k}^\sigma\phi,J_{\eps,k}^\sigma\psi)
 \longrightarrow
 2(\Lambda_k^\sigma)^2\nu^2\int_\Omega\phi\psi\,dx,\qquad \|(\Keps^{\rm c}-\tau_k^\sigma I)J_{\eps,k}^\sigma\phi\|_{a_\eps}
 =o(1)\|J_{\eps,k}^\sigma\phi\|_{a_\eps}.
 \notag
\end{equation}
Thus condition (ii) and the Gram hypothesis hold uniformly on the finite-dimensional
limit eigenspace.  Theorem~\ref{thm:abstract-fixed-cluster} now gives exact
rank, convergence of the isolated reciprocal eigenvalues and orthogonal
alignment; the same argument as for condition (i) shows that an open interval whose
closure lies in the half-line containing $\tau_k^\sigma$ and does not meet
$\Sigma_0$ contains no eigenvalue of $\Keps^{\rm c}$ for small $\eps$, so
the final assertion of Theorem~\ref{thm:abstract-fixed-cluster} gives the
count in an arbitrary isolating interval $\mathcal J_k^\sigma$.  Returning through $\Lambda=1/\tau$, and using the two-scale
identification \eqref{crit:eq:u1}, yields
\eqref{crit:eq:qual-eigenvalue}--\eqref{crit:eq:qual-eigenvectors-gradient}.
\end{proof}

For quantitative rates, the qualitative interior expansion must be modified
so that the Dirichlet condition on $\Dep$ is satisfied.

\subsection{Quantitative critical branch}
\label{crit:sec:reconstructor}

We assume from now on
\eqref{hyp:cell-regularity}--\eqref{hyp:macro-regularity}.  The cut-boundary
resolvent estimate and augmented exact layer are the only additional
boundary ingredients needed for the quantitative critical analysis.

For \(u\in H^2(\Om)\cap H_0^1(\Om)\) and \(|\Lambda|\le R\), set in
\(\OmE\)
\begin{equation}
 R_{\eps,\rm int}^{\Lambda}u
 :=u+\eps\sum_{j=1}^N(\chi^j)^\eps\partial_ju
 +\eps\Lambda(\chi^0)^\eps u.
 \notag
\end{equation}
The oscillatory correctors are evaluated only at points of \(\OmE\),
equivalently at fast variables in \(\mathcal W\); no extension into \(T\) is
needed to define the physical problem.  For the sole purpose of
distributional integrations on the torus, the periodic Sobolev extension
theorem provides auxiliary $W^{1,\infty}_\#(Y)$ extensions of
$\chi^j$ and $\chi^0$.  Their values in \(T\) are immaterial and never enter
a physical integral, since the fields
\(B_0^j,D_{0,j},G_0,\mathring F^0\) vanish there. We have
\begin{equation}
 \nabla R_{\eps,\rm int}^{\Lambda}u
 =\sum_{j=1}^N(e_j+\nabla_y\chi^j)^\eps\partial_ju
 +\Lambda(\nabla_y\chi^0)^\eps u
 +\eps r_\eps^\Lambda(u),
 \label{crit:eq:gradRint}
\end{equation}
where
\begin{equation}
 \norm{r_\eps^\Lambda(u)}_{L^2(\OmE)}
 \le C_R\norm{u}_{H^2(\Om)}.
 \label{crit:eq:rRint}
\end{equation}
Moreover,
\begin{equation}
 \norm{R_{\eps,\rm int}^{\Lambda}u-u}_{L^2(\OmE)}
 \le C_R\eps\norm{u}_{H^2(\Om)}.
 \label{crit:eq:Rint-L2}
\end{equation}

Its full Dirichlet defect on $D_\eps$ is controlled by the augmented
exact layer.  Apply
Corollary~\ref{cor:common-boundary-layer-augmented} with $B=A$,
$\chi_A^0=\chi^0$, and write
$\beta_\eps^\Lambda[u]:=\beta_{\eps,A}^\Lambda[u]$.  Then
\begin{equation}
 \beta_\eps^\Lambda[u]+R_{\eps,\rm int}^{\Lambda}u\in V_\eps,\qquad \aep(\beta_\eps^\Lambda[u],v)=0
 \quad \forall v\in V_\eps,
 \label{crit:eq:beta-harmonic}
\end{equation}
so the complete interior expansion is cancelled on every Dirichlet
boundary piece, including the cut perforations.
The augmented layer estimates give
\begin{align}
 \norm{\nabla\beta_\eps^\Lambda[u]}_{L^2(\OmE)}
 &\le C_R\eps^{1/2}\norm{u}_{H^2(\Om)},\notag\\
 \norm{\beta_\eps^\Lambda[u]}_{L^2(\OmE)}
 +\norm{d\nabla\beta_\eps^\Lambda[u]}_{L^2(\OmE)}
 &\le C_R\eps\norm{u}_{H^2(\Om)},
 \label{crit:eq:BL}\\
 \eps^{1/2}\norm{\beta_\eps^\Lambda[u]}_{L^2(\GamE)}
 &\le C_R\eps\norm{u}_{H^2(\Om)}.
 \label{crit:eq:beta-trace}
\end{align}
The weighted estimate is the part needed specifically in the critical mass
argument; its proof is already contained in the common layer lemma.

Applying the exact one-factor identity \eqref{crit:eq:ce-F0-boundary} to $\beta_\eps^\Lambda[u]$ and $v$, and then using \eqref{crit:eq:BL} and Hardy's inequality, gives for \(v\in V_\eps\),
\begin{equation}
 \abs{\ce(\beta_\eps^\Lambda[u],v)}
 \le C\norm{\beta_\eps^\Lambda[u]}_2\norm{\nabla v}_2
 +C\norm{d\nabla\beta_\eps^\Lambda[u]}_2\norm{v/d}_2 \le C\eps\norm{u}_{H^2}\norm{v}_{V_\eps}.
 \label{crit:eq:beta-c-estimate}
\end{equation}
This estimate explains why the weighted bound in \eqref{crit:eq:BL}, rather
than the energy bound \(O(\eps^{1/2})\) alone, is necessary.

We define
\begin{equation}
 \mathcal R_\eps^\Lambda u
 :=R_{\eps,\rm int}^{\Lambda}u+\beta_\eps^\Lambda[u].
 \notag
\end{equation}
Then \(\mathcal R_\eps^\Lambda u\in V_\eps\) and, uniformly for
\(|\Lambda|\le R\),
\begin{align}
 \norm{\mathcal R_\eps^\Lambda u-u}_{L^2(\OmE)}
 \le C_R\eps\norm{u}_{H^2(\Om)},
 \label{crit:eq:Rcritical-L2}\\
 \norm{\mathcal R_\eps^\Lambda u}_{V_\eps}
 \le C_R\norm{u}_{H^2(\Om)}.
 \label{crit:eq:Rcritical-H1}
\end{align}
The reconstructor has three components: the elliptic corrector
\(\sum_{j=1}^N\chi^j\partial_ju\), the critical spectral corrector
\(\Lambda u\chi^0\), and the single exact Dirichlet layer
\(\beta_\eps^\Lambda[u]\).

We now establish the critical $O(\eps)$ residual.
\begin{proposition}
\label{crit:prop:CQstar}
Assume \textup{(H1)--(H4)} and
$m_{\mathrm{eff}}=0$. Let \(|\Lambda|\le R\),
\(u\in H^2(\Om)\cap H_0^1(\Om)\), \(v\in V_\eps\), and
\(V=\Peps v\). Then
\begin{equation}
\aep(\mathcal R_\eps^\Lambda u,v)
 -\Lambda\ce(\mathcal R_\eps^\Lambda u,v) =
 \int_\Om\Ah\nabla u\cdot\nabla V\,dx
 -\Lambda^2\nu^2\int_\Om uV\,dx
 +\mathcal E_\eps^\Lambda(u,v),
 \label{crit:eq:CQstar}
\end{equation}
where
\begin{equation}
 \abs{\mathcal E_\eps^\Lambda(u,v)}
 \le C_R\eps\norm{u}_{H^2(\Om)}\norm{v}_{V_\eps}.
 \label{crit:eq:CQstar-error}
\end{equation}
\end{proposition}

\begin{proof}
Recall $B^j=A(e_j+\nabla_y\chi^j)$ and $D_j$ from \eqref{crit:eq:Dj}, and
set $G:=F^0\cdot\nabla_y\chi^0$. Denote by $B_0^j$, $D_{0,j}$, $G_0$ the
extensions of these fields by zero into $T$. They are bounded by
Lemma~\ref{lem:cell-regularity-H4}, and
\begin{equation}
 \Div_yB_0^j=0\ \text{in }\mathcal D'_\#(Y),\qquad
 \int_YB_0^j\,dy=\Ah e_j,\qquad
 \int_YD_{0,j}\,dy=0,\qquad
 \int_YG_0\,dy=\nu^2,
 \notag
\end{equation}
by the cell equation for $\chi^j$ (whose conormal flux vanishes on $\Gam$),
\eqref{crit:eq:Dmean0} and \eqref{crit:eq:F0-nu}.

We first expand the interior reconstruction. By \eqref{crit:eq:gradRint},
\begin{equation}
 \aep(R_{\eps,\rm int}^{\Lambda}u,v)
 =\sum_{j=1}^N\int_{\OmE}(B^j)^\eps\partial_ju\cdot\nabla v\,dx
 +\Lambda\int_{\OmE}(F^0)^\eps u\cdot\nabla v\,dx
 +\eps\int_{\OmE}A^\eps r_\eps^\Lambda(u)\cdot\nabla v\,dx,
 \label{crit:eq:a-Rint}
\end{equation}
with
$r_\eps^\Lambda(u)=\sum_{j=1}^N(\chi^j)^\eps\nabla\partial_ju+\Lambda(\chi^0)^\eps\nabla u$.
Applying \eqref{crit:eq:ce-F0-boundary} to $R_{\eps,\rm int}^{\Lambda}u$ and
$v$, and using \eqref{crit:eq:gradRint} once more,
\begin{equation}
\begin{aligned}
 \ce(R_{\eps,\rm int}^{\Lambda}u,v)
 ={}&\int_{\OmE}(F^0)^\eps u\cdot\nabla v\,dx
 +\sum_{j=1}^N\int_{\OmE}D_j^\eps\partial_ju\,v\,dx
 +\Lambda\int_{\OmE}G^\eps uv\,dx\\
 &+\int_{\OmE}(F^0)^\eps\bigl(R_{\eps,\rm int}^\Lambda u-u\bigr)\cdot\nabla v\,dx
 +\eps\int_{\OmE}(F^0)^\eps\cdot r_\eps^\Lambda(u)\,v\,dx.
\end{aligned}
\label{crit:eq:c-Rint-start}
\end{equation}
The last integrals in \eqref{crit:eq:a-Rint} and
\eqref{crit:eq:c-Rint-start}, and the next-to-last one in
\eqref{crit:eq:c-Rint-start}, are bounded by
$C_R\eps\|u\|_{H^2}\|v\|_{V_\eps}$ by the $L^\infty$ bounds on the cell
fields, \eqref{crit:eq:rRint}, \eqref{crit:eq:Rint-L2} and Poincar\'e's
inequality. Subtracting $\Lambda$ times \eqref{crit:eq:c-Rint-start} from
\eqref{crit:eq:a-Rint}, the two terms containing
$\Lambda(F^0)^\eps u\cdot\nabla v$ cancel exactly. Since the three fields
vanish in $T$ and $V=v$ in $\OmE$, the remaining integrals may be written on
$\Om$ without any boundary-cell remainder:
\begin{equation}
\begin{aligned}
 &\aep(R_{\eps,\rm int}^{\Lambda}u,v)
 -\Lambda\ce(R_{\eps,\rm int}^{\Lambda}u,v)\\
 &\quad=
 \sum_{j=1}^N\int_\Om(B_0^j)^\eps\partial_ju\cdot\nabla V\,dx
 -\Lambda\sum_{j=1}^N\int_\Om D_{0,j}^\eps\partial_ju\,V\,dx
 -\Lambda^2\int_\Om G_0^\eps uV\,dx
 +O_R(\eps)\|u\|_{H^2}\|v\|_{V_\eps}.
\end{aligned}
\label{crit:eq:core-expansion}
\end{equation}

We next average the oscillating coefficients. Lemma~\ref{lem:periodic-potentials}\textup{(c)} with $G=B_0^j-\Ah e_j$,
$w=\partial_ju$ and $\psi=V$, and
Lemma~\ref{lem:periodic-potentials}\textup{(d)} with $h=D_{0,j}$,
$\varphi=\partial_ju\,V$, and with $h=G_0-\nu^2$, $\varphi=uV$ (both in
$W_0^{1,1}(\Om)$), applied for each $j$ and summed over $j=1,\ldots,N$, give
\begin{equation}
\begin{aligned}
 \Bigl|\sum_{j=1}^N\int_\Om(B_0^j)^\eps\partial_ju\cdot\nabla V\,dx
 -\int_\Om\Ah\nabla u\cdot\nabla V\,dx\Bigr|
 +\Bigl|\sum_{j=1}^N\int_\Om D_{0,j}^\eps\partial_ju\,V\,dx\Bigr|&\\
 +\Bigl|\int_\Om G_0^\eps uV\,dx-\nu^2\int_\Om uV\,dx\Bigr|
 &\le C\eps\|u\|_{H^2}\|V\|_{H^1}.
\end{aligned}
\label{crit:eq:averaging}
\end{equation}
Inserting \eqref{crit:eq:averaging} into \eqref{crit:eq:core-expansion} and
using \eqref{eq:extension} yields
\begin{equation}
 \aep(R_{\eps,\rm int}^{\Lambda}u,v)
 -\Lambda\ce(R_{\eps,\rm int}^{\Lambda}u,v)
 =\int_\Om\Ah\nabla u\cdot\nabla V\,dx
 -\Lambda^2\nu^2\int_\Om uV\,dx
 +O_R(\eps)\norm{u}_{H^2}\norm{v}_{V_\eps}.
\label{crit:eq:CQ-Rint}
\end{equation}

It remains to add the exact layer. By \eqref{crit:eq:beta-harmonic}, $\aep(\beta_\eps^\Lambda[u],v)=0$, and
\eqref{crit:eq:beta-c-estimate} gives
$|\ce(\beta_\eps^\Lambda[u],v)|\le C\eps\|u\|_{H^2}\|v\|_{V_\eps}$. Since
$\mathcal R_\eps^\Lambda u=R_{\eps,\rm int}^{\Lambda}u+\beta_\eps^\Lambda[u]$,
\eqref{crit:eq:CQ-Rint} becomes
\eqref{crit:eq:CQstar}--\eqref{crit:eq:CQstar-error}.
\end{proof}

If \(u\in E_k^0\) and
\(\Lambda=\Lambda_k^\sigma\), the macroscopic term vanishes by
\eqref{crit:eq:quadratic-weak}. Hence
\begin{equation}
 \abs{
 \aep(\mathcal R_\eps^{\Lambda_k^\sigma}u,v)
 -\Lambda_k^\sigma
 \ce(\mathcal R_\eps^{\Lambda_k^\sigma}u,v)}
 \le C_k\eps\norm{u}_{H^2}\norm{v}_{V_\eps}.
 \label{crit:eq:CQ}
\end{equation}
Since \(\ce=\be/\eps\), the original form also satisfies
\begin{equation}
 \abs{
 \eps\aep(\mathcal R_\eps^{\Lambda_k^\sigma}u,v)
 -\Lambda_k^\sigma
 \be(\mathcal R_\eps^{\Lambda_k^\sigma}u,v)}
 \le C_k\eps^2\norm{u}_{H^2}\norm{v}_{V_\eps}.
 \notag
\end{equation}

We also need the bilinear mass expansion.

\begin{proposition}
\label{crit:prop:critical-mass-reconstructors}
Assume \textup{(H1)--(H4)} and
$m_{\mathrm{eff}}=0$. For \(|\Lambda|\le R\) and
\(u,w\in H^2(\Om)\cap H_0^1(\Om)\),
\begin{equation}
 \abs{
 \ce(\mathcal R_\eps^\Lambda u,\mathcal R_\eps^\Lambda w)
 -2\Lambda\nu^2\int_\Om uw\,dx}
 \le C_R\eps\norm{u}_{H^2}\norm{w}_{H^2}.
 \label{crit:eq:critical-mass-reconstructors}
\end{equation}
\end{proposition}

\begin{proof}
We first analyze the interior correctors.  To avoid any normal-trace
issue where periodic interfaces meet $\partial\Omega$, we localize the
periodic compensation away from the exterior boundary.

For this localization, write $R_u=R_{\eps,\rm int}^\Lambda u$ and
$R_w=R_{\eps,\rm int}^\Lambda w$.  Fix $C_0>2\operatorname{diam}Y$ and
choose $\kappa_\eps\in W^{1,\infty}(\Omega)$ such that
\[
 \kappa_\eps=0\quad\text{if }d<C_0\eps,
 \qquad
 \kappa_\eps=1\quad\text{if }d>2C_0\eps,
 \qquad
 |\nabla\kappa_\eps|\le C\eps^{-1}.
\]
Use the fixed periodic $W^{1,\infty}$ extensions of $\chi^j$ and $\chi^0$
to define corresponding auxiliary reconstructions
$\widehat R_u,\widehat R_w$ on the whole of $\Omega$.  Applying
\eqref{crit:eq:F0-div} to
$\kappa_\eps\widehat R_u\widehat R_w$ produces no exterior boundary term.
Moreover, every periodic surface meeting the support of $\kappa_\eps$ belongs
to a complete cell contained in $\Omega$, hence is part of $\Gamma_\eps$.
Consequently,
\begin{equation}
 \ce(R_u,R_w)
 =\int_{\OmE}(F^0)^\eps\cdot\nabla(R_uR_w)\,dx
 +\mathcal B_{\eps}(u,w),
 \label{crit:eq:CM-interior-boundary}
\end{equation}
where the remainder is supported in $\{d<2C_0\eps\}$.

The boundary-strip estimates, the bounded cell fields, and
\eqref{crit:eq:gradRint} give
\[
 \|R_u\|_{L^2(\{d<2C_0\eps\}\cap\OmE)}
 \le C_R\eps^{3/2}\|u\|_{H^2},
 \qquad
 \|\nabla R_u\|_{L^2(\{d<2C_0\eps\}\cap\OmE)}
 \le C_R\eps^{1/2}\|u\|_{H^2},
\]
and the analogous estimates for $R_w$.  For the retained surface contribution, every complete cell that
intersects $\{d<2C_0\eps\}$ is contained in the enlarged strip
$\{d<(2C_0+C_Y)\eps\}$, where $C_Y$ depends only on the cell diameter.
The scaled cell trace estimate on these complete cells, together with the
same strip estimates on the enlarged strip, therefore gives the same
bounds.  The term involving $\nabla\kappa_\eps$ is bounded by
$C\eps^{-1}\|R_u\|_2\|R_w\|_2$.  Hence
\begin{equation}
 |\mathcal B_{\eps}(u,w)|
 \le C_R\eps^2\|u\|_{H^2(\Omega)}\|w\|_{H^2(\Omega)}.
 \label{crit:eq:CM-boundary-rate}
\end{equation}

In the volume integral in \eqref{crit:eq:CM-interior-boundary}, replace the
undifferentiated factors $R_u,R_w$ by $u,w$. By
\eqref{crit:eq:Rint-L2}, \eqref{crit:eq:gradRint}, and the cell bounds,
\begin{equation}
 \abs{\int_{\OmE}(F^0)^\eps\cdot
 \bigl[(R_u-u)\nabla R_w+(R_w-w)\nabla R_u\bigr] \,dx} \le
 C_R\eps\norm{u}_{H^2}\norm{w}_{H^2}.
 \notag
\end{equation}
Using \eqref{crit:eq:gradRint} once more, the remaining principal terms are
\begin{equation}
 \int_{\OmE}(F^0)^\eps\cdot
 \sum_{j=1}^N\bigl[u(e_j+\nabla_y\chi^j)^\eps\partial_jw+w(e_j+\nabla_y\chi^j)^\eps\partial_ju\bigr] \,dx
 +2\Lambda\int_{\OmE}G^\eps uw\,dx.
 \label{crit:eq:CM-principal-interior}
\end{equation}
The two terms containing $\eps r_\eps^\Lambda$ are bounded by
$C_R\eps\norm{u}_{H^2}\norm{w}_{H^2}$. After extension by zero, the first
group in \eqref{crit:eq:CM-principal-interior} is
\begin{equation}
 \sum_{j=1}^N\int_\Om D_{0,j}^\eps
 \bigl(u\partial_jw+w\partial_ju\bigr)\,dx
 =O(\eps)\norm{u}_{H^2}\norm{w}_{H^2}
 \notag
\end{equation}
by Lemma~\ref{lem:periodic-potentials}\textup{(d)}, since
$u\partial_jw+w\partial_ju\in W_0^{1,1}(\Om)$ for each $j$; the second equals
\begin{equation}
 2\Lambda\nu^2\int_\Om uw\,dx
 +O_R(\eps)\norm{u}_{H^2}\norm{w}_{H^2}
 \notag
\end{equation}
by the same lemma applied to $G_0-\nu^2$ and $uw$. Together with \eqref{crit:eq:CM-boundary-rate}, we have thus proved
\begin{equation}
 \abs{\ce(R_u,R_w)-2\Lambda\nu^2\int_\Om uw\,dx}
 \le C_R\eps\norm{u}_{H^2}\norm{w}_{H^2}.
 \label{crit:eq:CM-interior-final}
\end{equation}

We next treat the cross terms involving the boundary layer. Set $\mathcal R_u=\mathcal R_\eps^\Lambda u$ and
$\mathcal R_w=\mathcal R_\eps^\Lambda w$. Bilinearity gives
\begin{equation}
 \ce(\mathcal R_u,\mathcal R_w)
 =\ce(R_u,R_w)
 +\ce(\beta_\eps^\Lambda[u],\mathcal R_w)
 +\ce(\beta_\eps^\Lambda[w],\mathcal R_u)
 -\ce(\beta_\eps^\Lambda[u],\beta_\eps^\Lambda[w]).
 \label{crit:eq:CM-beta-decomposition}
\end{equation}
Both final reconstructors belong to $V_\eps$; hence
\eqref{crit:eq:beta-c-estimate} and \eqref{crit:eq:Rcritical-H1} yield
\begin{equation}
 \abs{\ce(\beta_\eps^\Lambda[u],\mathcal R_w)}
 +\abs{\ce(\beta_\eps^\Lambda[w],\mathcal R_u)}
 \le C_R\eps\norm{u}_{H^2}\norm{w}_{H^2}.
 \label{crit:eq:CM-beta-cross}
\end{equation}

It remains to control the boundary-layer--boundary-layer term, which we estimate from the definition $c_\eps=b_\eps/\eps$. Bounds
\eqref{crit:eq:BL} and \eqref{crit:eq:beta-trace} give
\begin{align}
 \abs{\ce(\beta_\eps^\Lambda[u],\beta_\eps^\Lambda[w])}
 &\le \frac{\norm{q}_\infty}{\eps}
 \norm{\beta_\eps^\Lambda[u]}_2\norm{\beta_\eps^\Lambda[w]}_2
 +\norm{\rho}_\infty\norm{\beta_\eps^\Lambda[u]}_{L^2(\GamE)}
 \norm{\beta_\eps^\Lambda[w]}_{L^2(\GamE)}\notag\\
 &\le C\eps\norm{u}_{H^2}\norm{w}_{H^2}.
 \label{crit:eq:CM-beta-beta}
\end{align}
Combining \eqref{crit:eq:CM-interior-final},
\eqref{crit:eq:CM-beta-decomposition}--\eqref{crit:eq:CM-beta-beta} proves
\eqref{crit:eq:critical-mass-reconstructors}.
\end{proof}

For the critical quasimodes, recall from \eqref{crit:eq:Teps} that
$\Keps^{\rm c}$ is the compact self-adjoint
operator associated with $c_\eps$ and that its nonzero eigenvalues are
$\tau_\eps=1/\Lambda_\eps$. The nonzero limit spectrum identified by the
quadratic pencil is
\begin{equation}
 \Sigma_0
 =\left\{
 \tau_k^\pm:=\frac1{\Lambda_k^\pm}
 =\pm\frac{\nu}{\sqrt{\mu_k^0}}:k\ge1
 \right\}.
 \notag
\end{equation}
We use the indexing convention
\begin{equation}
 \tau_{\eps,1}^+\ge\tau_{\eps,2}^+\ge\cdots>0,
 \qquad
 \tau_{\eps,1}^-\le\tau_{\eps,2}^-\le\cdots<0,
 \notag
\end{equation}
with multiplicities repeated. Hence
$\Lambda_{\eps,k}^\sigma=1/\tau_{\eps,k}^\sigma$ and
$\lambda_{\eps,k}^\sigma=\Lambda_{\eps,k}^\sigma/\eps$.

Fix an eigenvalue \(\mu^0\) of multiplicity \(r\), fix
\(\sigma\in\{-1,+1\}\), and set
\begin{equation}
 \Lambda_k^\sigma
 =\sigma\frac{\sqrt{\mu^0}}{\nu},
 \qquad
 \tau_k^\sigma=\frac1{\Lambda_k^\sigma}.
 \notag
\end{equation}
Let \(\{\phi_1,\ldots,\phi_r\}\) be the basis of \(E_k^0\) normalized by
\eqref{crit:eq:limit-critical-normalization}. Define
\begin{equation}
 \varphi_j:=\frac{\phi_j}{\sqrt{|\Lambda_k^\sigma|}},
 \qquad
 r_{\eps,j}^\sigma
 :=\mathcal R_\eps^{\Lambda_k^\sigma}\varphi_j,
 \notag
\end{equation}
and the matrix
\begin{equation}
 G_{\eps,k}^\sigma
 :=\bigl(\aep(r_{\eps,i}^\sigma,r_{\eps,j}^\sigma)\bigr)_{i,j=1}^r.
 \notag
\end{equation}

This expansion yields critical reconstruction nondegeneracy.

\begin{lemma}
\label{crit:lem:Gram-critical}
Assume \textup{(H1)--(H4)}, $m_{\mathrm{eff}}=0$, and let the eigenspace
$E_k^0$ and the sign $\sigma\in\{-1,+1\}$ be fixed as above. There exists
\(C_k>0\) such that
\begin{equation}
 \norm{G_{\eps,k}^\sigma-I_r}_{\R^{r\times r}}
 \le C_k\eps.
 \label{crit:eq:Gram-rate}
\end{equation}
In particular, for sufficiently small \(\eps\), the reconstruction operator is
injective on \(E_k^0\), and
\begin{equation}
 \mathcal Q_{\eps,k}^\sigma
 :=\mathcal R_\eps^{\Lambda_k^\sigma}E_k^0
 \quad\text{satisfies}\quad
 \dim\mathcal Q_{\eps,k}^\sigma=r.
 \notag
\end{equation}
\end{lemma}

\begin{proof}
Proposition~\ref{crit:prop:critical-mass-reconstructors} and
\eqref{crit:eq:limit-critical-normalization} give
\begin{equation}
 \ce(\mathcal R_\eps^{\Lambda_k^\sigma}\phi_i,
     \mathcal R_\eps^{\Lambda_k^\sigma}\phi_j)
 =\sigma\delta_{ij}+O_k(\eps).
 \notag
\end{equation}
Apply \eqref{crit:eq:CQ} with \(u=\phi_i\) and
\(v=\mathcal R_\eps^{\Lambda_k^\sigma}\phi_j\). Estimate
\eqref{crit:eq:Rcritical-H1} gives
\begin{equation}
 \aep(\mathcal R_\eps^{\Lambda_k^\sigma}\phi_i,
       \mathcal R_\eps^{\Lambda_k^\sigma}\phi_j)
 =\Lambda_k^\sigma
 \ce(\mathcal R_\eps^{\Lambda_k^\sigma}\phi_i,
     \mathcal R_\eps^{\Lambda_k^\sigma}\phi_j)
 +O_k(\eps) =|\Lambda_k^\sigma|\delta_{ij}+O_k(\eps).
 \notag
\end{equation}
After division by \(|\Lambda_k^\sigma|\), we obtain
\eqref{crit:eq:Gram-rate}. The eigenvalues of \(G_{\eps,k}^\sigma\) therefore
belong to \([1-C_k\eps,1+C_k\eps]\), which proves injectivity and the stated
dimension.
\end{proof}

For \(\phi\in E_k^0\), \eqref{crit:eq:CQ} and the definition of \(\Keps^{\rm c}\)
give
\begin{equation}
 \norm{
 \mathcal R_\eps^{\Lambda_k^\sigma}\phi
 -\Lambda_k^\sigma \Keps^{\rm c}
  \mathcal R_\eps^{\Lambda_k^\sigma}\phi}_{a_\eps}
 \le C_k\eps\norm{\phi}_{H^2}.
 \notag
\end{equation}
The finite dimensionality of \(E_k^0\) and Lemma~\ref{crit:lem:Gram-critical}
allow \(\norm{\phi}_{H^2}\) to be replaced by the energy norm of the
quasimode. Thus, for every
\(z\in\mathcal Q_{\eps,k}^\sigma\),
\begin{equation}
 \norm{(\Keps^{\rm c}-\tau_k^\sigma I)z}_{a_\eps}
 \le C_k\eps\norm{z}_{a_\eps}.
 \label{crit:eq:quasimode}
\end{equation}

For the critical fixed-cluster estimates, set
\[
 \Sigma_0=\{\pm\nu/\sqrt{\mu_j^0}:j\ge1\},\qquad
 d_k^\sigma=\dist\bigl(\tau_k^\sigma,
 (\{0\}\cup\Sigma_0)\setminus\{\tau_k^\sigma\}\bigr)>0,
\]
and
$I_k^\sigma=(\tau_k^\sigma-d_k^\sigma/3,
\tau_k^\sigma+d_k^\sigma/3)$
and define $\Pi_{\eps,k}^\sigma=\ind_{I_k^\sigma}(\Keps^{\rm c})$. Let
$Q_{\eps,k}^\sigma$ be the $a_\eps$-orthogonal projector onto
$\mathcal Q_{\eps,k}^\sigma$.

These estimates close the critical cluster argument.

\begin{theorem}
\label{crit:thm:quantitative-cluster}
Assume \textup{(H1)--(H4)}, $m_{\mathrm{eff}}=0$, let
$\mu_k^0=\cdots=\mu_{k+r-1}^0$ be an isolated eigenvalue cluster of multiplicity $r$, and fix $\sigma\in\{-1,+1\}$. Then,
for all small $\eps$,
\begin{equation}
 \operatorname{rank}\Pi_{\eps,k}^\sigma=r,\qquad
 \|\Pi_{\eps,k}^\sigma-Q_{\eps,k}^\sigma\|_{\mathcal L(V_\eps,a_\eps)}
 \le C_k\eps,\qquad \max_{1\le j\le r}
 |\eps\lambda_{\eps,k+j-1}^\sigma-\Lambda_k^\sigma|
 \le C_k\eps.
 \notag
\end{equation}
Choose the limit basis $\{\phi_1,\ldots,\phi_r\}$ according to
\eqref{crit:eq:limit-critical-normalization} and microscopic eigenfunctions
normalized by
\[
 c_\eps(u_{\eps,k+i-1}^\sigma,u_{\eps,k+j-1}^\sigma)
 =\sigma\delta_{ij}.
\]
There exists a family of orthogonal matrices
$(\mathsf O_{\eps,k}^\sigma)_{0<\eps<1}\subset O(r)$, with
$\mathsf O_{\eps,k}^\sigma=(\mathsf o_{ij}^{\eps,\sigma})$, such that
\begin{equation}
 \max_{1\le i\le r}\left\|u_{\eps,k+i-1}^\sigma-
 \sum_{j=1}^r\mathsf o_{ij}^{\eps,\sigma}
 \mathcal R_\eps^{\Lambda_k^\sigma}\phi_j\right\|_{a_\eps}
 \le C_k\eps,\qquad \max_{1\le i\le r}\left\|u_{\eps,k+i-1}^\sigma-
 \sum_{j=1}^r\mathsf o_{ij}^{\eps,\sigma}\phi_j\right\|_{L^2(\OmE)}
 \le C_k\eps.
 \notag
\end{equation}
The extensions satisfy
\begin{equation}
 \max_{1\le i\le r}\left\|P_\eps u_{\eps,k+i-1}^\sigma
 -\sum_{j=1}^r\mathsf o_{ij}^{\eps,\sigma}\phi_j\right\|_{L^2(\Om)}
 \le C_k\eps.
 \label{crit:eq:eigenfunction-extension-L2-rate}
\end{equation}
Equivalently,
$\lambda_{\eps,k+j-1}^\sigma=\Lambda_k^\sigma/\eps+O_k(1)$. The constant $C_k$ depends only on the fixed critical cluster data
and its isolation distance.
\end{theorem}

\begin{proof}
Lemma~\ref{crit:lem:Gram-critical} gives the Gram bound and injectivity, while
\eqref{crit:eq:quasimode} gives a residual $C_k\eps$ on the reconstructed
space. Non-pollution follows from
\eqref{crit:eq:cross-mass-micro}--\eqref{crit:eq:limit-independent-family}.
Theorem~\ref{thm:abstract-fixed-cluster}, applied to
$\Keps^{\rm c}$ and $J_\eps=\mathcal R_\eps^{\Lambda_k^\sigma}|_{E_k^0}$,
gives the rank, projector bound, reciprocal eigenvalue rate and orthogonal
alignment. Since $\Lambda_{\eps,j}^\sigma=1/\tau_{\eps,j}^\sigma$, the
second estimate follows by the map $t\mapsto t^{-1}$. Returning from the
energy-normalized frames and using Poincar\'e's inequality together with
\eqref{crit:eq:Rcritical-L2} gives the physical-domain $L^2$ estimate.
Applying $P_\eps$ to the energy error and using
Lemma~\ref{lem:common-reconstructor-extension} for the aligned critical
reconstructors gives \eqref{crit:eq:eigenfunction-extension-L2-rate}.
\end{proof}

\section{Concluding remarks}
\label{sec:conclusion}

We have shown that, for the bulk--surface problem \eqref{eq:intro-problem},
the sign of the combined mean $m_{\rm eff}$ selects the spectral scales:
$O(1)$ and $\eps^{-2}$ when $m_{\rm eff}\ne0$, and $\eps^{-1}$ for both
branches when $m_{\rm eff}=0$. Three questions remain open.

The first concerns higher-order terms: in the critical regime our estimates give
$\lambda_\eps=\Lambda_k^\sigma/\eps+O(1)$; identifying the $O(1)$ term
requires a second-order two-scale expansion together with the
corresponding boundary-layer contributions.

The second concerns weaker regularity. The quantitative results use H\"older regularity
of $A$ and $\rho$ and a $C^{2,\alpha}$ interface. For the qualitative
factorized branch, one may expect the ground-state transform to remain valid
under weaker coefficient assumptions at the quadratic-form level, provided
uniform boundary positivity and two-sided bounds for the cell ground state
can be established.

The third concerns uniform estimates. All rates are for fixed clusters. Estimates uniform
in the spectral index, or norm-resolvent estimates for the indefinite pencil,
would require a different approach.

\section*{Acknowledgments}

The author is grateful to Professor Andrea Braides for suggesting this problem in 2013, after the author's PhD defense.

\section*{Statements and declarations}

\noindent\textbf{Competing interests.} The author declares no competing
interests.

\noindent\textbf{Data availability.} No data were used for the research
described in this article.

\end{document}